\documentclass[10pt, a4paper]{amsart}
\usepackage{stix}
\usepackage{amsmath,amsthm}
\usepackage{tikz}
\usetikzlibrary{cd}
\usepackage{amscd}
\usepackage[all]{xy}
\definecolor{rosewood}{rgb}{0.4, 0.0, 0.04}
\definecolor{rufous}{rgb}{0.66, 0.11, 0.03}
\usepackage[colorlinks,final,hyperindex]{hyperref}
\hypersetup{citecolor=rufous, urlcolor=blue, linkcolor=rosewood}
\usepackage[font=small, labelfont=bf, margin=10pt]{caption}
\usepackage{csquotes}

\newtheorem{theorem}{Theorem}[subsection]
\newtheorem*{question*}{Question}
\newtheorem*{thm*}{Theorem}
\newtheorem{defi}{Definition}[subsection]
\newtheorem*{claim*}{Claim}
\newtheorem{lemma}{Lemma}[subsection]
\newtheorem{prop}{Proposition}[subsection]
\newtheorem{sled}{Corollary}[subsection]
\theoremstyle{definition}
\newtheorem{example}{Example}[subsection]
\newtheorem{remark}{Remark}[subsection]
\newtheorem*{notation*}{Notations}
\newtheorem*{example*}{Example}
\usepackage{color}

\newcommand{\Gr}{\operatorname{\Gamma}}
\newcommand{\parti}{\operatorname{\Pi_{\mathsf{gr}}}}
\newcommand{\Path}{\operatorname{\mathsf{P}}}
\newcommand{\K}{\operatorname{\mathsf{K}}}
\newcommand{\Cyc}{\operatorname{\mathsf{C}}}

\DeclareMathOperator{\Leav}{Leav}

\newcommand{\Aut}{\operatorname{\mathsf{Aut}}}
\DeclareMathOperator{\Id}{Id}
\DeclareMathOperator{\Hom}{Hom}
\newcommand{\GrCol}{\operatorname{\mathsf{GrCol}}}
\newcommand{\M}{\operatorname{\mathcal{M}}}
\newcommand{\Orb}{\operatorname{\mathcal{O}}}

\newcommand{\V}{\operatorname{\mathcal{V}}}
\newcommand{\B}{\operatorname{\mathcal{B}}}
\newcommand{\Q}{\operatorname{\mathcal{Q}}}

\newcommand{\E}{\operatorname{\mathcal{E}}}
\newcommand{\R}{\operatorname{\mathcal{R}}}
\newcommand{\Ho}{\operatorname{\mathcal{H}}}
\newcommand{\T}{\operatorname{\mathbb{T}}}
\newcommand{\Pop}{\operatorname{\mathcal{P}}}
\newcommand{\C}{\operatorname{\mathcal{C}}}
\newcommand{\N}{\operatorname{\mathcal{N}}}
\newcommand{\cokoszul}{\operatorname{\text{!`}}}
\newcommand{\Com}{\operatorname{\mathsf{gcCom}}}

\newcommand{\Susp}{\operatorname{\mathcal{S}}}
\newcommand{\Ass}{\operatorname{\mathsf{gcAss}}}
\newcommand{\Lie}{\operatorname{\mathsf{gcLie}}}
\newcommand{\Tree}{\operatorname{\mathsf{Tree}}}
\newcommand{\RST}{\operatorname{\mathsf{RST}}}
\newcommand{\Ham}{\operatorname{\mathsf{Ham}}}
\newcommand{\BiPlan}{\operatorname{\mathsf{BiPlan}}}
\newcommand{\Plan}{\operatorname{\mathsf{PlanEq}}}
\newcommand{\CycHam}{\operatorname{\mathsf{CycHam}}}
\newcommand{\CycEq}{\operatorname{\mathsf{CycEq}}}
\newcommand{\Perm}{\operatorname{\mathsf{Perm}}}
\newcommand{\CycPerm}{\operatorname{\mathsf{CycPerm}}}
\newcommand{\HP}{\operatorname{\mathrm{HP}}}
\newcommand{\HC}{\operatorname{\mathrm{HC}}}
\newcommand{\PE}{\operatorname{\mathrm{PE}}}
\newcommand{\CE}{\operatorname{\mathrm{CE}}}

\author{Denis Lyskov}
\address{Laboratory of Algebraic Geometry, National Research University Higher School of Economics, 6 Usacheva street, Moscow 119048, Russia}
\title{Operadic structure on Hamiltonian paths and cycles}
\email{ddl2001@yandex.ru}
\date{}

\begin{document}
\begin{abstract}
We study Hamiltonian paths and cycles in undirected graphs from an operadic viewpoint. We show that the graphical collection $\Ham$ encoding directed Hamiltonian paths in connected graphs admits an operad-like structure, called a contractad. Similarly, we construct the graphical collection of Hamiltonian cycles $\CycHam$ that forms a right module over the contractad $\Ham$. We use the machinery of contractad generating series for counting Hamiltonian paths/cycles for particular types of graphs.
\end{abstract}
\maketitle
\setcounter{tocdepth}{1}
\tableofcontents
\section{Introduction}
The purpose of this paper is to introduce graph theory researchers to the theory of contractads developed by the author in~\cite{lyskov2023contractads} and to show how to use this theory to produce various numerical applications using the example
of Hamiltonian paths and cycles.
 
A Hamiltonian path/cycle is a path/cycle in a graph that meets each vertex exactly once. The finding or counting of Hamiltonian paths and cycles is a long-standing problem in graph theory. In this paper, we introduce a new approach to studying Hamiltonian paths/cycles based on operadic methods. Informally, the idea is to consider all possible Hamiltonian paths in all possible graphs as one global object, called a contractad, with an additional algebraic structure induced from substitution of paths into paths. In a similar way, we construct an algebraic structure encoding all Hamiltonian cycles that forms a right module over the contractad of Hamiltonian paths.

Let us define the contractad of Hamiltonian paths, for details see Section~\ref{sec::contractads}. For a connected graph $\Gr$, let $\Ham(\Gr)$ be a vector space generated by directed Hamiltonian paths. For each connected graph $\Gr$ and collection of vertices $G$ inducing a connected subgraph $\Gr|_G$, we have a substitution map
\[
\circ^{\Gr}_G\colon \Ham(\Gr/G)\otimes \Ham(\Gr|_G) \rightarrow \Ham(\Gr),
\] 
where $\Gr/G$ is the graph obtained from $\Gr$ by contracting $G$ to a single vertex. For Hamiltonian paths $\Path=(v_1,\cdots,v_{k-1},\{G\},v_{k+1},\cdots,v_n)$ and $\Path'=(w_1,w_2,\cdots, w_m)$ in the contracted $\Gr/G$ and induced $\Gr|_G$ graphs, we put $\Path\circ^{\Gr}_G \Path'=(v_1,\cdots,v_{k-1},w_1,w_2,\cdots, w_m,v_{k+1},\cdots,v_n)$, if $v_{k-1}$ is adjacent to $w_1$ and $w_m$ to $v_{k+1}$, and zero otherwise. The resulting graphical collection of vector spaces $\Ham:=\{\Ham(\Gr)\}_{\Gr\in \mathsf{CGr}}$ with substitution maps forms an algebraic structure, called a \textit{contractad}, introduced by the author in~\cite{lyskov2023contractads}. This structure is a natural generalization of the notion of an operad, whose set of operations is indexed by connected graphs and whose composition rules are numbered by contractions of connected subgraphs. 

For a connected graph $\Gr$, we let $\CycHam(\Gr)$ be the vector spaces generated by directed Hamiltonian cycles. Similarly, the operation of substitution of paths into cycles produces the substitution maps
\[
\circ^{\Gr}_G\colon \CycHam(\Gr/G)\otimes \Ham(\Gr|_G) \rightarrow \CycHam(\Gr).
\] In terms of contractads, these operations endow the graphical collection $\CycHam$ of Hamiltonian cycles with a structure of the right $\Ham$-module, for details see Section~\ref{sec::modules}.

While studying these two objects, we give an introduction to the theory of contractads and right modules over them. Also, we briefly discuss other contractad examples arising in graph theory, such as the contractad of acyclic graph directions $\Ass$, the contractad of rooted spanning trees $\RST$, and the contractad of planar equivalent binary trees $\Plan$.

The central results of this paper are concerning Koszul properties for $\Ham$ and $\CycHam$. Recall that, for a free symmetric algebra $\mathbb{S}(V)$, there is a so-called Koszul resolution $\bigwedge(V^*)\otimes \mathbb{S}(V)\overset{\simeq}{\to}\mathbb{k}$, where $\bigwedge(V^*)$ is an exterior algebra on the dual space. The construction of Koszul complexes was generalised to the case of quadratic associative algebras $A^{\cokoszul}\otimes^{\kappa} A\to \mathbb{k}_{A}$, see~\cite{polishchuk2005quadratic} for details, and then to the case of quadratic operads $\Pop^{\cokoszul}\circ^{\kappa}\Pop\to \mathbb{1}_{\Pop}$, see~\cite{loday2012algebraic} for details. Mysteriously, for most examples of quadratic algebras/operads, these complexes define minimal resolutions of trivial modules. We generalise the construction of Koszul complexes to the case of quadratic right modules over contractads $\M^{\cokoszul}\circ^{\kappa}\Pop\to \M$, for details see Section~\ref{sec::koszul}. We say that a module $\M$ is Koszul if the corresponding Koszul complex is a minimal resolution for $\M$. In particular, we say that a contractad $\Pop$ is Koszul if its trivial module $\mathbb{1}_{\Pop}$ is. The main result of this paper is

\begin{thm*}
The contractad of Hamiltonian path $\Ham$ and its right module of Hamiltonian cycles $\CycHam$ are Koszul.
\end{thm*}

At the end of the paper, we use the machinery of generating series for contractads developed in~\cite{khoroshkin2024hilbert} to obtain several applications for counting Hamiltonian paths and cycles for particular types of graphs, for details see Section~\ref{sec::applications}. The main idea is to deduce functional equations on the dimension of components of Koszul contractads by examining Euler characteristics of the corresponding Koszul complexes as it was done for algebras or operads. We list some of these applications in the introduction. For a graph $\Gr$, let $\HP(\Gr)$ and $\HC(\Gr)$ be the number of directed Hamiltonian paths and cycles respectively.
\begin{thm*}[Theorem~\ref{thm::reccurences_paths_and_cycles}]
\begin{itemize}
\item[(i)] For a connected graph $\Gr$, we have
\begin{equation}\label{eq::intro_hp_multi}
\sum_{I\vdash \Gr}(-1)^{|I|-1}|I|!\prod_{G\in I}\HP(\Gr|_G)=(-1)^{|V_{\Gr}|-1}\HP(\overline{\Gr}),
\end{equation} where the sum is taken over all vertex-partitions $I$ of $\Gr$ into connected subgraphs and $\overline{\Gr}$ is a complement graph.
\item[(ii)] For a connected graph $\Gr$ with at least one edge, we have
\begin{equation}\label{eq::intro_hc_multi}
\sum_{I\vdash \Gr} (-1)^{|I|-1}(|I|-1)!\prod_{G\in I} \HP(\Gr|_G)=\HC(\Gr)+(-1)^{|V_{\Gr}|-1}\HC(\overline{\Gr}),
\end{equation}
\end{itemize}
\end{thm*}

Since the class of entire graphs is too "large", it is natural to consider special families of graphs and try to find generating functions for them such that identities~\eqref{eq::intro_hp_multi} and \eqref{eq::intro_hc_multi} transform into functional equations. In this paper, we pay attention to the case of complete multipartite graphs. For a partition $\lambda=(\lambda_1\geq \lambda_2\geq\cdots\geq\lambda_k)$, the complete multipartite graph $\K_{\lambda}$ is the graph consisting of blocks of vertices of sizes $\lambda_1,\lambda_2,\cdots,\lambda_k$, such that two vertices are adjacent if and only if they belong to different blocks. In~\cite{khoroshkin2024hilbert}, the authors introduced a generating function $F_{\mathsf{Y}}(f)$ in the ring of symmetric functions $\Lambda_{\mathbb{Q}}=\underset{\to}{\lim}\mathbb{Q}[x_1,\cdots,x_n]^{\Sigma_n}$ that packages values of a graphic function $f$ on complete multipartite graphs. For example, using these generating functions, the authors obtained the following generating function for chromatic polynomials

\begin{thm*}\cite{khoroshkin2024hilbert}
The generating function of chromatic polynomials for complete multipartite graphs is given  by the formula
\begin{equation}
\sum_{\lambda} \chi_{\K_{\lambda}}(q)\frac{m_{\lambda}}{\lambda!}=(1+\sum_{n\geq 1} \frac{p_n}{n!})^q,
\end{equation} where $m_{\lambda}=\mathsf{Sym}(x^{\lambda})$ are monomial symmetric functions, $\lambda!:=\lambda_1!\cdot\lambda_2!\cdots\lambda_k!$, and $p_n=\sum^{\infty}_{i=1}x_i^n$ is the $n$-th power symmetric function.
\end{thm*} 
In~\cite{klarner1969number}, Klarner computed the generating function in the ring of symmetric functions that counts the numbers of Hamiltonian paths in complete multipartite graphs. We reprove and generalise its result using methods of generating functions.
\begin{thm*}[Theorem~\ref{thm::hp_hc_young_generating}]
\begin{itemize}
    \item[(i)]
The generating function of directed Hamiltonian paths in complete multipartite graphs is given by the formula\footnote{We use convention $\HP(\K_{(0)})=1$ for the empty graph $\K_{(0)}=\varnothing$}
\begin{equation*}
  \sum_{|\lambda|\geq 0} \HP(\K_{\lambda})\frac{m_{\lambda}}{\lambda!}=\frac{1}{1-\sum_{n\geq 1} (-1)^{n-1}p_n}.  
\end{equation*} 
\item[(ii)] The generating function of directed Hamiltonian cycles in complete multipartite graphs is given by the formula\footnote{We use convention $\HC(\K_{(1^2)})=1$ for connected graph on two vertices}
\begin{equation*}
\sum_{l(\lambda)\geq 2} \HC(\K_{\lambda})\frac{m_{\lambda}}{\lambda!}=-\log(1-\sum_{n\geq 1}(-1)^{n-1}p_n)-\sum_{n\geq 1}(-1)^{n-1}\frac{p_n}{n}.  
\end{equation*}
\end{itemize}
\end{thm*}
\noindent We expect that it is possible to find similar generating functions for various series of graphs that are closed under contractions and induced subgraphs.  

In~\cite{dotsenko2010anick}, the authors used operadic methods in the study of consecutive patterns avoiding permutations. In Section~\ref{sec::separable_permutations}, we find new applications of operadic-like structures in this area. Firstly, we reprove and generalise the generating series resolving Hertzsprung's problem~\cite[p.~737]{flajolet2009analytic}. Secondly, we explore the relation between separable permutations and the Koszul dual contractad $\Plan$ to the Hamiltonian contractad. Specifically, we reprove that separable permutations are those that have separating trees~\cite{bose1998pattern} and their numbers are given by little Shr\"oder numbers~\cite{shapiro1991bootstrap}. 

\subsection*{Organisation}  In Section~\ref{sec::contractads}, we recall the basic notions of contractads, give examples of contractads arising in graph theory and describe their presentations.  In Section~\ref{sec::modules}, we define right modules over contractads and consider several examples of modules. In Section~\ref{sec::koszul}, we develop a Koszul theory for right modules over contractads and verify the Koszul property for $\Ham$ and $\CycHam$. In Section~\ref{sec::applications}, we deduce several applications for counting Hamiltonian paths/cycles and consecutive patterns avoiding permutations. 

\subsection*{Acknowledgements} I am grateful to my advisor Anton Khoroshkin for his guidance and inspiration in the process of writing. I would like to thank Vladimir Dotsenko and Dmitri Piontkovski for useful discussions at various stages of the preparation of this paper and comments on its draft. 
\subsection*{Funding} This paper was supported by the grant RSF 24-21-00341 of Russian Science Foundation.

\section{Contractads and Hamiltonian paths}\label{sec::contractads}

In this section, we give a quick introduction to the theory of contractads introduced by the author in~\cite{lyskov2023contractads}. Informally a contractad is a graphical generalisation of operads, whose set of operations are indexed by connected graphs instead of finite sets and composition rules are numbered by contractions of connected subgraphs, that explains the nature of terminology.

Also, we describe several examples of contractads arising in graph theory. In particular, we describe the contractad $\Ham$ encoding directed Hamiltonian paths in graphs and show that this contractad admits a nice quadratic presentation.

\subsection{Graphs and contractads} In this subsection, we briefly recall several definitions of contractads. For more details, see~\cite[Sec 1]{lyskov2023contractads}.

Let us call a \textit{graph} to be a finite simple undirected  $\Gr=(V_{\Gr},E_{\Gr})$, where $V_{\Gr}$ represents a set of vertices, and $E_{\Gr}$ represents a set of edges. We shall use the following notations for particular types of graphs:
\begin{itemize}
\label{typesofgraphs}
\item the path graph $\Path_n$ on the vertex set $[n]=\{1,\cdots, n\}$ with edges $\{(i,i+1)| 1\leq i \leq n-1 \}$,
\item the cycle graph $\Cyc_n$ on the vertex set $[n]$ with edges $\{(i,i+1)| 1 \leq i \leq n-1\}\cup \{(n,1)\}$,
     \item the complete graph $\K_n$ on the vertex set $\{1,\cdots, n\}$ and the edges $\{(i,j)|i\neq j\}$,
    
     \item For a partition $\lambda=(\lambda_1\geq \lambda_2\geq\cdots\geq\lambda_k)$, we consider the complete multipartite graph $\K_{\lambda}$.  This graph consists of blocks of vertices of sizes $\lambda_1,\lambda_2,\cdots,\lambda_k$, such that two vertices are adjacent if and only if they belong to different blocks.
     \begin{figure}[ht]
\[
\vcenter{\hbox{\begin{tikzpicture}[scale=0.6]
    \fill (0,0) circle (2pt);
    \fill (0,1.5) circle (2pt);
    \fill (1.5,0) circle (2pt);
    \fill (1.5,1.5) circle (2pt);
    \node at (0.75,-0.6) {$\K_{(4)}$};
    \end{tikzpicture}}}
    \quad\quad
\vcenter{\hbox{\begin{tikzpicture}[scale=0.6]
    \fill (0,0) circle (2pt);
    \fill (-1,1.5) circle (2pt);
    \fill (0,1.5) circle (2pt);
    \fill (1,1.5) circle (2pt);
    \draw (0,0)--(-1,1.5);
    \draw (0,0)--(0,1.5);
    \draw (0,0)--(1,1.5);
    \node at (0,-0.6) {$\K_{(3,1)}$};
    \end{tikzpicture}}}
    \quad
\vcenter{\hbox{\begin{tikzpicture}[scale=0.6]
    \fill (0,0) circle (2pt);
    \fill (0,1.5) circle (2pt);
    \fill (1.5,0) circle (2pt);
    \fill (1.5,1.5) circle (2pt);
    \draw (0,0)--(1.5,0)--(1.5,1.5)--(0,1.5)-- cycle;
    \node at (0.75,-0.6) {$\K_{(2^2)}\cong \Cyc_4$};
    \end{tikzpicture}}}
    \quad
\vcenter{\hbox{\begin{tikzpicture}[scale=0.6]
    \fill (0,0) circle (2pt);
    \fill (0,1.5) circle (2pt);
    \fill (1.5,0) circle (2pt);
    \fill (1.5,1.5) circle (2pt);
    \draw (0,0)--(1.5,0)--(1.5,1.5)--(0,1.5)-- cycle;
    \draw (0,0)--(1.5,1.5);
    \node at (0.75,-0.6) {$\K_{(2,1^2)}$};
    \end{tikzpicture}}}
    \quad
\vcenter{\hbox{\begin{tikzpicture}[scale=0.6]
    \fill (0,0) circle (2pt);
    \fill (0,1.5) circle (2pt);
    \fill (1.5,0) circle (2pt);
    \fill (1.5,1.5) circle (2pt);
    \draw (0,0)--(1.5,0)--(1.5,1.5)--(0,1.5)-- cycle;
    \draw (0,0)--(1.5,1.5);
    \draw (1.5,0)--(0,1.5);
    \node at (0.75,-0.6) {$\K_{(1^4)}\cong \K_4$};
    \end{tikzpicture}}}
\]
\caption{List of complete multipartite graphs on 4 vertices.}
\end{figure}
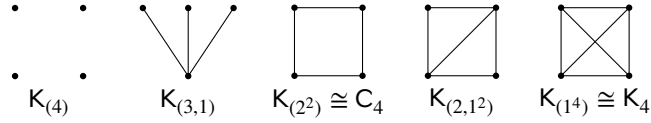
\end{itemize}
Consider the \textit{groupoid of connected graphs} $\mathsf{CGr}$ whose objects are non-empty connected simple graphs and whose morphisms are isomorphisms of graphs. Let $\mathsf{Vect}$ be the category of vector spaces over a ground field $\mathsf{k}$.
\begin{defi}
A graphical collection is a contravariant functor
\[\Orb\colon \mathsf{CGr}^{\mathrm{op}}\rightarrow \mathsf{Vect}.\] All graphical collections with natural transformations form a category $\GrCol$.
\end{defi}
Explicitly, a graphical collection $\Orb$ is a collection of vector spaces $\{\Orb(\Gr)\}$ indexed by connected graphs, and, for each isomorphism of graphs $\phi\colon \Gr\overset{\cong}{\to} \Gr'$, there is an isomorphism of vector spaces $\Orb(\phi)\colon \Orb(\Gr')\to \Orb(\Gr)$. In particular, for each component, there  is an action of graph-automorphisms $\Aut(\Gr)$ on $\Orb(\Gr)$. A morphism of graphical collections $f\colon \Orb\to \Q$ is a collection of linear maps $f(\Gr)\colon \Orb(\Gr)\to \Q(\Gr)$ compatible with isomorphisms of graphs.

For a graph $\Gr$ and a subset of vertices $S$,  the \textit{induced subgraph} is the graph $\Gr|_S$ with vertex set $S$ and edges coming from the original graph.
\begin{defi} 
\begin{itemize}
\item A tube of a graph $\Gr$ is a non-empty subset $G$ of vertices such that the induced subgraph $\Gr|_G$ is connected.  If the tube consists of one vertex, we call it trivial.
\item A \textit{partition of a graph} $\Gr$ is a partition of the vertex set whose blocks are tubes. We denote by $\parti(\Gr)$ the set of partitions of the graph $\Gr$. We use the notation $I\vdash \Gr$ for graph partitions.
\end{itemize}
\end{defi}
The operation dual to induced graphs is the contraction by partitions. 
\begin{defi}
\begin{itemize}
    \item For a partition $I$ of graph $\Gr$, the contracted graph, denoted $\Gr/I$, is the graph obtained from $\Gr$  by contracting each block of $I$ to a single vertex. Explicitly, vertices of $\Gr/I$ are partition blocks and edges are pairs $\{G\},\{H\}$ of blocks such that their union $G\cup H$ is a tube of $\Gr$. 
    \item Given a tube $G\subset V_{\Gr}$, we denote by $\Gr/G$ the contracted graph obtained from $\Gr$ by collapsing $G$ to a single vertex. Explicitly, $\Gr/G$ is the contracted graph associated with a partition $I=\{G\}\cup \{\{v\}|v\not\in G\}$.
    \end{itemize}
\end{defi}

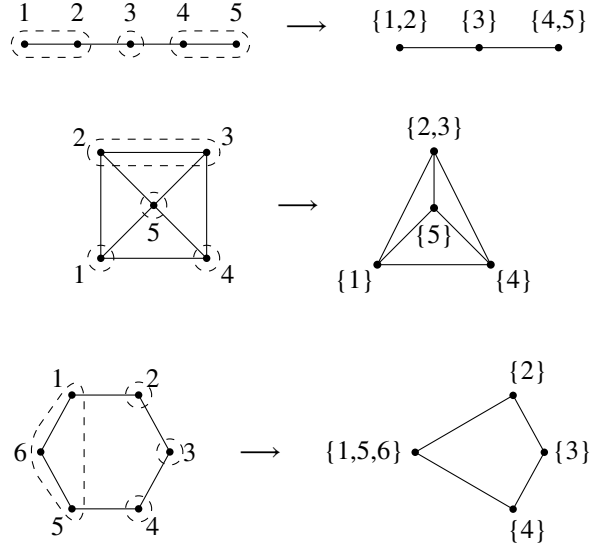
\begin{figure}[ht]
  \centering
  \begin{gather*}
  \vcenter{\hbox{\begin{tikzpicture}[scale=0.7]
    \fill (0,0) circle (2pt);
    \node at (0,0.6) {1};
    \fill (1,0) circle (2pt);
    \node at (1,0.6) {2};
    \fill (2,0) circle (2pt);
    \node at (2,0.6) {3};
    \fill (3,0) circle (2pt);
    \node at (3,0.6) {4};
    \fill (4,0) circle (2pt);
    \node at (4,0.6) {5};
    \draw (0,0)--(1,0)--(2,0)--(3,0)--(4,0);
    \draw[dashed, rounded corners=5pt] (-0.25,-0.25) rectangle ++(1.5,0.5);
    \draw[dashed, rounded corners=5pt] (2.75,-0.25) rectangle ++(1.5,0.5);
    \draw[dashed] (2,0) circle (7pt);
    \end{tikzpicture}}}
    \quad
    \longrightarrow
    \quad
  \vcenter{\hbox{\begin{tikzpicture}[scale=0.7]
    \fill (0,0) circle (2pt);
    \node at (0,0.5) {\{1,2\}};
    \fill (1.5,0) circle (2pt);
    \node at (1.5,0.5) {\{3\}};
    \fill (3,0) circle (2pt);
    \node at (3,0.5) {\{4,5\}};
    \draw (0,0)--(1.5,0)--(3,0);    
    \end{tikzpicture}}}
    \\
    \\
  \vcenter{\hbox{\begin{tikzpicture}[scale=0.7]
    \fill (0,0) circle (2pt);
    \node at (-0.4,-0.3) {1};
    \fill (2,0) circle (2pt);
    \node at (2.4,-0.3) {4};
    \fill (0,2) circle (2pt);
    \node at (-0.4,2.3) {2};
    \fill (1,1) circle (2pt);
    \node at (1,0.5) {5};
    \fill (2,2) circle (2pt);
    \node at (2.4,2.3) {3};
    \draw (0,0)--(2,0)--(2,2)--(0,2)--cycle;
    \draw (0,0)--(1,1)--(2,2);
    \draw (2,0)--(1,1)--(0,2);
    \draw[dashed] (0,0) circle (7pt);
    \draw[dashed] (1,1) circle (7pt);
    \draw[dashed] (2,0) circle (7pt);
    \draw[dashed, rounded corners=5pt] (-0.25,1.75) rectangle ++(2.5,0.5);
    \end{tikzpicture}}}
    \quad
    \longrightarrow
    \vcenter{\hbox{\begin{tikzpicture}[scale=0.75]
    \fill (0,0) circle (2pt);
    \node at (-0.4,-0.3) {\{1\}};
    \fill (1,1) circle (2pt);
    \node at (1,0.5) {\{5\}};
    \fill (1,2) circle (2pt);
    \node at (1,2.4) {\{2,3\}};
    \fill (2,0) circle (2pt);
    \node at (2.4,-0.3) {\{4\}};
    \draw (0,0)--(2,0)--(1,2)-- cycle;
    \draw (0,0)--(1,1)--(1,2);
    \draw (2,0)--(1,1);
    \end{tikzpicture}}}
    \\
    \\
    \vcenter{\hbox{\begin{tikzpicture}[scale=0.7]
    \fill (-0.63,1.075)  circle (2pt);
    \node at (-0.9,1.4) {1};
    \fill (0.63,1.075)  circle (2pt);
    \draw[dashed] (0.63,1.075) circle  (7pt);
    \node at (0.9,1.4) {2};
    \fill (-1.22,0) circle (2pt);
    \node at (-1.6,0) {6};
    \fill (1.22,0) circle (2pt);
    \node at (1.6,0) {3};
    \draw[dashed] (1.22,0) circle  (7pt);
    \fill (-0.63,-1.075)  circle (2pt);
    \node at (-0.9,-1.4) {5};
    \fill (0.63,-1.075)  circle (2pt);
    \node at (0.9,-1.4) {4};
    \draw[dashed] (0.63,-1.075) circle  (7pt);
    \draw (-1.22,0)--(-0.63,1.075)--(0.63,1.075)--(1.22,0)--(0.63,-1.075)--(-0.63,-1.075)--cycle;
    \draw[dashed] (-1.6,0)[rounded corners=15pt]--(-0.4,1.7)[rounded corners=15pt]--(-0.4,-1.7)[rounded corners=12pt]--cycle;
    \end{tikzpicture}}}
    \quad\longrightarrow\quad
    \vcenter{\hbox{\begin{tikzpicture}[scale=0.7]
     \fill (0.63,1.075)  circle (2pt);
     \node at (0.9,1.5) {\{2\}};
    \fill (-1.22,0) circle (2pt);
    \node at (-2.2,0) {\{1,5,6\}};
    \fill (1.22,0) circle (2pt);
    \node at (1.75,0) {\{3\}};
    \fill (0.63,-1.075)  circle (2pt);
    \node at (0.9,-1.5) {\{4\}};
    \draw (-1.22,0)--(0.63,1.075)--(1.22,0)--(0.63,-1.075)--cycle;
    \end{tikzpicture}}}
  \end{gather*}
  \caption{Examples of partitions of graphs and associated contractions.}
  \label{contrpic}
\end{figure}
For a pair of graphical collection $\Pop,\Q$, we define their \textit{contraction product}, that is the graphical collection $\Pop\circ\Q$ with components
\begin{equation}
\label{eq::contract::product}
    (\Pop \circ \Q)(\Gr) := \bigoplus_{I \vdash \Gr} \Pop(\Gr/I) \otimes \bigotimes_{G \in I} \Q(\Gr|_G),
\end{equation}  
where the sum ranges over all partitions of $\Gr$. For an element $\alpha\in \Pop(\Gr/I)$ and collection of elements $\beta_G\in \Q(\Gr|_G)$, we let $(\alpha_{\Gr/I};\beta_{G_1},\beta_{G_2},\cdots,\beta_{G_k})$ be the corresponding element in the product $(\Pop\circ\Q)(\Gr)$. This operation is associative, and the graphical collection $\mathbb{1}$ defined as
\[
\mathbb{1}(\Gr):= \begin{cases}
\mathsf{k}, \text{ for } \Gr \cong \Path_1,
\\
0, \text{ otherwise}.
\end{cases}
\] where $\mathsf{k}$ is the ground field, is the unit for this operation, $\mathbb{1}\circ\Pop\cong \Pop\cong \Pop\circ \mathbb{1}$. 
\begin{defi}[Monoidal definition of contractads]\label{def:monoidal}
A contractad is a monoid in the monoidal category of graphical collections equipped with the contraction product $\circ$.
\end{defi}
By definition, a contractad is a triple $(\Pop,\gamma,\eta)$, where $\Pop$ is a graphical collection, $\gamma$ is a product map $\gamma\colon \Pop\circ\Pop\to \Pop$, that is a collection of maps
\[
    \gamma_I^{\Gr}\colon \Pop(\Gr/I)\otimes \bigotimes_{G \in I} \Pop(\Gr|_G)\to \Pop(\Gr),
\] for each pair $(\Gr,I)$ of graph $\Gr$ and partition $I\vdash \Gr$, and $\eta\colon \mathbb{1}\to \Pop$ is a unit $u\colon \mathsf{k}\to \Pop(\Path_1)$. We let $\Id:=u(1)\in \Pop(\Path_1)$ be the \textit{identity element}, that is a unit for product $\gamma$.

In a dual fashion, we define a cocontractad as a comonoid in the category of graphical collections. In other words, it is a triple $(\Q,\Delta,\epsilon)$, where $\Delta\colon \Q\to\Q\circ\Q$ is a coproduct and $\epsilon\colon \Q(\Path_1)\to \mathsf{k}$ is a counit. In particular, for a contractad $\Pop$, its linear dual graphical collection $\Pop^*$ forms a cocontractad with the coproduct $\Delta:=\gamma^*$ and counit $\epsilon:=\eta^*$.

\begin{remark}
Note that in the definition of contractads, we could replace vector spaces with any symmetric monoidal category. For example, we define a set-contractad as a monoid in the category of set-graphical collection $\Pop\colon\mathsf{CGr}^{\mathrm{op}}\to \mathsf{Sets}$ with values in sets for set-contraction product
\[
(\Pop \circ \Q)(\Gr) := \coprod_{I \vdash \Gr} \Pop(\Gr/I) \times \prod_{G \in I} \Q(\Gr|_G),
\] If we take a linear span in each component of a set contractad $\Pop$, we obtain the contractad $\mathsf{k}[\Pop]$ in the category of vector spaces.
\end{remark}

An equivalent way to present contractads is via \textit{infinitesimal compositions}. When $\Pop$ is a contractad, for each pair of a graph $\Gr$ and tube $G$, there exists a map \[\circ^{\Gr}_G\colon \Pop(\Gr/G)\otimes \Pop(\Gr|_G)\to \Pop(\Gr),\] called the infinitesimal composition and defined by
\begin{gather}
 \Pop(\Gr/G)\otimes\Pop(\Gr|_G) \cong \Pop(\Gr/G)\otimes\Pop(\Gr|_G)\otimes \bigotimes_{v \not\in G} \mathsf{k} \overset{\Id \otimes u^{\otimes}}{\hookrightarrow} \Pop(\Gr/G)\otimes\Pop(\Gr|_G)\otimes \bigotimes_{v \not\in G}\Pop(\Gr|_{\{v\}}) \overset{\gamma}{\rightarrow} \Pop(\Gr).
\\
\alpha\circ^{\Gr}_G \beta:=\gamma(\alpha; \Id,\Id,\cdots,\beta,\cdots, \Id).
\end{gather}
Conversely, from infinitesimal compositions, one can recover all the structure maps of a contractad. 
\subsection{Examples of contractads}\label{sec::examples_of_contractads} We list several examples of contractads, arising in combinatorics of graphs. For other examples of contractads, see~\cite{lyskov2023contractads,khoroshkin2024hilbert}. The description of the contractad $\Plan$ of planar equivalent binary trees we delay to Section~\ref{sec::planartrees}.
\subsubsection{Commutative contractad}
The simplest example of a contractad is the commutative contractad $\Com$ in the category of sets. This contractad has the components $\Com(\Gr)=\{*\}$ with the obvious infinitesimal compositions \[
\circ^{\Gr}_G\colon \Com(\Gr/G)\times \Com(\Gr|_G)\to\Com(\Gr),\quad (*,*)\mapsto*
\]
\subsubsection{Acyclic Directions}
A direction of a graph is an assignment of direction for each edge, i.e., is a function on the edge set $d\colon E_{\Gr}\to \{\rightarrow,\leftarrow\}$. A direction of a graph is acyclic, if there are no directed cycles $v_1\to v_2\to v_3\to\cdots\to v_n\to v_1$. For a graph $\Gr$, we let $\Ass(\Gr)$ be the set of acyclic directions of the underlying graph. Thanks to deletion-contraction formula for the number of acyclic directions, we have $|\Ass(\Gr)|=(-1)^{|V_{\Gr}|}\chi_{\Gr}(-1)$, where $\chi_{\Gr}(q)$ is a chromatic polynomial.

The infinitesimal compositions $\circ^{\Gr}_G\colon \Ass(\Gr/G)\times\Ass(\Gr|_G)\to \Ass(\Gr)$ are defined as follows. For a pair of acyclic directions $\alpha\in \Ass(\Gr/G)$ and $\beta\in\Ass(\Gr|_G)$, we define a new one $\alpha\circ^{\Gr}_G\beta$ by the rule
\[
\begin{cases}
v\to v'\text{, if }v,v'\not\in G\text{ and } v\to v' \text{ in }\alpha,
\\
w\to w'\text{, if }w,w'\in G\text{ and } w\to w' \text{ in }\beta,
\\
v\to w\text{, if }v\not\in G, w\in G\text{ and } v\to \{G\} \text{ in }\alpha,
\\
w\to v\text{, if }v\not\in G, w\in G\text{ and } \{G\}\to v \text{ in }\alpha,
\end{cases}
\] \begin{example}
For a graph $\Gr=\vcenter{\hbox{\begin{tikzpicture}[scale=0.3]
    \fill (0,0) circle (2pt);
    \fill (0,1.5) circle (2pt);
    \fill (1.5,0) circle (2pt);
    \fill (1.5,1.5) circle (2pt);
    \draw (0,0)--(1.5,0)--(1.5,1.5)--(0,1.5)-- cycle;
    \draw (0,0)->(1.5,1.5);
    \node at (-0.25,1.75) {\scriptsize$1$};
    \node at (1.75,1.75) {\scriptsize$2$};
    \node at (1.75,-0.25) {\scriptsize$3$};
    \node at (-0.25,-0.25) {\scriptsize$4$};
    \end{tikzpicture}}}$ and tube $G=\{2,4\}$, we have 
    \begin{gather*}
    \hbox{\begin{tikzpicture}[scale=0.6, edge/.style={->,> = latex, thick}]
    \fill (0,0) circle (2pt);
    \fill (1.5,0) circle (2pt);
    \fill (3,0) circle (2pt);
    \node at (0,0.4) {$1$};
    \node at (1.5,0.45) {\small$\{2,4\}$};
    \node at (3,0.4) {$3$};
    \draw[edge] (1.5,0)--(0,0);
    \draw[edge] (1.5,0)--(3,0);
    \end{tikzpicture}}
    \circ^{\Gr}_{G}
    \hbox{\begin{tikzpicture}[scale=0.6, edge/.style={->,> = latex, thick}]
    \fill (0,0) circle (2pt);
    \fill (1.5,0) circle (2pt);
    \node at (0,0.4) {$4$};
    \node at (1.5,0.4) {$2$};
    \draw[edge] (0,0)--(1.5,0);
    \end{tikzpicture}}=
    \vcenter{\hbox{\begin{tikzpicture}[scale=0.6, edge/.style={->,> = latex, thick}]
    \fill (0,0) circle (2pt);
    \fill (0,1.5) circle (2pt);
    \fill (1.5,0) circle (2pt);
    \fill (1.5,1.5) circle (2pt);
    \draw[edge] (1.5,1.5)--(0,1.5);
    \draw[edge] (1.5,1.5)--(1.5,0);
    \draw[edge] (0,0)--(1.5,0);
    \draw[edge] (0,0)--(0,1.5);
    \draw[edge] (0,0)--(1.5,1.5);
    \node at (-0.25,1.75) {$1$};
    \node at (1.75,1.75) {$2$};
    \node at (1.75,-0.25) {$3$};
    \node at (-0.25,-0.25) {$4$};
    \end{tikzpicture}}},\quad 
    \hbox{\begin{tikzpicture}[scale=0.6, edge/.style={->,> = latex, thick}]
    \fill (0,0) circle (2pt);
    \fill (1.5,0) circle (2pt);
    \fill (3,0) circle (2pt);
    \node at (0,0.4) {$1$};
    \node at (1.5,0.45) {\small$\{2,4\}$};
    \node at (3,0.4) {$3$};
    \draw[edge] (0,0)--(1.5,0);
    \draw[edge] (1.5,0)--(3,0);
    \end{tikzpicture}}
    \circ^{\Gr}_{G}
    \hbox{\begin{tikzpicture}[scale=0.6, edge/.style={->,> = latex, thick}]
    \fill (0,0) circle (2pt);
    \fill (1.5,0) circle (2pt);
    \node at (0,0.4) {$4$};
    \node at (1.5,0.4) {$2$};
    \draw[edge] (1.5,0)--(0,0);
    \end{tikzpicture}}=
    \vcenter{\hbox{\begin{tikzpicture}[scale=0.6, edge/.style={->,> = latex, thick}]
    \fill (0,0) circle (2pt);
    \fill (0,1.5) circle (2pt);
    \fill (1.5,0) circle (2pt);
    \fill (1.5,1.5) circle (2pt);
    \draw[edge] (0,1.5)--(1.5,1.5);
    \draw[edge] (1.5,1.5)--(1.5,0);
    \draw[edge] (0,0)--(1.5,0);
    \draw[edge] (0,1.5)--(0,0);
    \draw[edge] (1.5,1.5)--(0,0);
    \node at (-0.25,1.75) {$1$};
    \node at (1.75,1.75) {$2$};
    \node at (1.75,-0.25) {$3$};
    \node at (-0.25,-0.25) {$4$};
    \end{tikzpicture}}}
\end{gather*}
\end{example} 
\begin{remark}
This contractad was first introduced in~\cite[Sec.~5.1]{lyskov2023contractads} for the description of the topological contractad $\mathcal{D}_1$ generalising the little intervals operad. This contractad encodes configurations of intervals in unit one with certain intersection conditions depending on graphs. In particular, the author showed that acyclic directions encode connected components of this contractad $\pi_0(\mathcal{D}_1)\cong \Ass$.
\end{remark}
\subsubsection{Rooted spanning trees}\label{subsec:spantrees}
The next example of a contractad is the contractad of rooted spanning trees, that generalises the rooted trees operad introduced by Chapoton~\cite{chapoton2001pre}. A spanning tree of $\Gr$ is a subgraph (not induced) $T$ of $\Gr$ that is a tree on the same vertex set as $\Gr$.  A rooted spanning tree is a spanning tree with a selected vertex, root.

We define the rooted spanning trees contractad, denoted $\RST$, as follows. Each component $\RST(\Gr)$ is the vector space generated by spanning trees with a marked vertex: the root. For a rooted tree $(T,r)$ and its tube $G$, we let $(T,r)|_G:=(T|_G,r_G)$ be the induced rooted tree, where $r_G$ is the nearest vertex to $r$ from $G$, and $(T,r)/G:=(T/G,r/G)$ be the contracted rooted tree, where $r/G$ is the image of $r$ under contraction $\Gr\to\Gr/G$. The infinitesimal composition $\circ^{\Gr}_G\colon \RST(\Gr/G)\otimes\RST(\Gr|_G)\to \RST(\Gr)$ is defined as
\[
(T_1,r_1)\circ^{\Gr}_G (T_2,r_2)=\sum_{ (T,r)/G=(T_1,r_1), (T,r)|_G=(T_2,r_2), } (T,r),
\] where the sum ranges over rooted spanning trees $(T,r)$ with contracted and induced rooted trees $(T_1,r_1)$ and $(T_2,r_2)$ respectively.
\begin{example} For the complete graph $\K_4$ and tube $\{1,2\}$, we have
\begin{gather*}
\vcenter{\hbox{\begin{tikzpicture}[
        scale=0.5,
        vert/.style={circle,  draw=black!30!black, fill=white!5, thick, minimum size=1mm, inner sep=1.5pt},
        root/.style={rectangle,  draw=black!30!black, fill=white!5, thick, minimum size=1mm, inner sep=1.5pt},
        edge/.style={-,black!30!black, thick},
        gedge/.style={-, densely dotted, thick},
        ]
        \node[root] (1) at (0,0) {$\scalebox{0.6}{\{1,2\}}$};
        \node[vert] (2) at (2,0) {\small$4$};
        \node[vert] (3) at (1,1.73) {\small$3$};
        \draw[edge] (1)--(2);
        \draw[gedge] (2)--(3);
        \draw[edge] (1)--(3);
    \end{tikzpicture}}}\circ^{\K_4}_{\{1,2\}} \vcenter{\hbox{\begin{tikzpicture}[
        scale=0.5,
        vert/.style={circle,  draw=black!30!black, fill=white!5, thick, minimum size=1mm, inner sep=1.5pt},
        root/.style={rectangle,  draw=black!30!black, fill=white!5, thick, minimum size=1mm},
        edge/.style={-,black!30!black, thick},
        gedge/.style={-, densely dotted, thick},
        ]
        \node[root] (1) at (0,0) {\small$1$};
        \node[vert] (2) at (0,2) {\small$2$};
        \draw[edge] (1)--(2);
\end{tikzpicture}}}=\vcenter{\hbox{\begin{tikzpicture}[
        scale=0.5,
        vert/.style={circle,  draw=black!30!black, fill=white!5, thick, minimum size=1mm, inner sep=1.5pt},
        root/.style={rectangle,  draw=black!30!black, fill=white!5, thick, minimum size=1mm, inner sep=1.5pt},
        edge/.style={-,black!30!black, thick},
        gedge/.style={-, densely dotted, thick},
        ]
        \node[root] (1) at (0,0) {\small$1$};
        \node[vert] (2) at (0,2) {\small$2$};
        \node[vert] (3) at (2,2) {\small$3$};
        \node[vert] (4) at (2,0) {\small$4$};
        \draw[edge] (1)--(2);
        \draw[gedge] (2)--(3);
        \draw[gedge] (3)--(4);
        \draw[edge] (4)--(1);
        \draw[edge] (1)--(3);
        \draw[gedge] (2)--(4);
    \end{tikzpicture}}}+\vcenter{\hbox{\begin{tikzpicture}[
        scale=0.5,
        vert/.style={circle,  draw=black!30!black, fill=white!5, thick, minimum size=1mm, inner sep=1.5pt},
        root/.style={rectangle,  draw=black!30!black, fill=white!5, thick, minimum size=1mm, inner sep=1.5pt},
        edge/.style={-,black!30!black, thick},
        gedge/.style={-, densely dotted, thick},
        ]
        \node[root] (1) at (0,0) {\small$1$};
        \node[vert] (2) at (0,2) {\small$2$};
        \node[vert] (3) at (2,2) {\small$3$};
        \node[vert] (4) at (2,0) {\small$4$};
        \draw[edge] (1)--(2);
        \draw[gedge] (2)--(3);
        \draw[gedge] (3)--(4);
        \draw[gedge] (4)--(1);
        \draw[edge] (1)--(3);
        \draw[edge] (2)--(4);
    \end{tikzpicture}}}
\\
\vcenter{\hbox{\begin{tikzpicture}[
        scale=0.5,
        vert/.style={circle,  draw=black!30!black, fill=white!5, thick, minimum size=1mm, inner sep=1.5pt},
        root/.style={rectangle,  draw=black!30!black, fill=white!5, thick, minimum size=1mm, inner sep=1.5pt},
        edge/.style={-,black!30!black, thick},
        gedge/.style={-, densely dotted, thick},
        ]
        \node[vert] (1) at (0,0) {$\scalebox{0.6}{\{1,2\}}$};
        \node[vert] (2) at (2,0) {\small$4$};
        \node[root] (3) at (1,1.73) {\small$3$};
        \draw[gedge] (1)--(2);
        \draw[edge] (2)--(3);
        \draw[edge] (1)--(3);
    \end{tikzpicture}}}\circ^{\K_4}_{\{1,2\}}  \vcenter{\hbox{\begin{tikzpicture}[
        scale=0.5,
        vert/.style={circle,  draw=black!30!black, fill=white!5, thick, minimum size=1mm, inner sep=1.5pt},
        root/.style={rectangle,  draw=black!30!black, fill=white!5, thick, minimum size=1mm},
        edge/.style={-,black!30!black, thick},
        gedge/.style={-, densely dotted, thick},
        ]
        \node[root] (1) at (0,0) {\small$1$};
        \node[vert] (2) at (0,2) {\small$2$};
        \draw[edge] (1)--(2);
\end{tikzpicture}}}=\vcenter{\hbox{\begin{tikzpicture}[
        scale=0.5,
        vert/.style={circle,  draw=black!30!black, fill=white!5, thick, minimum size=1mm, inner sep=1.5pt},
        root/.style={rectangle,  draw=black!30!black, fill=white!5, thick, minimum size=1mm, inner sep=1.5pt},
        edge/.style={-,black!30!black, thick},
        gedge/.style={-, densely dotted, thick},
        ]
        \node[vert] (1) at (0,0) {\small$1$};
        \node[vert] (2) at (0,2) {\small$2$};
        \node[root] (3) at (2,2) {\small$3$};
        \node[vert] (4) at (2,0) {\small$4$};
        \draw[edge] (1)--(2);
        \draw[gedge] (2)--(3);
        \draw[edge] (3)--(4);
        \draw[gedge] (4)--(1);
        \draw[edge] (1)--(3);
        \draw[gedge] (2)--(4);
        \end{tikzpicture}}}
\end{gather*}
\end{example}
\subsubsection{Hamiltonian paths}
We introduce the main example of a contractad. Recall that a \textit{Hamiltonian path} in a graph $\Gr$ is a path in $\Gr$ that visits each vertex exactly once. Note that a Hamiltonian path admits two natural directions from one endpoint to another one. We shall use notation $(v_1,\cdots,v_n)$ for directed Hamiltonian path $v_1\to v_2\to \cdots\to  v_{n-1}\to v_n$.

For a graph $\Gr$, let $\Ham(\Gr)$ be the linear span of all directed Hamiltonian paths in $\Gr$. For a tube $G\subset V_{\Gr}$, we define the substitution of paths
\[
\circ^{\Gr}_G\colon \Ham(\Gr/G)\otimes \Ham(\Gr|_G)\to \Ham(\Gr)
\] as follows. Let $\Path=(v_1,\cdots,v_{k-1},\{G\},v_{k+1},\cdots,v_n)$ and $\Path'=(w_1,w_2,\cdots, w_m)$ be directed Hamiltonian paths in contracted $\Gr/G$ and induced $\Gr|_G$ graphs respectively. If $v_{k-1}$ is adjacent to $w_1$ and $w_m$ to $v_{k+1}$, we put $\Path\circ^{\Gr}_G \Path'=(v_1,\cdots,v_{k-1},w_1,w_2,\cdots, w_m,v_{k+1},\cdots,v_n)$, otherwise, we put zero. 

\begin{example}
For a graph $\Gr=\vcenter{\hbox{\begin{tikzpicture}[scale=0.3]
    \fill (0,0) circle (2pt);
    \fill (0,1.5) circle (2pt);
    \fill (1.5,0) circle (2pt);
    \fill (1.5,1.5) circle (2pt);
    \draw (0,0)--(1.5,0)--(1.5,1.5)--(0,1.5)-- cycle;
    \draw (0,0)->(1.5,1.5);
    \node at (-0.25,1.75) {\scriptsize$1$};
    \node at (1.75,1.75) {\scriptsize$2$};
    \node at (1.75,-0.25) {\scriptsize$3$};
    \node at (-0.25,-0.25) {\scriptsize$4$};
    \end{tikzpicture}}}$ and tube $G=\{1,2\}$, we have
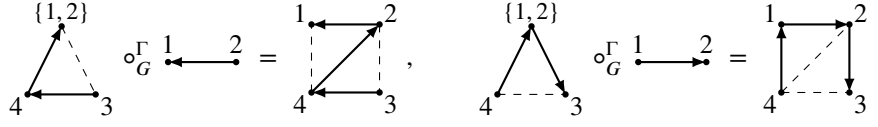
\begin{gather*}
    \vcenter{\hbox{\begin{tikzpicture}[scale=0.6, edge/.style={->,> = latex, thick}]
    \fill (0,0) circle (2pt);
    \fill (1.5,0) circle (2pt);
    \fill (0.75,1.5) circle (2pt);
    \draw[dashed] (0.75,1.5)--(1.5,0);
    \draw[edge] (1.5,0)--(0,0);
    \draw[edge] (0,0)--(0.75,1.5);
    \node at (0.75,1.8) {\small$\{1,2\}$};
    \node at (1.75,-0.25) {$3$};
    \node at (-0.25,-0.25) {$4$};
    \end{tikzpicture}}}   \circ^{\Gr}_{G}
    \hbox{\begin{tikzpicture}[scale=0.6, edge/.style={->,> = latex, thick}]
    \fill (0,0) circle (2pt);
    \fill (1.5,0) circle (2pt);
    \node at (0,0.4) {$1$};
    \node at (1.5,0.4) {$2$};
    \draw[edge] (1.5,0)--(0,0);   \end{tikzpicture}}=\vcenter{\hbox{\begin{tikzpicture}[scale=0.6, edge/.style={->,> = latex, thick}]
    \fill (0,0) circle (2pt);
    \fill (0,1.5) circle (2pt);
    \fill (1.5,0) circle (2pt);
    \fill (1.5,1.5) circle (2pt);
    \draw[edge] (1.5,1.5)--(0,1.5);
    \draw[dashed] (1.5,1.5)--(1.5,0);
    \draw[edge] (1.5,0)--(0,0);
    \draw[dashed] (0,0)--(0,1.5);
    \draw[edge] (0,0)--(1.5,1.5);
    \node at (-0.25,1.75) {$1$};
    \node at (1.75,1.75) {$2$};
    \node at (1.75,-0.25) {$3$};
    \node at (-0.25,-0.25) {$4$};
    \end{tikzpicture}}},\qquad
    \vcenter{\hbox{\begin{tikzpicture}[scale=0.6, edge/.style={->,> = latex, thick}]
    \fill (0,0) circle (2pt);
    \fill (1.5,0) circle (2pt);
    \fill (0.75,1.5) circle (2pt);
    \draw[edge] (0.75,1.5)--(1.5,0);
    \draw[dashed] (0,0)--(1.5,0);
    \draw[edge] (0,0)--(0.75,1.5);
    \node at (0.75,1.8) {\small$\{1,2\}$};
    \node at (1.75,-0.25) {$3$};
    \node at (-0.25,-0.25) {$4$};
    \end{tikzpicture}}}   \circ^{\Gr}_{G}
    \hbox{\begin{tikzpicture}[scale=0.6, edge/.style={->,> = latex, thick}]
    \fill (0,0) circle (2pt);
    \fill (1.5,0) circle (2pt);
    \node at (0,0.4) {$1$};
    \node at (1.5,0.4) {$2$};
    \draw[edge] (0,0)--(1.5,0);   \end{tikzpicture}}=\vcenter{\hbox{\begin{tikzpicture}[scale=0.6, edge/.style={->,> = latex, thick}]
    \fill (0,0) circle (2pt);
    \fill (0,1.5) circle (2pt);
    \fill (1.5,0) circle (2pt);
    \fill (1.5,1.5) circle (2pt);
    \draw[edge] (0,1.5)--(1.5,1.5);
    \draw[edge] (1.5,1.5)--(1.5,0);
    \draw[dashed] (0,0)--(1.5,0);
    \draw[edge] (0,0)--(0,1.5);
    \draw[dashed] (0,0)--(1.5,1.5);
    \node at (-0.25,1.75) {$1$};
    \node at (1.75,1.75) {$2$};
    \node at (1.75,-0.25) {$3$};
    \node at (-0.25,-0.25) {$4$};
    \end{tikzpicture}}}.
\end{gather*} But since vertices $1$ and $3$ are not adjacent, we have
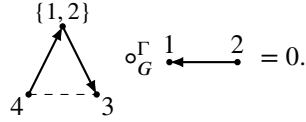
\begin{gather*}    \vcenter{\hbox{\begin{tikzpicture}[scale=0.6, edge/.style={->,> = latex, thick}]
    \fill (0,0) circle (2pt);
    \fill (1.5,0) circle (2pt);
    \fill (0.75,1.5) circle (2pt);
    \draw[edge] (0.75,1.5)--(1.5,0);
    \draw[dashed] (0,0)--(1.5,0);
    \draw[edge] (0,0)--(0.75,1.5);
    \node at (0.75,1.8) {\small$\{1,2\}$};
    \node at (1.75,-0.25) {$3$};
    \node at (-0.25,-0.25) {$4$};
    \end{tikzpicture}}}   \circ^{\Gr}_{G}
    \hbox{\begin{tikzpicture}[scale=0.6, edge/.style={->,> = latex, thick}]
    \fill (0,0) circle (2pt);
    \fill (1.5,0) circle (2pt);
    \node at (0,0.4) {$1$};
    \node at (1.5,0.4) {$2$};
    \draw[edge] (1.5,0)--(0,0);   \end{tikzpicture}}=0.
\end{gather*}  
\end{example}
By direct inspections, we see that path substitutions endow the graphical collection $\Ham$ into the contractad, that we call the contractad of Hamiltonian paths
\subsection{Presentation of contractads}
Similarly to operads, contractads can be defined in terms of rooted trees and their compositions with small modifications. A \textit{rooted tree} is a connected directed tree $T$ in which each vertex has at least one input edge and exactly one output edge. This tree should have exactly one external outgoing edge, output. The endpoint of this edge is called the \textit{root}. The endpoints of incoming external edges that are not vertices are called \textit{leaves}. A tree with a single vertex is called a \textit{corolla}. For a rooted tree $T$ and edge $e$, let $T_e$ be the subtree  with the root at $e$.

\begin{defi}
For a connected graph $\Gr$, a $\Gr$-\textit{admissible} rooted tree is a rooted tree $T$ with leaves labeled by the vertex set $V_{\Gr}$ of the given graph such that, for each edge $e$ of the tree, the leaves of subtree $T_e$ form a tube of $\Gr$.
\end{defi} 

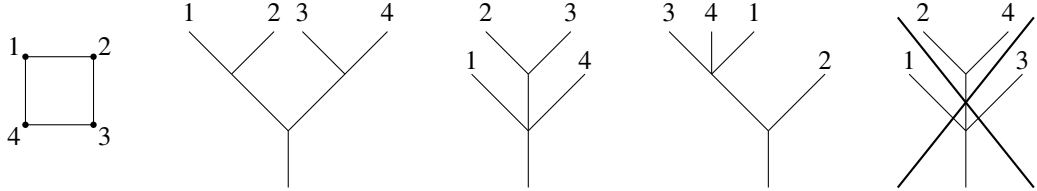
\begin{figure}[ht] 
  \centering
\[
\vcenter{\hbox{\begin{tikzpicture}[scale=0.6]
    \fill (0,0) circle (2pt);
    \fill (0,1.5) circle (2pt);
    \fill (1.5,0) circle (2pt);
    \fill (1.5,1.5) circle (2pt);
    \draw (0,0)--(1.5,0)--(1.5,1.5)--(0,1.5)-- cycle;
    \node at (-0.25,1.75) {$1$};
    \node at (1.75,1.75) {$2$};
    \node at (1.75,-0.25) {$3$};
    \node at (-0.25,-0.25) {$4$};
  \end{tikzpicture}}}
\qquad
\vcenter{\hbox{\begin{tikzpicture}[scale=0.75]
        \draw (0,0)--(0,1);
        \draw (0,1)--(1,2);
        \draw (0,1)--(-1,2);
        \draw (1,2)--(1.75,2.75);
        \draw (1,2)--(0.25,2.75);
        \draw (-1,2)--(-1.75,2.75);
        \draw (-1,2)--(-0.25,2.75);
        \node at (1.75,3.1) {$4$};
        \node at (0.25,3.1) {$3$};
        \node at (-1.75,3.1) {$1$};
        \node at (-0.25,3.1) {$2$};
  \end{tikzpicture}}}
\qquad 
\vcenter{\hbox{\begin{tikzpicture}[scale=0.75]
        \draw (0,0)--(0,1);
        \draw (0,1)--(1,2);
        \draw (0,1)--(-1,2);
        \draw (0,1)--(0,2);
        \draw (0,2)--(0.75,2.75);
        \draw (0,2)--(-0.75,2.75);
        \node at (-1,2.25) {$1$};
        \node at (-0.75,3.1) {$2$};
        \node at (0.75,3.1) {$3$};
        \node at (1,2.25) {$4$};
  \end{tikzpicture}}}
\qquad 
\vcenter{\hbox{\begin{tikzpicture}[scale=0.75]
        \draw (0,0)--(0,1);
        \draw (0,1)--(1,2);
        \draw (0,1)--(-1,2);
        \draw (-1,2)--(-1.75,2.75);
        \draw (-1,2)--(-1,2.75);
        \draw (-1,2)--(-0.25,2.75);
        \node at (-1.75,3.1) {$3$};
        \node at (-1,3.1) {$4$};
        \node at (-0.25,3.1) {$1$};
        \node at (1,2.25) {$2$};
  \end{tikzpicture}}}
\qquad 
\vcenter{\hbox{\begin{tikzpicture}[scale=0.75]
        \draw[thick] (-1.2,3)--(1.2,0);
        \draw[thick] (-1.2,0)--(1.2,3);
        \draw (0,0)--(0,1);
        \draw (0,1)--(1,2);
        \draw (0,1)--(-1,2);
        \draw (0,1)--(0,2);
        \draw (0,2)--(0.75,2.75);
        \draw (0,2)--(-0.75,2.75);
        \node at (-1,2.25) {$1$};
        \node at (-0.75,3.1) {$2$};
        \node at (0.75,3.1) {$4$};
        \node at (1,2.25) {$3$};
  \end{tikzpicture}}}
\]
\caption{Graph $\Cyc_4$ (on the left side) and examples of $\Cyc_4$-admissible trees. The first three are $\Cyc_4$-admissible, but the fourth is not, since leaves 2,4 do not form a tube.}
\label{roottrees}
\end{figure}

 We denote by $\Tree(\Gr)$ the set of all $\Gr$-admissible rooted trees. Note that the corolla with leaves labeled by the vertex set is always $\Gr$-admissible.
 \begin{example}
For particular types of graphs, we have
\begin{itemize}
\item For complete graphs, $\K_n$-admissible trees are ordinary rooted trees since each vertex subset of a complete graph is a tube.

\item For  paths, $\Path_n$-admissible trees are those that can be embedded in the plane with leaves $[n]=\{1,2,3,...,n\}$ arranged in increasing order. Indeed, this follows from the  fact  that tubes of $\Path_n$ are ordered intervals of $[n]$.
 
\item For  cycles, $\Cyc_n$-admissible trees are those that can be embedded in the plane such that leaves are arranged in cyclic order, as in Figure~\ref{roottrees}.
\end{itemize}
\end{example}
For a partition $I$ of the graph $\Gr$, we define the grafting map
\[
    \Tree(\Gr/I)\times\prod_{G \in I} \Tree(\Gr|_G)\to \Tree(\Gr),
\] which joins roots of $\Gr|_G$-admissible trees to leaves of $\Gr/I$-admissible trees, as in Figure~\ref{substitution}. These operations endow the graphical collection $\Tree$ with the contractad structure.
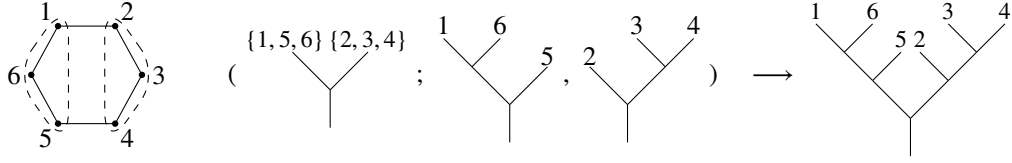
\begin{figure}[ht]
  \centering
  \[
  \vcenter{\hbox{\begin{tikzpicture}[scale=0.6]
    \fill (-0.63,1.075)  circle (2pt);
    \node at (-0.9,1.4) {1};
    \fill (0.63,1.075)  circle (2pt);
    \node at (0.9,1.4) {2};
    \fill (1.22,0) circle (2pt);
    \node at (1.6,0) {3};
    \fill (0.63,-1.075)  circle (2pt);
    \node at (0.9,-1.4) {4};
    \fill (-0.63,-1.075)  circle (2pt);
    \node at (-0.9,-1.4) {5};
    \fill (-1.22,0) circle (2pt);
    \node at (-1.6,0) {6};
    \draw (-1.22,0)--(-0.63,1.075)--(0.63,1.075)--(1.22,0)--(0.63,-1.075)--(-0.63,-1.075)--cycle;
    \draw[dashed] (-1.6,0)[rounded corners=15pt]--(-0.4,1.7)[rounded corners=15pt]--(-0.4,-1.7)[rounded corners=12pt]--cycle;
    \draw[dashed] (1.6,0)[rounded corners=15pt]--(0.4,1.7)[rounded corners=15pt]--(0.4,-1.7)[rounded corners=12pt]--cycle;
    \end{tikzpicture}}}
    \quad\quad
    (\vcenter{\hbox{\begin{tikzpicture}[scale=0.5]
        \draw (0,0)--(0,1);
        \draw (0,1)--(1,2);
        \draw (0,1)--(-1,2);
        \node at (-1.2,2.3) {\small $\{1,5,6\}$};
        \node at (1,2.3) {\small $\{2,3,4\}$};
     \end{tikzpicture}}};
     \vcenter{\hbox{\begin{tikzpicture}[scale=0.5]
        \draw (0,0)--(0,1);
        \draw (0,1)--(1,2);
        \draw (0,1)--(-1,2);
        \draw (-1,2)--(-1.75,2.75);
        \draw (-1,2)--(-0.25,2.75);
        \node at (-1.75,3.1) {$1$};
        \node at (-0.25,3.1) {$6$};
        \node at (1,2.3) {$5$};
     \end{tikzpicture}}},
     \vcenter{\hbox{\begin{tikzpicture}[scale=0.5]
        \draw (0,0)--(0,1);
        \draw (0,1)--(1,2);
        \draw (0,1)--(-1,2);
        \draw (1,2)--(1.75,2.75);
        \draw (1,2)--(0.25,2.75);
        \node at (-1,2.3) {$2$};
        \node at (0.25,3.1) {$3$};
        \node at (1.75,3.1) {$4$};
    \end{tikzpicture}}})
\quad \longrightarrow
    \vcenter{\hbox{\begin{tikzpicture}[scale=0.5 ]
        \draw (0,0)--(0,1);
        \draw (0,1)--(1,2);
        \draw (0,1)--(-1,2);
        \draw (1,2)--(1.75,2.75);
        \draw (1,2)--(0.25,2.75);
        \draw (1.75,2.75)--(2.5,3.5);
        \draw (1.75,2.75)--(1,3.5);
        \draw (-1,2)--(-1.75,2.75);
        \draw (-1,2)--(-0.25,2.75);
        \draw (-1.75,2.75)--(-2.5,3.5);
        \draw (-1.75,2.75)--(-1,3.5);
        \node at (-2.5,3.85) {\small$1$};
        \node at (-1,3.85) {\small$6$};
        \node at (-0.25,3.1) {\small$5$};
        \node at (0.25,3.1) {\small$2$};
        \node at (1,3.85) {\small$3$};
        \node at (2.5,3.85) {\small$4$};
        
  \end{tikzpicture}}}
  \]
  \caption{Example of substitution.}
  \label{substitution}
\end{figure}

The grafting of trees allows us to give an explicit construction of free contractads $\T(\E)$. By the "free" as usual we mean, that any contractad morphism $\T(\E)\to \Pop$ is determined by the image of generators, that is $\Hom_{\mathsf{Con}}(\T(\E),\Pop)\cong \Hom_{\GrCol}(\E,\Pop)$.   Informally, for a graphical collection $\E$, the free contractad $\T(\E)$ generated by $\E$, is the contractad with components $\T(\E)(\Gr)$ generated $\Gr$-admissible rooted tree in which internal vertices are decorated by elements from $\E$ and contractad structure is given by substitution of decorated trees. For explicit description, see~\cite[Sec.~1.3]{lyskov2023contractads}.
\begin{example}\label{ex::freecont}
To make some intuition, let us consider several examples of free contractads $\T(\E)$ for particular types of generators $\E$
\begin{itemize}
    \item Let $\E$ be one-dimensional in each component, $\E(\Gr)\cong\mathsf{k}$, then the free contractad $\T(\E)$ is isomorphic to the linearisation of contractads of rooted trees $\Tree$.
    \item A contractad is called \textit{binary} if its generators $\E_2$ are concentrated in the component $\Path_2$. Since $\Aut(\Path_2)\cong \mathbb{Z}_2$, $\E_2$ is just a $\mathbb{Z}_2$-presentation. Each component $\T(\E_2)(\Gr)$ of this contractad is a linear span of $\Gr$-admissible binary\footnote{binary tree is a rooted tree in which internal vertices have two inputs}, in which internal vertices are labelled by elements of $\E_2$.   

    \item Consider the contractad $\BiPlan$ of planar binary trees, whose components $\BiPlan(\Gr)$ consist of $\Gr$-admissible binary trees with additional planar structure. For example, the component $\BiPlan(\Path_3)$ consists of eight different planar trees
    \begin{gather*}
    \vcenter
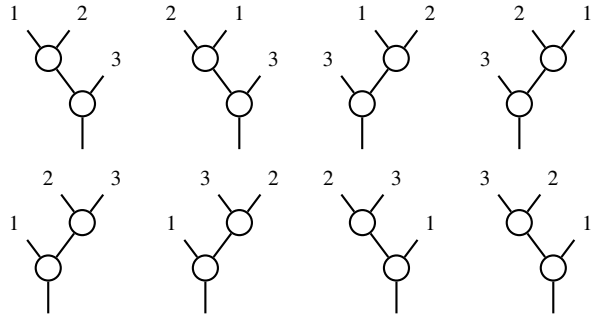
{\hbox{\begin{tikzpicture}[
        scale=0.6,
        vert/.style={circle,  draw=black!30!black, thick, minimum size=1mm},
        leaf/.style={rectangle, thick, minimum size=1mm},
        edge/.style={-,black!30!black, thick},
        ]
        \node[vert] (1) at (0,1) {\footnotesize$\space$};
        \node[leaf] (l1) at (0.75,2) {\footnotesize$3$};
        \node[vert] (2) at (-0.75,2) {\footnotesize$\space$};
        \node[leaf] (l2) at (0,3) {\footnotesize$2$};
        \node[leaf] (3) at (-1.5,3) {\footnotesize$1$};
        \draw[edge] (0,0)--(1);
        \draw[edge] (1)--(2)--(3);
        \draw[edge] (1)--(l1);
        \draw[edge] (2)--(l2);
    \end{tikzpicture}}}\quad
    \vcenter{\hbox{\begin{tikzpicture}[
        scale=0.6,
        vert/.style={circle,  draw=black!30!black, thick, minimum size=1mm},
        leaf/.style={rectangle, thick, minimum size=1mm},
        edge/.style={-,black!30!black, thick},
        ]
        \node[vert] (1) at (0,1) {\footnotesize$\space$};
        \node[leaf] (l1) at (0.75,2) {\footnotesize$3$};
        \node[vert] (2) at (-0.75,2) {\footnotesize$\space$};
        \node[leaf] (l2) at (0,3) {\footnotesize$1$};
        \node[leaf] (3) at (-1.5,3) {\footnotesize$2$};
        \draw[edge] (0,0)--(1);
        \draw[edge] (1)--(2)--(3);
        \draw[edge] (1)--(l1);
        \draw[edge] (2)--(l2);
    \end{tikzpicture}}}\quad
    \vcenter{\hbox{\begin{tikzpicture}[
        scale=0.6,
        vert/.style={circle,  draw=black!30!black, thick, minimum size=1mm},
        leaf/.style={rectangle, thick, minimum size=1mm},
        edge/.style={-,black!30!black, thick},
        ]
        \node[vert] (1) at (0,1) {\footnotesize$\space$};
        \node[leaf] (l1) at (-0.75,2) {\footnotesize$3$};
        \node[vert] (2) at (0.75,2) {\footnotesize$\space$};
        \node[leaf] (l2) at (0,3) {\footnotesize$1$};
        \node[leaf] (3) at (1.5,3) {\footnotesize$2$};
        \draw[edge] (0,0)--(1);
        \draw[edge] (1)--(2)--(3);
        \draw[edge] (1)--(l1);
        \draw[edge] (2)--(l2);
    \end{tikzpicture}}}\quad
    \vcenter{\hbox{\begin{tikzpicture}[
        scale=0.6,
        vert/.style={circle,  draw=black!30!black, thick, minimum size=1mm},
        leaf/.style={rectangle, thick, minimum size=1mm},
        edge/.style={-,black!30!black, thick},
        ]
        \node[vert] (1) at (0,1) {\footnotesize$\space$};
        \node[leaf] (l1) at (-0.75,2) {\footnotesize$3$};
        \node[vert] (2) at (0.75,2) {\footnotesize$\space$};
        \node[leaf] (l2) at (0,3) {\footnotesize$2$};
        \node[leaf] (3) at (1.5,3) {\footnotesize$1$};
        \draw[edge] (0,0)--(1);
        \draw[edge] (1)--(2)--(3);
        \draw[edge] (1)--(l1);
        \draw[edge] (2)--(l2);
    \end{tikzpicture}}}
    \\
    \vcenter{\hbox{\begin{tikzpicture}[
        scale=0.6,
        vert/.style={circle,  draw=black!30!black, thick, minimum size=1mm},
        leaf/.style={rectangle, thick, minimum size=1mm},
        edge/.style={-,black!30!black, thick},
        ]
        \node[vert] (1) at (0,1) {\footnotesize$\space$};
        \node[leaf] (l1) at (-0.75,2) {\footnotesize$1$};
        \node[vert] (2) at (0.75,2) {\footnotesize$\space$};
        \node[leaf] (l2) at (0,3) {\footnotesize$2$};
        \node[leaf] (3) at (1.5,3) {\footnotesize$3$};
        \draw[edge] (0,0)--(1);
        \draw[edge] (1)--(2)--(3);
        \draw[edge] (1)--(l1);
        \draw[edge] (2)--(l2);
    \end{tikzpicture}}}\quad     \vcenter{\hbox{\begin{tikzpicture}[
        scale=0.6,
        vert/.style={circle,  draw=black!30!black, thick, minimum size=1mm},
        leaf/.style={rectangle, thick, minimum size=1mm},
        edge/.style={-,black!30!black, thick},
        ]
        \node[vert] (1) at (0,1) {\footnotesize$\space$};
        \node[leaf] (l1) at (-0.75,2) {\footnotesize$1$};
        \node[vert] (2) at (0.75,2) {\footnotesize$\space$};
        \node[leaf] (l2) at (0,3) {\footnotesize$3$};
        \node[leaf] (3) at (1.5,3) {\footnotesize$2$};
        \draw[edge] (0,0)--(1);
        \draw[edge] (1)--(2)--(3);
        \draw[edge] (1)--(l1);
        \draw[edge] (2)--(l2);
    \end{tikzpicture}}}\quad \vcenter{\hbox{\begin{tikzpicture}[
        scale=0.6,
        vert/.style={circle,  draw=black!30!black, thick, minimum size=1mm},
        leaf/.style={rectangle, thick, minimum size=1mm},
        edge/.style={-,black!30!black, thick},
        ]
        \node[vert] (1) at (0,1) {\footnotesize$\space$};
        \node[leaf] (l1) at (0.75,2) {\footnotesize$1$};
        \node[vert] (2) at (-0.75,2) {\footnotesize$\space$};
        \node[leaf] (l2) at (0,3) {\footnotesize$3$};
        \node[leaf] (3) at (-1.5,3) {\footnotesize$2$};
        \draw[edge] (0,0)--(1);
        \draw[edge] (1)--(2)--(3);
        \draw[edge] (1)--(l1);
        \draw[edge] (2)--(l2);
    \end{tikzpicture}}}\quad
    \vcenter{\hbox{\begin{tikzpicture}[
        scale=0.6,
        vert/.style={circle,  draw=black!30!black, thick, minimum size=1mm},
        leaf/.style={rectangle, thick, minimum size=1mm},
        edge/.style={-,black!30!black, thick},
        ]
        \node[vert] (1) at (0,1) {\footnotesize$\space$};
        \node[leaf] (l1) at (0.75,2) {\footnotesize$1$};
        \node[vert] (2) at (-0.75,2) {\footnotesize$\space$};
        \node[leaf] (l2) at (0,3) {\footnotesize$2$};
        \node[leaf] (3) at (-1.5,3) {\footnotesize$3$};
        \draw[edge] (0,0)--(1);
        \draw[edge] (1)--(2)--(3);
        \draw[edge] (1)--(l1);
        \draw[edge] (2)--(l2);
    \end{tikzpicture}}}
\end{gather*} This contractad is a special case of a binary contractad when $\E_2=\langle \nu,\nu^{(12)}\rangle$ is a regular $\mathbb{Z}_2$-representation. Indeed, the identification $\T(\nu,\nu^{(12)})\cong \BiPlan$ follows from the correspondence  
\[
\T(\nu,\nu^{(12)})\overset{\cong}{\to} \BiPlan, \quad \vcenter{\hbox{\begin{tikzpicture}[
        scale=0.6,
        vert/.style={circle,  draw=black!30!black, thick, minimum size=1mm},
        leaf/.style={rectangle, thick, minimum size=1mm},
        edge/.style={-,black!30!black, thick},
        ]
        \node[vert] (1) at (0,1) {\footnotesize$\nu$};
        \node[leaf] (l1) at (-0.75,2) {\footnotesize$1$};
        \node[leaf] (l2) at (0.75,2) {\footnotesize$2$};
        \draw[edge] (0,0)--(1);
        \draw[edge] (1)--(l1);
        \draw[edge] (1)--(l2);
    \end{tikzpicture}}} \mapsto \vcenter{\hbox{\begin{tikzpicture}[
        scale=0.6,
        vert/.style={circle,  draw=black!30!black, thick, minimum size=1mm},
        leaf/.style={rectangle, thick, minimum size=1mm},
        edge/.style={-,black!30!black, thick},
        ]
        \node[vert] (1) at (0,1) {\footnotesize$\space$};
        \node[leaf] (l1) at (-0.75,2) {\footnotesize$1$};
        \node[leaf] (l2) at (0.75,2) {\footnotesize$2$};
        \draw[edge] (0,0)--(1);
        \draw[edge] (1)--(l1);
        \draw[edge] (1)--(l2);
\end{tikzpicture}}}, \vcenter{\hbox{\begin{tikzpicture}[
        scale=0.6,
        vert/.style={inner sep=1pt, circle,  draw=black!30!black, thick, minimum size=1mm},
        leaf/.style={rectangle, thick, minimum size=1mm},
        edge/.style={-,black!30!black, thick},
        ]
        \node[vert] (1) at (0,1) {\scriptsize$\nu^{\scalebox{0.7}{(12)}}$};
        \node[leaf] (l1) at (-0.75,2) {\footnotesize$1$};
        \node[leaf] (l2) at (0.75,2) {\footnotesize$2$};
        \draw[edge] (0,0)--(1);
        \draw[edge] (1)--(l1);
        \draw[edge] (1)--(l2);
    \end{tikzpicture}}} \mapsto \vcenter{\hbox{\begin{tikzpicture}[
        scale=0.6,
        vert/.style={circle,  draw=black!30!black, thick, minimum size=1mm},
        leaf/.style={rectangle, thick, minimum size=1mm},
        edge/.style={-,black!30!black, thick},
        ]
        \node[vert] (1) at (0,1) {\footnotesize$\space$};
        \node[leaf] (l1) at (-0.75,2) {\footnotesize$2$};
        \node[leaf] (l2) at (0.75,2) {\footnotesize$1$};
        \draw[edge] (0,0)--(1);
        \draw[edge] (1)--(l1);
        \draw[edge] (1)--(l2);
\end{tikzpicture}}}.
\] For example, we have
\[
\nu\circ^{\Path_3}_{\{1,2\}}\nu=    \vcenter{\hbox{\begin{tikzpicture}[
        scale=0.6,
        vert/.style={circle,  draw=black!30!black, thick, minimum size=1mm},
        leaf/.style={rectangle, thick, minimum size=1mm},
        edge/.style={-,black!30!black, thick},
        ]
        \node[vert] (1) at (0,1) {\footnotesize$\space$};
        \node[leaf] (l1) at (0.75,2) {\footnotesize$3$};
        \node[vert] (2) at (-0.75,2) {\footnotesize$\space$};
        \node[leaf] (l2) at (0,3) {\footnotesize$2$};
        \node[leaf] (3) at (-1.5,3) {\footnotesize$1$};
        \draw[edge] (0,0)--(1);
        \draw[edge] (1)--(2)--(3);
        \draw[edge] (1)--(l1);
        \draw[edge] (2)--(l2);
    \end{tikzpicture}}},\quad \nu\circ^{\Path_3}_{\{1,2\}}\nu^{\scalebox{0.7}{(12)}}=    \vcenter{\hbox{\begin{tikzpicture}[
        scale=0.6,
        vert/.style={circle,  draw=black!30!black, thick, minimum size=1mm},
        leaf/.style={rectangle, thick, minimum size=1mm},
        edge/.style={-,black!30!black, thick},
        ]
        \node[vert] (1) at (0,1) {\footnotesize$\space$};
        \node[leaf] (l1) at (0.75,2) {\footnotesize$3$};
        \node[vert] (2) at (-0.75,2) {\footnotesize$\space$};
        \node[leaf] (l2) at (0,3) {\footnotesize$1$};
        \node[leaf] (3) at (-1.5,3) {\footnotesize$2$};
        \draw[edge] (0,0)--(1);
        \draw[edge] (1)--(2)--(3);
        \draw[edge] (1)--(l1);
        \draw[edge] (2)--(l2);
    \end{tikzpicture}}}, \quad \nu^{\scalebox{0.7}{(\{1,2\}3)}}\circ^{\Path_3}_{\{1,2\}}\nu=\vcenter{\hbox{\begin{tikzpicture}[
        scale=0.6,
        vert/.style={circle,  draw=black!30!black, thick, minimum size=1mm},
        leaf/.style={rectangle, thick, minimum size=1mm},
        edge/.style={-,black!30!black, thick},
        ]
        \node[vert] (1) at (0,1) {\footnotesize$\space$};
        \node[leaf] (l1) at (-0.75,2) {\footnotesize$3$};
        \node[vert] (2) at (0.75,2) {\footnotesize$\space$};
        \node[leaf] (l2) at (0,3) {\footnotesize$1$};
        \node[leaf] (3) at (1.5,3) {\footnotesize$2$};
        \draw[edge] (0,0)--(1);
        \draw[edge] (1)--(2)--(3);
        \draw[edge] (1)--(l1);
        \draw[edge] (2)--(l2);
    \end{tikzpicture}}}.
\]
\end{itemize}
\end{example}
We say that a contractad $\Pop$ is presented by generators $\E$ and relations $\R\subset \T(\E)$, if $\Pop$ is the quotient contractad
$\T(\E)/\langle \R \rangle$ where $\langle \R \rangle$ is the minimal ideal containing $\R$. 

\begin{example}\label{ex::presentations} Let us describe presentations of contractads from Section~\ref{sec::examples_of_contractads}. 
\begin{itemize}
\item \cite[Pr.~4.3.2]{lyskov2023contractads} The commutative contractad $\Com$ is generated by a symmetric generator $m$, $m^{(12)}=m$, in the component $\Path_2$, satisfying the relations
\begin{gather*}
m \circ_{\{1,2\}}^{\mathsf{K_3}} m - m \circ_{\{2,3\}}^{\mathsf{K_3}} m
\\
m \circ_{\{1,2\}}^{\mathsf{P_3}} m - m \circ_{\{2,3\}}^{\mathsf{P_3}} m
\end{gather*} Note that we do not explicitly mention the relations produced by graph automorphisms. For example, the action of the transposition $(12)$ on the second relation results in the following relation:
\[
 (m \circ_{\{1,2\}}^{\mathsf{K_3}} m - m \circ_{\{2,3\}}^{\mathsf{K_3}} m)^{(12)}=m \circ_{\{1,2\}}^{\mathsf{K_3}} m - m \circ_{\{1,3\}}^{\mathsf{K_3}} m.
\] In terms of (non-planar) binary trees, both relations are written in the following way  
\[
    \vcenter{\hbox{\begin{tikzpicture}[
        scale=0.6,
        vert/.style={circle,  draw=black!30!black, thick, minimum size=1mm, inner sep=2pt},
        leaf/.style={rectangle, thick, minimum size=1mm},
        edge/.style={-,black!30!black, thick},
        ]
        \node[vert] (1) at (0,1) {\footnotesize$m$};
        \node[leaf] (l1) at (0.75,2) {\footnotesize$3$};
        \node[vert] (2) at (-0.75,2) {\footnotesize$m$};
        \node[leaf] (l2) at (0,3) {\footnotesize$2$};
        \node[leaf] (3) at (-1.5,3) {\footnotesize$1$};
        \draw[edge] (0,0)--(1);
        \draw[edge] (1)--(2)--(3);
        \draw[edge] (1)--(l1);
        \draw[edge] (2)--(l2);
    \end{tikzpicture}}}=    \vcenter{\hbox{\begin{tikzpicture}[
        scale=0.6,
        vert/.style={circle,  draw=black!30!black, thick, minimum size=1mm, inner sep=2pt},
        leaf/.style={rectangle, thick, minimum size=1mm},
        edge/.style={-,black!30!black, thick},
        ]
        \node[vert] (1) at (0,1) {\footnotesize$m$};
        \node[leaf] (l1) at (-0.75,2) {\footnotesize$1$};
        \node[vert] (2) at (0.75,2) {\footnotesize$m$};
        \node[leaf] (l2) at (0,3) {\footnotesize$2$};
        \node[leaf] (3) at (1.5,3) {\footnotesize$3$};
        \draw[edge] (0,0)--(1);
        \draw[edge] (1)--(2)--(3);
        \draw[edge] (1)--(l1);
        \draw[edge] (2)--(l2);
    \end{tikzpicture}}}
\]

\item \cite[Th.~5.1.1]{lyskov2023contractads} The Associative contractad  $\Ass$ is the contractad with a generator $\nu$ in the component $\Path_2$, satisfying the relations
\begin{gather*}
    \nu\circ^{\Path_3}_{\{1,2\}}\nu=\nu\circ^{\Path_3}_{\{2,3\}}\nu,
    \\
    \nu^{(12)}\circ^{\Path_3}_{\{1,2\}}\nu=\nu \circ^{\Path_3}_{\{2,3\}}\nu^{(12)},
    \\
    \nu\circ^{\Path_3}_{\{1,2\}}\nu^{(12)}=\nu^{(12)} \circ^{\Path_3}_{\{2,3\}}\nu,
    \\
    \nu\circ^{\K_3}_{\{1,2\}}\nu=\nu\circ^{\K_3}_{\{2,3\}}\nu.
\end{gather*} Let me note that, the generator $\nu$ is not symmetric, $\nu\neq\nu^{(12)}$.

\item \cite[Pr.~2.4.1]{lyskov2023contractads} The rooted spanning trees contractad $\RST$ is generated by a generator $\mu$ in the component $\Path_2$, satisfying the relations
 \begin{gather*}
     \mu \circ^{\Path_3}_{\{1,2\}} \mu= \mu \circ^{\Path_3}_{\{2,3\}} \mu
     \\
     \mu \circ^{\Path_3}_{\{1,2\}} \mu^{(12)} = \mu^{(12)} \circ^{\Path_3}_{\{2,3\}} \mu
     \\
     \mu^{(12)} \circ^{\Path_3}_{\{1,2\}} \mu=0
     \\
     \mu \circ^{\K_3}_{\{1,2\}} \mu- \mu \circ^{\K_3}_{\{2,3\}} \mu= (\mu \circ^{\K_3}_{\{1,2\}} \mu- \mu \circ^{\K_3}_{\{2,3\}} \mu)^{(23)}
 \end{gather*}
\end{itemize}
\end{example}
Let us state the main result of this section.
\begin{theorem}\label{thm::hampres}
 The Hamiltonian contractad $\Ham$ is generated by binary generators $\nu=\vcenter{\hbox{\begin{tikzpicture}[
        scale=0.4,
        vert/.style={circle,  draw=black!30!black, thick, minimum size=1mm, inner sep=2pt},
        leaf/.style={rectangle, thick, minimum size=1mm, inner sep=2pt},
        edge/.style={-,black!30!black, thick},
        ]
        \node[vert] (1) at (0,1) {\scriptsize$\space$};
        \node[leaf] (l1) at (-0.75,2) {\scriptsize$1$};
        \node[leaf] (l2) at (0.75,2) {\scriptsize$2$};
        \draw[edge] (0,0.2)--(1);
        \draw[edge] (1)--(l1);
        \draw[edge] (1)--(l2);
\end{tikzpicture}}}$ and its opposite $\nu^{(12)}=\vcenter{\hbox{\begin{tikzpicture}[
        scale=0.4,
        vert/.style={circle,  draw=black!30!black, thick, minimum size=1mm, inner sep=2pt},
        leaf/.style={rectangle, thick, minimum size=1mm, inner sep=2pt},
        edge/.style={-,black!30!black, thick},
        ]
        \node[vert] (1) at (0,1) {\scriptsize$\space$};
        \node[leaf] (l1) at (-0.75,2) {\scriptsize$2$};
        \node[leaf] (l2) at (0.75,2) {\scriptsize$1$};
        \draw[edge] (0,0.2)--(1);
        \draw[edge] (1)--(l1);
        \draw[edge] (1)--(l2);
\end{tikzpicture}}}$ satisfying the relations: for a graph $\Gr$ on $3$ vertices, we have
\begin{gather}
\label{eq::ham_assrel}\vcenter{\hbox{\begin{tikzpicture}[
        scale=0.6,
        vert/.style={circle,  draw=black!30!black, thick, minimum size=1mm},
        leaf/.style={rectangle, thick, minimum size=1mm},
        edge/.style={-,black!30!black, thick},
        ]
        \node[vert] (1) at (0,1) {\footnotesize$\space$};
        \node[leaf] (l1) at (0.75,2) {\footnotesize$v_3$};
        \node[vert] (2) at (-0.75,2) {\footnotesize$\space$};
        \node[leaf] (l2) at (0,3) {\footnotesize$v_2$};
        \node[leaf] (3) at (-1.5,3) {\footnotesize$v_1$};
        \draw[edge] (0,0)--(1);
        \draw[edge] (1)--(2)--(3);
        \draw[edge] (1)--(l1);
        \draw[edge] (2)--(l2);
\end{tikzpicture}}}=\vcenter{\hbox{\begin{tikzpicture}[
        scale=0.6,
        vert/.style={circle,  draw=black!30!black, thick, minimum size=1mm},
        leaf/.style={rectangle, thick, minimum size=1mm},
        edge/.style={-,black!30!black, thick},
        ]
        \node[vert] (1) at (0,1) {\footnotesize$\space$};
        \node[leaf] (l1) at (-0.75,2) {\footnotesize$v_1$};
        \node[vert] (2) at (0.75,2) {\footnotesize$\space$};
        \node[leaf] (l2) at (0,3) {\footnotesize$v_2$};
        \node[leaf] (3) at (1.5,3) {\footnotesize$v_3$};
        \draw[edge] (0,0)--(1);
        \draw[edge] (1)--(2)--(3);
        \draw[edge] (1)--(l1);
        \draw[edge] (2)--(l2);
    \end{tikzpicture}}}, \text{ if } (v_1,v_2),(v_2,v_3)\in E_{\Gr}
    \\
    \label{eq::ham_monrel1}\vcenter{\hbox{\begin{tikzpicture}[
        scale=0.6,
        vert/.style={circle,  draw=black!30!black, thick, minimum size=1mm},
        leaf/.style={rectangle, thick, minimum size=1mm},
        edge/.style={-,black!30!black, thick},
        ]
        \node[vert] (1) at (0,1) {\footnotesize$\space$};
        \node[leaf] (l1) at (0.75,2) {\footnotesize$v_3$};
        \node[vert] (2) at (-0.75,2) {\footnotesize$\space$};
        \node[leaf] (l2) at (0,3) {\footnotesize$v_2$};
        \node[leaf] (3) at (-1.5,3) {\footnotesize$v_1$};
        \draw[edge] (0,0)--(1);
        \draw[edge] (1)--(2)--(3);
        \draw[edge] (1)--(l1);
        \draw[edge] (2)--(l2);
\end{tikzpicture}}}=0, \text{ if } (v_1,v_2)\in E_{\Gr}\text{, but }(v_2,v_3)\not\in E_{\Gr}
\\
\label{eq::ham_monrel2}\vcenter{\hbox{\begin{tikzpicture}[
        scale=0.6,
        vert/.style={circle,  draw=black!30!black, thick, minimum size=1mm},
        leaf/.style={rectangle, thick, minimum size=1mm},
        edge/.style={-,black!30!black, thick},
        ]
        \node[vert] (1) at (0,1) {\footnotesize$\space$};
        \node[leaf] (l1) at (-0.75,2) {\footnotesize$v_1$};
        \node[vert] (2) at (0.75,2) {\footnotesize$\space$};
        \node[leaf] (l2) at (0,3) {\footnotesize$v_2$};
        \node[leaf] (3) at (1.5,3) {\footnotesize$v_3$};
        \draw[edge] (0,0)--(1);
        \draw[edge] (1)--(2)--(3);
        \draw[edge] (1)--(l1);
        \draw[edge] (2)--(l2);
    \end{tikzpicture}}}=0,\text{ if } (v_2,v_3)\in E_{\Gr}\text{, but }(v_1,v_2)\not\in E_{\Gr}
\end{gather}
\end{theorem}
\begin{proof}
Let $\nu:=\hbox{\begin{tikzpicture}[scale=0.4, edge/.style={->,> = latex, thick}]
    \fill (0,0) circle (3pt);
    \fill (1.5,0) circle (3pt);
    \node at (0,0.45) {\scriptsize$1$};
    \node at (1.5,0.45) {\scriptsize$2$};
    \draw[edge] (0,0)--(1.5,0);   \end{tikzpicture}}$ be a Hamiltonian path of $\Path_2$. In particular, we have $\nu^{(12)}=\hbox{\begin{tikzpicture}[scale=0.4, edge/.style={->,> = latex, thick}]
    \fill (0,0) circle (3pt);
    \fill (1.5,0) circle (3pt);
    \node at (0,0.45) {\scriptsize$1$};
    \node at (1.5,0.45) {\scriptsize$2$};
    \draw[edge] (1.5,0)--(0,0);   \end{tikzpicture}}$. Let us verify the relations. For simplicity, we shall use the infinitesimal notations. For the path $\Path_3$, we have
\begin{gather*}
    \nu\circ^{\Path_3}_{\{1,2\}}\nu=\begin{tikzpicture}[scale=0.5, edge/.style={->,> = latex, thick}]
    \fill (0,0) circle (3pt);
    \fill (1.5,0) circle (3pt);
    \node at (0,0.45) {\scriptsize$\{1,2\}$};
    \node at (1.5,0.45) {\scriptsize$3$};
    \draw[edge] (0,0)--(1.5,0);   \end{tikzpicture}\circ^{\Path_3}_{\{1,2\}}\hbox{\begin{tikzpicture}[scale=0.5, edge/.style={->,> = latex, thick}]
    \fill (0,0) circle (3pt);
    \fill (1.5,0) circle (3pt);
    \node at (0,0.45) {\scriptsize$1$};
    \node at (1.5,0.45) {\scriptsize$2$};
    \draw[edge] (0,0)--(1.5,0);   \end{tikzpicture}}=\hbox{\begin{tikzpicture}[scale=0.5, edge/.style={->,> = latex, thick}]
    \fill (0,0) circle (3pt);
    \fill (1.5,0) circle (3pt);
    \fill (3,0) circle (3pt);
    \node at (0,0.45) {\scriptsize$1$};
    \node at (1.5,0.45) {\scriptsize$2$};
    \node at (3,0.45) {\scriptsize$3$};
    \draw[edge] (0,0)--(1.5,0);
    \draw[edge] (1.5,0)--(3,0);\end{tikzpicture}}=\hbox{\begin{tikzpicture}[scale=0.5, edge/.style={->,> = latex, thick}]
    \fill (0,0) circle (3pt);
    \fill (1.5,0) circle (3pt);
    \node at (0,0.45) {\scriptsize$1$};
    \node at (1.5,0.45) {\scriptsize$\{2,3\}$};
    \draw[edge] (0,0)--(1.5,0);   \end{tikzpicture}}\circ^{\Path_3}_{\{2,3\}}\hbox{\begin{tikzpicture}[scale=0.5, edge/.style={->,> = latex, thick}]
    \fill (0,0) circle (3pt);
    \fill (1.5,0) circle (3pt);
    \node at (0,0.45) {\scriptsize$2$};
    \node at (1.5,0.45) {\scriptsize$3$};
    \draw[edge] (0,0)--(1.5,0);   \end{tikzpicture}}=\nu\circ^{\Path_3}_{\{2,3\}}\nu,
    \\
    \nu\circ^{\Path_3}_{\{1,2\}}\nu^{(12)}=\hbox{\begin{tikzpicture}[scale=0.5, edge/.style={->,> = latex, thick}]
    \fill (0,0) circle (3pt);
    \fill (1.5,0) circle (3pt);
    \node at (0,0.45) {\scriptsize$\{1,2\}$};
    \node at (1.5,0.45) {\scriptsize$3$};
    \draw[edge] (0,0)--(1.5,0);   \end{tikzpicture}}\circ^{\Path_3}_{\{1,2\}}\hbox{\begin{tikzpicture}[scale=0.5, edge/.style={->,> = latex, thick}]
    \fill (0,0) circle (3pt);
    \fill (1.5,0) circle (3pt);
    \node at (0,0.45) {\scriptsize$1$};
    \node at (1.5,0.45) {\scriptsize$2$};
    \draw[edge] (1.5,0)--(0,0);   \end{tikzpicture}}=0,\text{ since vertices } 1,3\text{ are not adjacent},
    \\
    \nu^{(12)}\circ^{\Path_3}_{\{1,2\}}\nu=\hbox{\begin{tikzpicture}[scale=0.5, edge/.style={->,> = latex, thick}]
    \fill (0,0) circle (3pt);
    \fill (1.5,0) circle (3pt);
    \node at (0,0.45) {\scriptsize$\{1,2\}$};
    \node at (1.5,0.45) {\scriptsize$3$};
    \draw[edge] (1.5,0)--(0,0);   \end{tikzpicture}}\circ^{\Path_3}_{\{1,2\}}\hbox{\begin{tikzpicture}[scale=0.5, edge/.style={->,> = latex, thick}]
    \fill (0,0) circle (3pt);
    \fill (1.5,0) circle (3pt);
    \node at (0,0.45) {\scriptsize$1$};
    \node at (1.5,0.45) {\scriptsize$2$};
    \draw[edge] (0,0)--(1.5,0);   \end{tikzpicture}}=0,\text{ since vertices } 1,3\text{ are not adjacent}.
\end{gather*} For the complete graph $\K_3$, we have
\[
\nu\circ^{\K_3}_{\{1,2\}}\nu=\hbox{\begin{tikzpicture}[scale=0.5, edge/.style={->,> = latex, thick}]
    \fill (0,0) circle (3pt);
    \fill (1.5,0) circle (3pt);
    \node at (0,0.45) {\scriptsize$\{1,2\}$};
    \node at (1.5,0.45) {\scriptsize$3$};
    \draw[edge] (0,0)--(1.5,0);   \end{tikzpicture}}\circ^{\K_3}_{\{1,2\}}\hbox{\begin{tikzpicture}[scale=0.5, edge/.style={->,> = latex, thick}]
    \fill (0,0) circle (3pt);
    \fill (1.5,0) circle (3pt);
    \node at (0,0.45) {\scriptsize$1$};
    \node at (1.5,0.45) {\scriptsize$2$};
    \draw[edge] (0,0)--(1.5,0);   \end{tikzpicture}}=\vcenter{\hbox{\begin{tikzpicture}[scale=0.5, edge/.style={->,> = latex, thick}]
    \fill (0,0) circle (3pt);
    \fill (2,0) circle (3pt);
    \fill (1,1.73) circle (3pt);
    \node at (-0.3,-0.25) {\scriptsize$1$};
    \node at (2.3,-0.25) {\scriptsize$3$};
    \node at (1,2.2) {\scriptsize$2$};
    \draw[dashed] (0,0)--(1.5,0); 
    \draw[edge] (0,0)--(1,1.73);
    \draw[edge] (1,1.73)--(2,0);
\end{tikzpicture}}}=\hbox{\begin{tikzpicture}[scale=0.5, edge/.style={->,> = latex, thick}]
    \fill (0,0) circle (3pt);
    \fill (1.5,0) circle (3pt);
    \node at (0,0.45) {\scriptsize$1$};
    \node at (1.5,0.45) {\scriptsize$\{2,3\}$};
    \draw[edge] (0,0)--(1.5,0);   \end{tikzpicture}}\circ^{\K_3}_{\{2,3\}}\hbox{\begin{tikzpicture}[scale=0.5, edge/.style={->,> = latex, thick}]
    \fill (0,0) circle (3pt);
    \fill (1.5,0) circle (3pt);
    \node at (0,0.45) {\scriptsize$2$};
    \node at (1.5,0.45) {\scriptsize$3$};
    \draw[edge] (0,0)--(1.5,0);   \end{tikzpicture}}=\nu\circ^{\K_3}_{\{2,3\}}\nu.
\] So, we obtain a well-defined morphism of contractads $\pi\colon\Pop\to \Ham$, where $\Pop:=\T(\nu,\nu^{(12)})/\langle \R\rangle$ is the contractad obtained from generator $\nu$ and relations above. Note that $\pi$ is surjective since contractad $\Ham$ is generated in component $\Path_2$. Also, for a graph $\Gr$, each planar binary $\Gr$-admissible tree $T$ is zero modulo relations~\eqref{eq::ham_monrel1} or \eqref{eq::ham_monrel2} or can be rewritten using the identity~\eqref{eq::ham_assrel} in the form
\begin{equation}\label{fig:prenorm}
\vcenter{\hbox{\begin{tikzpicture}[
        scale=0.7,
        vert/.style={inner sep=3pt, circle,draw, thick},
        leaf/.style={inner sep=2pt, rectangle, thick},
        edge/.style={-,black!30!black, thick},
        ]
        \node[leaf] (l1) at (-1.5,3) {\footnotesize$v_1$};
        \node[leaf] (l2) at (0,3) {\footnotesize$v_2$};
        \node[leaf] (l3) at (0.75,2) {\footnotesize$v_3$};
        \node[leaf] (lk) at (2.25,0) {\footnotesize$v_k$};
        \node[leaf] (lk+1) at (3,-1) {\footnotesize$v_{k+1}$};
        \node[leaf] (lk+2) at (3.75,-2) {\footnotesize$v_{k+2}$};
        \node[leaf] (ln) at (5.25,-4) {\footnotesize$v_{n}$};
        \node[vert] (top) at (-0.75,2) {\space};
        \node[vert] (v3) at (0,1) {\space};
        \node[circle, draw=white] (dots) at (0.75,0) {\footnotesize$\cdots$};
        \node[vert] (vk) at (1.5,-1) {\space};
        \node[vert] (vk+1) at (2.25,-2) {\space};
        \node[vert] (vk+2) at (3,-3) {\space};
        \node[circle, draw=white] (dots2) at (3.75,-4){\footnotesize$\cdots$};
        \node[vert] (bot) at (4.5,-5) {\space};
        \draw[edge] (4.5,-6)--(bot)--(dots2)--(vk+2)--(vk+1)--(vk)--(dots)--(v3)--(top)--(l1);
        \draw[edge] (top)--(l2);
        \draw[edge] (v3)--(l3);
        \draw[edge] (vk)--(lk);
        \draw[edge] (vk+1)--(lk+1);
        \draw[edge] (vk+2)--(lk+2);
        \draw[edge] (ln)--(bot);
    \end{tikzpicture}}}    
\end{equation} with an additional condition on the labeling of leaves: for each $i$, vertices $v_i$, $v_{i+1}$ are adjacent. So, the monomials of the form~\eqref{fig:prenorm} spans each component of $\Pop$. Moreover, the image of such monomial in $\Ham$ defines the Hamiltonian path $(v_1, v_2,\cdots ,v_n)$. Hence we have an one-to-one correspondence between spanning monomials~\eqref{fig:prenorm} and Hamiltonian paths, that by dimension reasons implies the isomorphism $\Pop\cong \Ham$. 
\end{proof}
As one of the consequences, we conclude the following simple upper bound on the number of Hamiltonian paths
\begin{sled}\label{cor::upperbound_paths}
For a graph $\Gr$, we have
\[
\HP(\Gr)\leq (-1)^{|V_{\Gr}|}\chi_{\Gr}(-1),
\] where $\HP(\Gr)$ is the number of oriented Hamiltonian paths, $(-1)^{|V_{\Gr}|}\chi_{\Gr}(-1)$ is the number of acyclic directions of $\Gr$. Moreover, the equality holds if and only if graph $\Gr$ is complete.
\end{sled}
\begin{proof}
Thanks to Theorem~\ref{thm::hampres} and Example~\ref{ex::presentations}, we have the inclusion of quadratic relations $\R_{\Ass}\subset \R_{\Ham}$, so there is a surjective morphism of contractads
\[
\Ass\twoheadrightarrow \Ham.
\] For $\Path_3$, the inclusion $\R_{\Ass}(\Path_3)\subsetneq \R_{\Ham}(\Path_3)$ is proper, while, for $\K_3$, we have $\R_{\Ass}(\K_3)=\R_{\Ham}(\K_3)$. So, the map $\Ass(\Gr)\twoheadrightarrow\Ham(\Gr)$ is isomorphism if and only if $\Path_3$ can not be obtained from $\Gr$ using operations of induced and contracted graphs, that is $\Gr$ is complete.
\end{proof}
\subsection{Planar equivalent trees}\label{sec::planartrees} In this subsection, we define and describe the contractad $\Plan$ of planar-equivalent binary trees. The motivation of this contractad follows from the notion of Koszul duality described in Section~\ref{sec::koszul}. Also, in Section~\ref{sec::separable_permutations} we explore the relations between components of this contractad with pattern-avoidance permutations.

For a planar binary tree $T$ with $n$ leaves, let $\iota_T\colon [n]\to \Leav(T)$ be the ordering of leaves induced from the planar structure. We say that two planar binary trees $T,T'$ are planar equivalent if the corresponding orderings coincide $\iota_T=\iota_{T'}$. Consider the contractad of planar binary graph-admissible trees $\BiPlan$ from Example~\ref{ex::freecont}. Note that the substitutions of planar binary trees are compatible with the planar equivalence relation. So, the graphical collection of equivalence classes $\Plan$, with components $\Plan(\Gr):=\BiPlan(\Gr)/\sim_{\mathrm{planar}}$, admits the contractad structure given by substitutions of trees. In particular, we have the surjective morphism of contractads
\[
\mathsf{BiPlanar}\twoheadrightarrow \Plan.
\]

There is an alternative description of this contractad. An equivalence class of a $\Gr$-admissible planar rooted tree $(T,\iota_T)$ is uniquely determined by an ordering $\iota_T$, or equivalently the tuple $(v_1,v_2,\cdots,v_n)$, where vertices are arranged concerning planar structure. So, for a graph $\Gr$ on $n$ vertices, we have a map $\Plan(\Gr)\to\mathsf{Bij}([n],V_{\Gr})$.  Note that this map is always injective, but not necessarily surjective.
\begin{example}\label{ex::separable_paths}
For the path $\Path_n$, a binary $\Path_n$-admissible planar tree is a binary planar tree $T$ with leaves $(v_1,\cdots,v_n)$ arranged in that order, such that for each internal edge, the leave of subtree $T_e$ rooted at $e$ are $v_i,v_{i+1},\cdots,v_{i+l}$, then this set form a subrange of $[n]$, i.e,  $\{v_i,v_{i+1},\cdots,v_{i+l}\}=\{l,l+1,\cdots,l+k\}$ for some $l$. Such trees are called \textit{separating trees} and were introduced in~\cite{bose1998pattern}. They showed that a permutation $(v_1,\cdots,v_n)$ is realised by a separating tree if and only if it does not contain sub-patterns $(3,1,4,2)$ and $(2,4,1,3)$. We reprove this result in Section~\ref{sec::separable_permutations}. In particular, the map $\Plan(\Path_n)\to \Sigma_n$ is not surjective for $n\geq 4$.
\end{example} In general, there is an inductive procedure to determine when a tuple comes from planar tree: a $n$-tuple $(v_1,v_2,\cdots,v_n)$ is realised by a planar $\Gr$-admissible binary tree if and only if there is an index $i$, such that vertices $v_i,v_{i+1}$ are adjacent and the tuple $(v_1,\cdots,v_{i-1},\{v_i,v_{i+1}\},\cdots,v_n)$ is realised by a planar structure of $\Gr/\{v_{i},v_{i+1}\}$. In terms of tuples, the contractad structure is given by substitution of tuples. Let me record a useful technical lemma.
\begin{lemma}\label{lemma:edges_in_tuples}
Let $\sigma=(v_1,v_2,\cdots,v_n)$ be a tuple realised by a $\Gr$-admissible planar binary tree. If $(v_i,v_{i+1})\in E_{\Gr}$, then $\sigma_i=(v_1,\cdots,\{v_i,v_{i+1}\},\cdots, v_n)$ is realised by a $(\Gr/\{v_i,v_{i+1}\})$-admissible planar binary tree..
\end{lemma}
\begin{proof} The proof is done by induction on the number of vertices in $\Gr$. The base of induction is obvious. Let $\sigma=(v_1,\cdots,v_n)\in \Plan(\Gr)$ and $i$ an index with $(v_i,v_{i+1})\in E_{\Gr}$. By the construction, we have an index $j$ such that $(v_j,v_{j+1})\in E_{\Gr}$ and $\sigma_j=(v_1,\cdots,\{v_j,v_{j+1}\},\cdots, v_n)\in \Plan(\Gr/\{v_j,v_{j+1}\})$. Suppose $j\neq i$. If $j>i+1$, then by induction step, we have \[\sigma_{ij}=(\cdots,\{v_i,v_{i+1}\},\cdots,\{v_j,v_{j+1}\},\cdots)\in \Plan(\Gr/\{\{v_i,v_{i+1}\},\{v_j,v_{j+1}\}\}\}),\] hence $\sigma_i=\sigma_{ij}\circ^{\Gr/\{\{v_i,v_{i+1}\},\{v_j,v_{j+1}\}\}}_{\{v_i,v_{i+1}\}} (v_i,v_{i+1})$. The case $i>j+1$ is done in a similarly. If $j=i+1$, then by induction step we have
\[
\sigma_{i,i+1}=(v_1,\cdots,\{v_i,v_{i+1},v_{i+2}\},\cdots,v_n)\in \Plan(\Gr/\{v_i,v_{i+1},v_{i+2}\}),
\] hence $\sigma_i=\sigma_{i,i+1}\circ^{\Gr/\{v_i,v_{i+1}\}}_{\{v_i,\{v_{i},v_{i+1}\}\}} (\{v_i,v_{i+1}\},v_{i+2})$. The case $j+1=i$ is done similarly.
\end{proof} 

Let us state the presentation of contractad $\Plan$.
\begin{theorem}\label{thm::planpres}
 The contractad of planar-equivalent trees $\Plan$ is generated by binary generators $\mu=\vcenter{\hbox{\begin{tikzpicture}[
        scale=0.4,
        vert/.style={circle,  draw=black!30!black, thick, minimum size=1mm, inner sep=2pt},
        leaf/.style={rectangle, thick, minimum size=1mm, inner sep=2pt},
        edge/.style={-,black!30!black, thick},
        ]
        \node[vert] (1) at (0,1) {\scriptsize$\space$};
        \node[leaf] (l1) at (-0.75,2) {\scriptsize$1$};
        \node[leaf] (l2) at (0.75,2) {\scriptsize$2$};
        \draw[edge] (0,0.2)--(1);
        \draw[edge] (1)--(l1);
        \draw[edge] (1)--(l2);
\end{tikzpicture}}}$ and its opposite $\mu^{(12)}=\vcenter{\hbox{\begin{tikzpicture}[
        scale=0.4,
        vert/.style={circle,  draw=black!30!black, thick, minimum size=1mm, inner sep=2pt},
        leaf/.style={rectangle, thick, minimum size=1mm, inner sep=2pt},
        edge/.style={-,black!30!black, thick},
        ]
        \node[vert] (1) at (0,1) {\scriptsize$\space$};
        \node[leaf] (l1) at (-0.75,2) {\scriptsize$2$};
        \node[leaf] (l2) at (0.75,2) {\scriptsize$1$};
        \draw[edge] (0,0.2)--(1);
        \draw[edge] (1)--(l1);
        \draw[edge] (1)--(l2);
\end{tikzpicture}}}$ satisfying the relations: for a graph $\Gr$ on $3$ vertices, we have
\begin{equation}\label{eq::planeq_rel}
\vcenter{\hbox{\begin{tikzpicture}[
        scale=0.6,
        vert/.style={circle,  draw=black!30!black, thick, minimum size=1mm},
        leaf/.style={rectangle, thick, minimum size=1mm},
        edge/.style={-,black!30!black, thick},
        ]
        \node[vert] (1) at (0,1) {\footnotesize$\space$};
        \node[leaf] (l1) at (0.75,2) {\footnotesize$v_3$};
        \node[vert] (2) at (-0.75,2) {\footnotesize$\space$};
        \node[leaf] (l2) at (0,3) {\footnotesize$v_2$};
        \node[leaf] (3) at (-1.5,3) {\footnotesize$v_1$};
        \draw[edge] (0,0)--(1);
        \draw[edge] (1)--(2)--(3);
        \draw[edge] (1)--(l1);
        \draw[edge] (2)--(l2);
\end{tikzpicture}}}=\vcenter{\hbox{\begin{tikzpicture}[
        scale=0.6,
        vert/.style={circle,  draw=black!30!black, thick, minimum size=1mm},
        leaf/.style={rectangle, thick, minimum size=1mm},
        edge/.style={-,black!30!black, thick},
        ]
        \node[vert] (1) at (0,1) {\footnotesize$\space$};
        \node[leaf] (l1) at (-0.75,2) {\footnotesize$v_1$};
        \node[vert] (2) at (0.75,2) {\footnotesize$\space$};
        \node[leaf] (l2) at (0,3) {\footnotesize$v_2$};
        \node[leaf] (3) at (1.5,3) {\footnotesize$v_3$};
        \draw[edge] (0,0)--(1);
        \draw[edge] (1)--(2)--(3);
        \draw[edge] (1)--(l1);
        \draw[edge] (2)--(l2);
    \end{tikzpicture}}}, \text{ if } (v_1,v_2),(v_2,v_3)\in E_{\Gr}
\end{equation}
\end{theorem}
\begin{proof} By the construction, we have the surjective morphism of contractads $\BiPlan\twoheadrightarrow \Plan$ that sends a planar tree to its planar structure. Note that the relation~\eqref{eq::planeq_rel} remains the orderings of leaves, hence the projection factors through $\BiPlan\twoheadrightarrow \Pop\twoheadrightarrow \Plan$, where $\Pop$ is the quadratic contractad obtained from the generators and relations above. We define the section $\iota\colon \Plan\to \Pop$ inductively as follows. For $\Path_2$, we set $\iota=\pi^{-1}\colon \Plan(\Path_2)\to \Pop(\Path_2)$. Next, given an element $\sigma=(v_1,v_2,\cdots,v_n)\in \Plan(\Gr)$, Thanks to Lemma~\ref{lemma:edges_in_tuples}, for an index $l$ with $(v_l,v_{l+1})\in E_{\Gr}$ we have $\sigma_l:=(v_1,\cdots,v_{i-1},\{v_l,v_{l+1}\},\cdots,v_n)\in \Plan(\Gr/\{v_l,v_{l+1}\})$. So, we put 
  \[
\iota(\sigma):=\iota(\sigma_l)\circ^{\Gr}_{\{v_l,v_{l+1}\}} \iota((v_l,v_{l+1})).
  \]
Let us prove by induction on the number of vertices that this correspondence is well-defined. Suppose we have another index $k$ with $(v_k,v_{k+1})$. Without loss of generality, we can assume $k>l$. If $k>l+1$, Thanks to Lemma~\ref{lemma:edges_in_tuples}, we have
\begin{multline*}
    \iota(\sigma_l)\circ^{\Gr}_{\{v_l,v_{l+1}\}} \iota((v_l,v_{l+1}))=[\iota(\sigma_{lk})\circ^{\Gr/\{v_l,v_{l+1}\}}_{\{v_k,v_{k+1}\}}\iota((v_k,v_{k+1}))]\circ^{\Gr}_{\{v_l,v_{l+1}\}} \iota((v_l,v_{l+1})),
\end{multline*} where $\sigma_{lk}=(v_1,\cdots,\{v_l,v_{l+1}\},\cdots\{v_k,v_{k+1}\},v_n)$. By induction assumption, the element $\iota(\sigma_{lk})$ is well-defined. Thanks to contractad axioms, we can change the order of contractad compositions, hence the right hand side is rewritten as follows
\begin{multline*}
[\iota(\sigma_{lk})\circ^{\Gr/\{v_l,v_{l+1}\}}_{\{v_k,v_{k+1}\}}\iota((v_k,v_{k+1}))]\circ^{\Gr}_{\{v_l,v_{l+1}\}} \iota((v_l,v_{l+1}))=\\=[\iota(\sigma_{lk})\circ^{\Gr/\{v_k,v_{k+1}\}}_{\{v_l,v_{l+1}\}}\iota((v_l,v_{l+1}))]\circ^{\Gr}_{\{v_k,v_{k+1}\}} \iota((v_k,v_{k+1}))=\iota(\sigma_k)\circ^{\Gr}_{\{v_k,v_{k+1}\}} \iota((v_k,v_{k+1})),    
\end{multline*} When $k=l+1$, by similar computations, we have
\begin{multline*}
\iota(\sigma_l)\circ^{\Gr}_{\{v_l,v_{l+1}\}} \iota((v_l,v_{l+1}))=[\iota(\sigma_{l,l+1})\circ^{\Gr/\{v_l,v_{l+1}\}}_{\{\{v_l,v_{l+1}\},v_{l+2}\}}\iota((\{v_l,v_{l+1}\},v_{l+2}))]\circ^{\Gr}_{\{v_l,v_{l+1}\}} \iota((v_l,v_{l+1}))=\\=\iota(\sigma_{l,l+1})\circ^{\Gr}_{\{v_l,v_{l+1},v_{l+2}\}}[\iota((\{v_l,v_{l+1}\},v_{l+2}))\circ^{\Gr|_{\{v_l,v_{l+1},v_{l+2}\}}}_{\{v_l,v_{l+1}\}} \iota((v_l,v_{l+1}))],
\end{multline*} where $\sigma_{l,l+1}=(v_1,\cdots,\{v_l,v_{l+1},v_{l+2}\},\cdots,v_n)$. By induction assumption, this element is well-defined. Thanks to quadratic relations in $\Pop$, the right hand side is rewritten as follows
\begin{multline*}
\iota(\sigma_{l,l+1})\circ^{\Gr}_{\{v_l,v_{l+1},v_{l+2}\}}[\iota((\{v_l,v_{l+1}\},v_{l+2}))\circ^{\Gr|_{\{v_l,v_{l+1},v_{l+2}\}}}_{\{v_l,v_{l+1}\}} \iota((v_l,v_{l+1}))]=\\=\iota(\sigma_{l,l+1})\circ^{\Gr}_{\{v_l,v_{l+1},v_{l+2}\}}[\iota((v_l,\{v_{l+1},v_{l+2}\}))\circ^{\Gr|_{\{v_l,v_{l+1},v_{l+2}\}}}_{\{v_{l+1},v_{l+2}\}} \iota((v_{l+1},v_{l+2}))]=\\=[\iota(\sigma_{l,l+1})\circ^{\Gr/\{v_l,v_{l+1}\}}_{\{v_l,\{v_{l+1},v_{l+2}\}\}}\iota((v_l,\{v_{l+1},v_{l+2}\}))]\circ^{\Gr}_{\{v_{l+1},v_{l+2}\}} \iota((v_{l+1},v_{l+2}))=\\=\iota(\sigma_{l+1})\circ^{\Gr}_{\{v_{l+1},v_{l+2}\}} \iota((v_{l+1},v_{l+2})).
\end{multline*} By computations above, we see that the map $\iota$ is well-defined. Similarly, let us show, by induction on the number of vertices, that $\iota$ is compatible with graph automorphisms:
\begin{multline*}
\iota((v_1,v_2,\cdots,v_l,v_{l+1},\cdots,v_n)^{\tau})=\iota((\tau(v_1),\tau(v_l),\tau(v_{l+1}),\cdots,\tau(v_n)))=\\=\iota((\tau(v_1),\cdots,\tau(\{v_l,v_{l+1}\}),\cdots,\tau(v_n)))\circ^{\Gr}_{\tau(\{v_l,v_{l+1}\})} \iota((\tau(v_l),\tau(v_{l+1})))=\\=\iota((v_1,\cdots,\{v_l,v_{l+1}\},\cdots,v_n))^{\tau}\circ^{\Gr}_{\tau(\{v_l,v_{l+1}\})} \iota((v_l,v_{l+1}))^{\tau}=\\=(\iota((v_1,\cdots,\{v_l,v_{l+1}\},\cdots,v_n))\circ^{\Gr}_{\{v_l,v_{l+1}\}} \iota((v_l,v_{l+1})))^{\tau}=\iota((v_1,v_2,\cdots,v_l,v_{l+1},\cdots,v_n))^{\tau}.
\end{multline*} So, $\iota$ defines a morphism of graphical collections. Although, by the construction, this correspondence is a morphism of contractads. Since compositions $\iota\circ \pi$ and $\pi\circ\iota$ are identity on the generator  $\mu$, we conclude that  $\pi$ and $\iota$ are inverse to each other. Hence $\Pop\cong \Plan$.
\end{proof}
In a dual fashion to Corollary~\ref{cor::upperbound_paths}, we have the following lower bound
\begin{sled}
For a graph $\Gr$, we have
\[
  \PE(\Gr)\geq (-1)^{|V_{\Gr}|}\chi_{\Gr}(-1).
\] where $\PE(\Gr)$ is the number of tuples $(v_1,\cdots,v_n)$ realised by $\Gr$-admissible planar binary trees. Moreover, the equality holds if and only if graph $\Gr$ is complete
\end{sled}
\section{Right modules over contractads and Hamiltonian cycles}\label{sec::modules}
In this section, we develop a theory of right modules over contractads. In particular, we show that the graphical collection of Hamiltonian cycles forms a right module over the contractad of Hamiltonian paths. Also, we construct the right modules of cyclic equivalent trees and ( cyclic ) permutations over the contractad of planar equivalent binary trees.
\subsection{Right modules over contractads} In this subsection, we provide the necessary theory of right modules over contractads.
In a natural, we define right modules over contractads as follows. 
\begin{defi}
A right module over contractad $\Pop$ is a graphical collection $\M$ endowed with a product map
\[
\gamma_{\M}\colon \M\circ\Pop\to \M
\] that is compatible with the contractad product map $\gamma_{\Pop}$.
\end{defi} Similarly to contractads, the right modules are determined by the collection of infinitesimal compositions $\circ^{\Gr}_G\colon \M(\Gr/G)\otimes\Pop(\Gr|_G)\to \M(\Gr)$. Let $\mathrm{Mod}_{\Pop}$ be the category of right $\Pop$-modules. Since the contraction product is additive on the left part, $(\Orb\oplus \Q)\circ \Pop\cong (\Orb\circ\Pop)\oplus (\Q\circ\Pop)$, the product of two $\M\oplus\N$ of two right $\Pop$-modules is also the right $\Pop$-modules. It is a good exercise in category theory for a reader to prove the following proposition.

\begin{prop}
For a contractad $\Pop$, the category $\mathrm{Mod}_{\Pop}$ of right $\Pop$-modules forms an Abelian category.
\end{prop} 
For a graphical collection $\V$, we have a natural structure of the free right $\Pop$-module on the product $\V\circ\Pop$. By "free" as usual we mean, that any morphism $\V\circ\Pop\to \M$ of modules is uniquely determined by images of generators, that is $\Hom_{\mathrm{Mod}_{\Pop}}(\V\circ\Pop,\M)\cong\Hom_{\GrCol}(\V,\M)$. We say that a right-module $\M$ is generated by generators $\V$ and relations $\Ho\subset \V\circ\Pop$, if there is exact sequence
\[
\Ho\circ\Pop\to \V\circ\Pop \to \M\to 0.
\]
\subsection{Hamiltonian cycles} The motivated example of a right module is the module of Hamiltonian cycles constructed as follows. Recall that, a \textit{Hamiltonian cycle} in a graph $\Gr$ is a cycle in $\Gr$ that visits each vertex exactly once.  Note that a Hamiltonian cycle admits two different directions.  We shall use notation $[v_1,\cdots,v_n]$ for directed Hamiltonian cycle $v_1\to \cdots\to  v_{n-1}\to v_n\to v_1$. For contractad purposes, we let the loop $[1]=\vcenter{\hbox{\begin{tikzpicture}[
        scale=0.4,
        oredge/.style={->,> = latex, thick},edge/.style={-, thick}
        ]
        \node at (-0.25,0) {\scriptsize$1$};
        \fill (0,0) circle (2pt);
        \draw[edge] (0,0) to [out=45,in=90, looseness=1.2] (1.4,0);
        \draw[oredge] (1.4,0) to [out=270,in=315, looseness=1.2] (0,0);
    \end{tikzpicture}}}$ be a unique directed Hamiltonian cycle in $\Path_1$, and $[1,2]=\vcenter{\hbox{\begin{tikzpicture}[
        scale=0.4,
        edge/.style={->,> = latex, thick},
        ]
        \node at (-0.25,0) {\scriptsize$1$};
        \node at (2.25,0) {\scriptsize$2$};
        \fill (0,0) circle (2pt);
        \fill (2,0) circle (2pt);
        \draw[edge] (0,0) to [bend right] (2,0);
        \draw[edge] (2,0) to [bend right] (0,0);
    \end{tikzpicture}}}$ be a unique directed Hamiltonian cycle in $\Path_2$.

Consider the graphical collection of Hamiltonian cycles $\CycHam$, whose component $\CycHam(\Gr)$ is the linear span of directed Hamiltonian cycles on $\Gr$. For a tube $G\subset V_{\Gr}$, we define the substitution of paths
\[
\circ^{\Gr}_G\colon \CycHam(\Gr/G)\otimes \Ham(\Gr|_G)\to \CycHam(\Gr)
\] as follows. Let $\Cyc=[v_1,\cdots,v_{k-1},\{G\},v_{k+1},\cdots,v_n]$ be a directed Hamiltonian cycle in $\Gr/G$ and $\Path=(w_1,w_2,\cdots, w_m)$ a directed Hamiltonian path in $\Gr|_G$. If $v_{k-1}$ is adjacent to $w_1$ and $w_m$ to $v_{k+1}$, we put $\Cyc\circ^{\Gr}_G \Path=[v_1,\cdots,v_{k-1},w_1,w_2,\cdots, w_m,v_{k+1},\cdots,v_n]$, otherwise we put zero. These operations endow $\CycHam$ with a structure of the right $\Ham$-module.
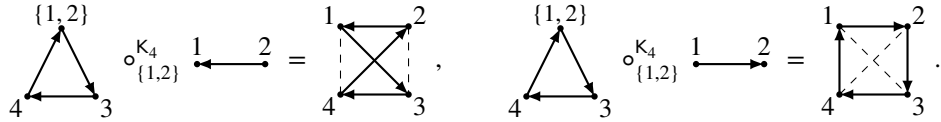
\begin{gather*}
    \vcenter{\hbox{\begin{tikzpicture}[scale=0.6, edge/.style={->,> = latex, thick}]
    \fill (0,0) circle (2pt);
    \fill (1.5,0) circle (2pt);
    \fill (0.75,1.5) circle (2pt);
    \draw[edge] (0.75,1.5)--(1.5,0);
    \draw[edge] (1.5,0)--(0,0);
    \draw[edge] (0,0)--(0.75,1.5);
    \node at (0.75,1.8) {\small$\{1,2\}$};
    \node at (1.75,-0.25) {$3$};
    \node at (-0.25,-0.25) {$4$};
    \end{tikzpicture}}}  \circ^{\K_4}_{\{1,2\}}
    \hbox{\begin{tikzpicture}[scale=0.6, edge/.style={->,> = latex, thick}]
    \fill (0,0) circle (2pt);
    \fill (1.5,0) circle (2pt);
    \node at (0,0.4) {$1$};
    \node at (1.5,0.4) {$2$};
    \draw[edge] (1.5,0)--(0,0);   \end{tikzpicture}}=\vcenter{\hbox{\begin{tikzpicture}[scale=0.6, edge/.style={->,> = latex, thick}]
    \fill (0,0) circle (2pt);
    \fill (0,1.5) circle (2pt);
    \fill (1.5,0) circle (2pt);
    \fill (1.5,1.5) circle (2pt);
    \draw[edge] (1.5,1.5)--(0,1.5);
    \draw[edge] (0,1.5)--(1.5,0);
    \draw[dashed] (1.5,1.5)--(1.5,0);
    \draw[edge] (1.5,0)--(0,0);
    \draw[dashed] (0,0)--(0,1.5);
    \draw[edge] (0,0)--(1.5,1.5);
    \node at (-0.25,1.75) {$1$};
    \node at (1.75,1.75) {$2$};
    \node at (1.75,-0.25) {$3$};
    \node at (-0.25,-0.25) {$4$};
    \end{tikzpicture}}},\qquad
    \vcenter{\hbox{\begin{tikzpicture}[scale=0.6, edge/.style={->,> = latex, thick}]
    \fill (0,0) circle (2pt);
    \fill (1.5,0) circle (2pt);
    \fill (0.75,1.5) circle (2pt);
    \draw[edge] (0.75,1.5)--(1.5,0);
    \draw[edge] (1.5,0)--(0,0);
    \draw[edge] (0,0)--(0.75,1.5);
    \node at (0.75,1.8) {\small$\{1,2\}$};
    \node at (1.75,-0.25) {$3$};
    \node at (-0.25,-0.25) {$4$};
    \end{tikzpicture}}}   \circ^{\K_4}_{\{1,2\}}
    \hbox{\begin{tikzpicture}[scale=0.6, edge/.style={->,> = latex, thick}]
    \fill (0,0) circle (2pt);
    \fill (1.5,0) circle (2pt);
    \node at (0,0.4) {$1$};
    \node at (1.5,0.4) {$2$};
    \draw[edge] (0,0)--(1.5,0);   \end{tikzpicture}}=\vcenter{\hbox{\begin{tikzpicture}[scale=0.6, edge/.style={->,> = latex, thick}]
    \fill (0,0) circle (2pt);
    \fill (0,1.5) circle (2pt);
    \fill (1.5,0) circle (2pt);
    \fill (1.5,1.5) circle (2pt);
    \draw[edge] (0,1.5)--(1.5,1.5);
    \draw[edge] (1.5,1.5)--(1.5,0);
    \draw[edge] (1.5,0)--(0,0);
    \draw[edge] (0,0)--(0,1.5);
    \draw[dashed] (0,0)--(1.5,1.5);
    \draw[dashed] (0,1.5)--(1.5,0);
    \node at (-0.25,1.75) {$1$};
    \node at (1.75,1.75) {$2$};
    \node at (1.75,-0.25) {$3$};
    \node at (-0.25,-0.25) {$4$};
    \end{tikzpicture}}}.
\end{gather*}

Let $\Path\subset \Gr$ be a directed Hamiltonian path in $\Gr$. If its endpoints are connected by an edge, then this path is extended to a unique directed Hamiltonian cycle. Such observation produces a surjective morphism of graphical collections $\Cyc\colon \Ham\twoheadrightarrow \CycHam$, that sends a path to an extended cycle, if it exists, and to zero otherwise.
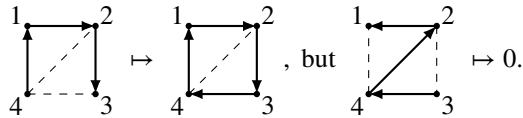
\[
    \vcenter{\hbox{\begin{tikzpicture}[scale=0.6, edge/.style={->,> = latex, thick}]
    \fill (0,0) circle (2pt);
    \fill (0,1.5) circle (2pt);
    \fill (1.5,0) circle (2pt);
    \fill (1.5,1.5) circle (2pt);
    \draw[edge] (0,1.5)--(1.5,1.5);
    \draw[edge] (1.5,1.5)--(1.5,0);
    \draw[dashed] (0,0)--(1.5,0);
    \draw[edge] (0,0)--(0,1.5);
    \draw[dashed] (0,0)--(1.5,1.5);
    \node at (-0.25,1.75) {$1$};
    \node at (1.75,1.75) {$2$};
    \node at (1.75,-0.25) {$3$};
    \node at (-0.25,-0.25) {$4$};
    \end{tikzpicture}}}\mapsto     \vcenter{\hbox{\begin{tikzpicture}[scale=0.6, edge/.style={->,> = latex, thick}]
    \fill (0,0) circle (2pt);
    \fill (0,1.5) circle (2pt);
    \fill (1.5,0) circle (2pt);
    \fill (1.5,1.5) circle (2pt);
    \draw[edge] (0,1.5)--(1.5,1.5);
    \draw[edge] (1.5,1.5)--(1.5,0);
    \draw[edge] (1.5,0)--(0,0);
    \draw[edge] (0,0)--(0,1.5);
    \draw[dashed] (0,0)--(1.5,1.5);
    \node at (-0.25,1.75) {$1$};
    \node at (1.75,1.75) {$2$};
    \node at (1.75,-0.25) {$3$};
    \node at (-0.25,-0.25) {$4$};
    \end{tikzpicture}}},\text{ but }
\vcenter{\hbox{\begin{tikzpicture}[scale=0.6, edge/.style={->,> = latex, thick}]
    \fill (0,0) circle (2pt);
    \fill (0,1.5) circle (2pt);
    \fill (1.5,0) circle (2pt);
    \fill (1.5,1.5) circle (2pt);
    \draw[edge] (1.5,1.5)--(0,1.5);
    \draw[dashed] (1.5,1.5)--(1.5,0);
    \draw[edge] (1.5,0)--(0,0);
    \draw[dashed] (0,0)--(0,1.5);
    \draw[edge] (0,0)--(1.5,1.5);
    \node at (-0.25,1.75) {$1$};
    \node at (1.75,1.75) {$2$};
    \node at (1.75,-0.25) {$3$};
    \node at (-0.25,-0.25) {$4$};
    \end{tikzpicture}}}\mapsto 0.
\]
Note that the collection of these maps defines the surjective morphism of right $\Ham$-modules
\[
\Cyc\colon \Ham \twoheadrightarrow \CycHam,
\] where $\Ham$ is seemed as a module over itself.
\begin{theorem}\label{thm::cyclespres}
The right $\Ham$-module $\CycHam$ is generated by an element $\alpha\in \CycHam(\Path_1)$, satisfying the relations
\begin{equation}\label{eq::cychampres}
 \alpha\circ^{\Path_2}_{\{1,2\}} \nu=\alpha \circ^{\Path_2}_{\{1,2\}} \nu^{(12)}.   
\end{equation}
\end{theorem}
\begin{proof}
As we have mentioned before, we have an onto morphism of right $\Ham$-modules $\Cyc\colon \Ham\to\CycHam$. Note that this map is given by the rule $\Cyc(\tau)=\alpha\circ^{\Gr}_{V_{\Gr}}\tau$, where $\alpha\in \CycHam(\Path_1)$ is the loop. Hence this right module is generated by a loop $\alpha$.  Also, Hamiltonian paths $\nu=(1,2)$ and $\nu^{(12)}=(2,1)$ determine the same Hamiltonian cycle in $\Path_2$. So, we have the onto morphism $\M\twoheadrightarrow \CycHam$, where $\M$ is the quadratic $\Ham$-module obtained from the generators and relations above. The fact that this morphism is indeed an isomorphism follows from the following lemma

\begin{claim*}
For a Hamiltonian path $\Path\in \Ham(\Gr)$ that does no lift to a Hamiltonian cycle, its image $\alpha\circ\Path$ in $\M$ is zero. Moreover, for a pair of Hamiltonian paths $\Path,\Path'$ whose corresponding Hamiltonian cycles coincide, their images also coincide in $\M$.
\end{claim*}
\begin{proof}
Let $\Path=(v_1,v_2,\cdots,v_n)$ be a Hamiltonian path of $\Gr$ that does not lift to a Hamiltonian cycle, that is $v_1$ and $v_n$ are not adjacent. So, we have
\begin{multline}\label{eq::nonextended_equal_zero}
\alpha\circ^{\Gr}_{V_{\Gr}} (v_1,v_2,\cdots,v_n)=\\=[\alpha\circ_{\{v_1,\{v_2,\cdots,v_n\}\}}^{\Gr/\{v_2,\cdots,v_n\}} (v_1,\{v_2,\cdots v_n\})]\circ^{\Gr}_{\{v_2,\cdots v_n\}}(v_2,\cdots, v_n)=^{\eqref{eq::cychampres}}\\=[\alpha\circ_{\{v_1,\{v_2,\cdots,v_n\}\}}^{\Gr/\{v_2,\cdots,v_n\}} (\{v_2,\cdots v_n\},v_1)]\circ^{\Gr}_{\{v_2,\cdots v_n\}}(v_2,\cdots, v_n)=\\=\alpha\circ^{\Gr}_{V_{\Gr}} [(\{v_2,\cdots v_n\},v_1)\circ^{\Gr}_{\{v_2,\cdots v_n\}}(v_2,\cdots, v_n))]=0.
\end{multline}
The right hand side is zero since $v_n$ and $v_1$ are not adjacent. So, we proved the first statement of the claim. Now, suppose a Hamiltonian path $\Path=(v_1,v_2,\cdots,v_n)$ lifts to a Hamiltonian cycle. Similarly to computations~\eqref{eq::nonextended_equal_zero}, we deduce the identity
\[
\alpha\circ^{\Gr}_{V_{\Gr}} (v_1,v_2,\cdots,v_n)=\alpha\circ^{\Gr}_{V_{\Gr}} (v_2,\cdots,v_n,v_1)=\cdots=\alpha\circ^{\Gr}_{V_{\Gr}} (v_n,v_1,v_2,\cdots, v_{n-1}),
\]that implies the second statement of the claim. 
\end{proof}
\end{proof}
\subsection{Cyclic equivalent trees} In a similar way, there is a right module over the contractad $\Plan$ of Planar-equivalent trees. Recall that a cyclic ordering of a $n$-element set $V$ is an orbit of $\mathbb{Z}_n$-action on the set of orderings $\mathsf{Bij}([n],V)$. We denote by $[v_1,v_2,\cdots,v_n]$ the cycle class of the ordering $(v_1,v_2,\cdots,v_n)$. We define the right $\Plan$-module of cyclic equivalent trees, denoted $\CycEq$, with component $\CycEq(\Gr)=\BiPlan(\Gr)/\sim_{\mathrm{cyclic}}$ is the set of cyclic-equivalent planar $\Gr$-admissible binary trees, and right-module structure
\[
\circ^{\Gr}_G\colon \CycEq(\Gr/G)\times \Plan(\Gr|_G)\to \CycEq(\Gr),
\] given by the grafting of planar binary trees. Similarly to the tuple description of $\Plan$, the component $\CycEq(\Gr)$ consists of cyclic vertex-orderings $[v_1,v_2,\cdots,v_n]$ that can be realised by planar $\Gr$-admissible binary trees. Sending an ordering to its cyclic class produces a morphism of $\Plan$-modules
\[
\Plan\twoheadrightarrow\CycEq,\quad (v_1,\cdots,v_n)\mapsto [v_1,\cdots,v_n].
\] Similarly to Lemma~\ref{lemma:edges_in_tuples}, we have
\begin{lemma}\label{lemma:edges_in_cyctuples}
Let $\tau=[v_1,v_2,\cdots,v_n]$ be a cyclic tuple realised by a $\Gr$-admissible planar binary tree. If $(v_i,v_{i+1})\in E_{\Gr}$, then $\tau_i=[v_1,\cdots,\{v_i,v_{i+1}\},\cdots, v_n]$ is realised by a  $(\Gr/\{v_i,v_{i+1}\})$-admissible planar binary tree.
\end{lemma}
Let us state the presentation of the module $\CycEq$.
\begin{theorem}
The right $\Plan$-module of cyclic equivalent binary trees $\CycEq$ is generated by an element $\beta \in \CycEq(\Path_1)$, satisfying the relations
\[
\beta \circ^{\Path_2}_{\{1,2\}} \mu=\beta \circ^{\Path_2}_{\{1,2\}} \mu^{(12)}.
\]
\end{theorem}
\begin{proof} Let $\beta\in \CycEq(\Path_1)$ be an unary tree. We see that this element generates $\CycEq$ as a right $\Plan$-module. Moreover, this element satisfies one quadratic relation
\[
\beta\circ^{\Path_2}_{\{1,2\}}\mu^{(12)}=\beta\circ^{\Path_2}_{\{1,2\}}\mu.
\] So, we have the surjective morphism of $\Plan$-modules $\M\twoheadrightarrow \CycEq$, where $\M$ is the quadratic $\Plan$-module obtained from the generators and relations above. We claim that this morphism is an isomorphism. We construct a section $j\colon \CycEq\to \M$ inductively by the rule
 \[
j([v_1,v_2,\cdots,v_i,v_{i+1},\cdots,v_n])=j([v_1,v_2,\cdots,\{v_i,v_{i+1}\},\cdots,v_n])\circ_{\{v_i,v_{i+1}\}}(v_i,v_{i+1}),
 \] for a pair $v_{i},v_{i+1}$ of adjacent vertices. The fact that this correspondence is well-defined and defines a morphism of right modules is analogues to ones in the proof of Theorem~\ref{thm::planpres}.
\end{proof}
\subsection{Permutations as right-module}
We describe other examples of right $\Plan$-modules arising from permutations. We consider the permutation contractad $\Perm$ with components $\Perm(V_{\Gr})=\mathsf{Bij}([n], V_{\Gr})$, where $n=|V_{\Gr}|$, and substitutions as contractad maps. Note that the contractad $\Plan$ forms a subcontractad of $\Perm$ generated in the component $\Path_2$. Let us describe $\Perm$ as a right $\Plan$-module. For a graph $\Gr$, its complement graph, denoted $\overline{\Gr}$, is the graph on the same vertex set but with the complement set of edges $E_{\overline{\Gr}}:=\overline{E_{\Gr}}$. We define a graphical collection $\overline{\Ham}$, whose component $\overline{\Ham}(\Gr):=\Ham(\overline{\Gr})$ consists of Hamiltonian paths in the complement graph $\overline{\Gr}$.
\begin{theorem}\label{thm::perm_plan} The contractad $\Perm$ is the free right $\Plan$-module generated by Hamiltonian paths in complement graphs
\[
\Perm\cong \overline{\Ham}\circ \Plan.
\] 
\end{theorem}
\begin{proof}
We have a morphism of right $\Plan$-modules induced from contractad structure on $\Perm$
\[
\pi\colon\overline{\Ham}\circ\Plan \hookrightarrow \Perm\circ\Perm\overset{\gamma_{\Perm}}{\to} \Perm.
\]  To state the isomorphism, we construct the inverse morphism $\iota\colon \Perm\to \overline{\Ham}\circ\Plan$ inductively as follows. Let $\sigma=(v_1,\cdots,v_n)\in \Perm(\Gr)$ some permutation. If there is no index $i$ with $(v_i,v_{i+1})\in E_{\Gr}$, then $\sigma$ determines a Hamiltonian path in the complement $\overline{\Gr}$, so we put $\iota(\sigma)=\sigma\in \overline{\Ham}(\Gr)$. Otherwise, we could find a pair of adjacent vertices $v_i,v_{i+1}$ and by induction assumption, we set $\iota(\sigma)=\iota(\sigma_i)\circ^{\Gr}_{\{v_i,v_{i+1}\}} (v_i,v_{i+1})$. The fact that this correspondence is well-defined and defines a morphism of right modules is analogues to ones in the proof of Theorem~\ref{thm::planpres}. By the construction, this morphism is a section of $\pi$, so $\pi\iota=\Id_{\Perm}$. Also, by the construction, the endomorphism $\iota\pi$ of $\overline{\Ham}\circ\Plan$ is identical on the generators $\overline{\Ham}$, so $\iota\pi=\Id_{\overline{\Ham}\circ\Plan}$.
\end{proof}
Also, we consider the right $\Perm$-module of cyclic equivalent permutations $\CycPerm$, with components $\CycPerm(\Gr)=\mathsf{Bij}([n], V_{\Gr})/\sim_{\mathrm{cyclic}}$. We define a graphical collection $\overline{\CycHam}$ encoding Hamiltonian cycles in complement graphs, except the convention $\overline{\CycHam}(\Path_1)=\varnothing$. Similarly to the case of permutations, we state the following isomorphism, the proof of which repeats the previous case
\begin{theorem}\label{thm::cycperm_plan} We have an isomorphism of right $\Plan$-modules
\[
\CycPerm\cong \CycEq\oplus (\overline{\CycHam}\circ \Plan).
\]
\end{theorem}

\section{Koszul contractads and modules}\label{sec::koszul}
In this section, we briefly recall the Koszul theory for contractads and extend this notion to the case of right contractad modules. We show that the contractad $\Ham$ of Hamiltonian paths is Koszul and its right module $\CycHam$ of Hamiltonian cycles is Koszul.
\subsection{Free resolutions} In this subsection, we explain the motivation of Koszul theory for contractads and its modules.

Through this section, we assume $\Pop$ to be a contractad with no extra differential graded structure. A quasi-free $\Pop$-module is a free right graded $\Pop$-module $\C_{\bullet}\circ\Pop$ with square-zero endomorphism $\partial$, $\partial^2=0$, of degree $|\partial|=-1$. Since $\C_{\bullet}\circ\Pop$ is free, the differential $\partial$ is uniquely determined by the images of generators
\[
\C_{\bullet}\to \C_{\bullet-1}\circ\Pop.
\]  A free $\Pop$-resolution of a right module $\M$ is a quasi-free module $(\C_{\bullet}\circ\Pop,\partial)$ with a surjective quasi-isomorphism $(\C_{\bullet}\circ\Pop,\partial)\overset{\simeq}{\twoheadrightarrow} (\M,0)$.

By classical homological algebra arguments, it can be shown that any two free resolutions for a module are homotopy equivalent. So, we are interested in finding a "good candidate" for free resolution, i.e., to find a minimal possible resolution.

Let $\Pop$ be a contractad with a contractad morphism $\epsilon\colon \Pop\to \mathbb{1}$, called augmentation. We let $\mathbb{1}_{\Pop}$ be the trivial right $\Pop$-module with the action induced from augmentation. We also let $\overline{\Pop}=\ker \epsilon$ to be an augmentation ideal. In this case, for a module $\M$, a "good" resolution is a \textit{minimal resolution}, that is a free resolution $\C_{\bullet}\circ\Pop$ with \textit{decomposable} differential, $\partial(\C_{\bullet})\subset \C_{\bullet-1}\circ\overline{\Pop}$.

There is a standard inductive way in homological algebra for constructing minimal models for modules. For a right $\Pop$-module $\M$, let $\V$ be a minimal generating collection, so we have a surjective morphism
\[
\V\circ \Pop\overset{\pi}{\to} \M\to 0.
\] Since $\V$ is minimal, the projection is decomposable. We let $\mathcal{C}_0=\V$. Next, we define a graphical collection $\mathcal{C}_1$ as a minimal generating set for a kernel of the projection $\pi$. In other words, $\mathcal{C}_1=\Ho$ is a minimal collection of relations of $\M$. For $i>1$, we define $\mathcal{C}_{i+1}$ inductively as the minimal generating collection for a kernel of $\C_{i}\circ \Pop\overset{\partial_{i}}{\to}\C_{i-1}\circ\Pop$. All in all, we have constructed a minimal resolution
\[
\cdots\to s^i\C_{i}\circ\Pop\overset{\partial_i}{\to}s^{i-1}\C_{i-1}\circ\Pop\to\cdots\to s\Ho\circ\Pop\to\V\circ\Pop\overset{\pi}{\to}\M\to 0,
\] where suspensions $s^i\C_i$ are placed for homological reasons. The components $\C_i$ in minimal resolution are called \textit{syzygies}. Note that the construction above implies the uniqueness of syzygies up to isomorphism.

In particular, for a trivial module $\mathbb{1}_{\Pop}$, the first terms of its minimal resolution are
\[
\cdots\to s^2\R\circ\Pop\to s\E\circ\Pop\to\Pop\to \mathbb{1}_{\Pop}\to 0,
\] where $\E$ and $\R$ are minimal generators and relations of the contractad $\Pop$. In particular, its first syzygies are $\C_0=\mathbb{1}$, $\C_1=\E, \C_2=\R$.

Unfortunately, in a general case, the description of higher syzygies is highly non-trivial task. But for particular types of contractads and their modules, called \textit{Koszul}, we can describe syzygies explicitly. The rest of this section is dedicated to these modules.
\begin{remark}
In~\cite{fresse2003koszul}, a framework of homological algebra for operadic modules was proposed. In particular, the author introduced Bar-Cobar construction for right modules over (co)operads for building free resolutions of right modules $\mathsf{B}\M\circ\Pop\to \M$ in a functorial way, but these models are too "large" to be minimal. We note to the reader that these constructions naturally extended to our case. We omit these constructions since we are interested in the constructions of minimal resolutions for  modules.
\end{remark}
\subsection{Koszul dual contractads and modules}
In this subsection, we recall the notion of a Koszul dual contractad, for details, see~\cite[Sec.~3]{lyskov2023contractads}, and extend this notion to the case of right contractad modules.

Starting from this point we are working in the category of differential graded Vector spaces with Koszul signs rules. In this category, the switching map $\tau\colon V_{\bullet}\otimes W_{\bullet}\to V_{\bullet}$ is given by $\tau(v\otimes w)=(-1)^{|v|\cdot|w|}w\otimes v$. In particular, for a pair of linear maps $f\colon V\to W$ and $g\colon W\to V$, we define $f\otimes g\colon V\otimes W\to V'\otimes W'$ as $$(f\otimes g)(v\otimes w)=(-1)^{|g|\cdot |v|}f(v)\otimes g(w).$$

\subsubsection{Suspension contractad} When we deal with homologically graded contractads and modules, it is important to make a correct definition of "suspension". Similarly to operads, we define the suspension of contractads as follows. The Suspension contractad, denoted $\Susp$, is the contractad with components $\Susp(\Gr)=\Hom(\mathsf{k}s^{\otimes V_{\Gr}},\mathsf{k}s)$, where $\mathsf{k}s$ is a linear span of a homological degree $1$ element $s$, with the contractad product given by the composition of functions,
\begin{gather*}
    \gamma^{\Gr}_I\colon \Susp(\Gr/I)\otimes\bigotimes_{G\in I} \Susp(\Gr|_G) \to \Susp(\Gr)
    \\
    \gamma^{\Gr}_I(g;f_1,f_2,...,f_k)=g\circ(f_1\otimes f_2\otimes...\otimes f_k).
\end{gather*} Note that each component $\Susp(\Gr)$ is a one-dimensional vector space concentrated in degree $(1-n)$. For each ordering $(v_1,v_2,\cdots,v_n)$ of the vertex set $V_{\Gr}$, we consider a generator $\omega_{v_1,v_2,\cdots,v_n}:=s_{v_1}\otimes s_{v_2}\otimes\cdots\otimes s_{v_n}\mapsto s$ of $\Susp(\Gr)$. For different orderings we have $\omega_{\sigma(v_1),\cdots,\sigma(v_n)}=(-1)^{|\sigma|}\omega_{v_1,\cdots,v_n}$. Also, by Koszul sign rules, we have non-trivial signs in contractad compositions. For example, we have the identity of the form
\[
\omega_{\{1,2\},3}\circ^{\Path_3}_{\{1,2\}} \omega_{1,2}=\omega_{1,2,3},\quad\text{while}\quad \omega_{1,\{2,3\}}\circ^{\Path_3}_{\{2,3\}} \omega_{2,3}=-\omega_{1,2,3}.
\] Indeed, from the rule $(f\otimes g)(a\otimes b)=(-1)^{|g||a|}f(a)\otimes g(b)$, we have
\begin{multline*}
(\omega_{1,\{2,3\}}\circ^{\Path_3}_{\{2,3\}} \omega_{2,3})(s_1\otimes s_2\otimes s_3)=\omega_{1,\{2,3\}}\circ [\Id\otimes \omega_{2,3}](s_1\otimes s_2\otimes s_3)=-\omega_{1,\{2,3\}}(s_1\otimes s_{\{2,3\}})=-s.
\end{multline*}
For a contractad $\Pop$, we define its suspension $\Susp\Pop$ as the Hadamard product $\Susp\Pop:=\Susp\underset{\mathrm{H}}{\otimes}\Pop$ (take usual tensor product componentwise $\Susp(\Gr)\otimes\Pop(\Gr)$) with obvious contractad structure. In particular, if $\M$ is a right $\Pop$-module, then its suspension $\Susp\M$ forms a right $\Susp\Pop$-module.

In a dual way, we define the desuspension cocontractad $\Susp^{-1}$ as linear dual to suspension contractad $\Susp^{-1}:=\Susp^*$. Each component of this contractad is $\Susp^{-1}(\Gr)=\Hom(\mathsf{k}[s],\mathsf{k}[s]^{\otimes V_{\Gr}})$ is a one-dimensional vector space concentrated in degree $(n-1)$ with the obvious cocontractad structure.

\subsubsection{Koszul dual contractads}
For a pair of graphical collections $\Pop,\Q$, we define its \textit{infinitesimal product} by the rule
\[
(\Pop\circ'\Q)(\Gr)=\bigoplus_{G}\Pop(\Gr/G)\otimes\Q(\Gr|_G),
\] where the sum is taken over all tubes.

Note that the infinitesimal product of $\Pop$ with itself is isomorphic to the weight 2 part of the free contractad on $\Pop$
\[
\Pop\circ'\Pop \cong \T^{(2)}(\Pop).
\] We say that a contractad $\Pop$ is quadratic if it admits a presentation with generators $\E$ and relations $\R\subset \E\circ'\E$. Note that all contractads that we described in Section~\ref{sec::examples_of_contractads}, ~\ref{sec::planartrees} are quadratic.

\begin{defi}[Koszul dual contractads]
Let $\Pop=\Pop(\E,\R)$ be a quadratic contractad, whose generators $\E$ are finite-dimensional componentwise. We define the Koszul dual contractad $\Pop^!$ as the quadratic contractad with the presentation
\begin{equation}\label{eq::Koszul::dual}
 \Pop^!=\Pop(s^{-1}\Susp^{-1}\E^*,\R^{\bot}),  
\end{equation}
where the space of relation $\R^{\bot}$ is the orthogonal with respect to the following pairing 
$$\langle-,-\rangle\colon (\E\circ'\E)\otimes (s^{-1}\Susp^{-1}\E^*\circ' s^{-1}\Susp^{-1}\E^*)\rightarrow \mathsf{k}.$$ 
Note that, the homological shift is hidden in the pairing due to appropriate suspensions.
\end{defi}
Note that the operation of taking Koszul dual contractad is an involutive
\[
(\Pop^!)^!\cong \Pop.
\] Also, we define the Koszul dual cocontractad $\Pop^{\cokoszul}$, by the rule
\[
\Pop^{\cokoszul}:=\Susp^{-1}(\Pop^{!})^*,
\] where $\Susp^{-1}$ is a desuspension cocontractad. It was shown in~\cite{lyskov2023contractads}, that the Koszul dual cocontractad $\Pop^{\cokoszul}$ is a quadratic cocontractad $\Pop^{\cokoszul}=\Q(s\E,s^2\R)$ with cogenerators $s\E$ and corelations $s^2\R$.

\begin{remark}
The suspension in the dualization~\eqref{eq::Koszul::dual} was initially suggested by Ginzburg and Kapranov in~\cite{ginzburg1994koszul}. 
This choice of suspensions is good enough for binary contractads. Indeed, if $\Pop(\E,\R)$ is a quadratic operad with $\E$ concentrated in the component $\Path_2$ that does not have extra homological grading ($\E$ is an ordinary vector space), the generators $s^{-1}\Susp^{-1}\E^*=\E^*\otimes \mathrm{sgn}_2$ of the Koszul dual contractad $\Pop^{!}$ also belongs to an ordinary vector space and all homological shifts are cancelled.
\end{remark}

\begin{example}
Let us consider several examples of Koszul dual contractads.
\begin{itemize}
    \item For the commutative contractad $\Com$, its Koszul dual, denoted $\Lie$, is the graphical analogue of the Lie operad. This contractad generated by an anti-symmetric generator $b$, $b^{(12)}=-b$, satisfying the relations
    \begin{gather}
    b \circ_{\{1,2\}}^{\Path_3} b =  b \circ_{\{2,3\}}^{\Path_3} b,
    \\
    b\circ_{\{1,2\}}^{\K_3} b + (b\circ_{\{1,2\}}^{\K_3} b)^{(123)} + (b\circ_{\{1,2\}}^{\K_3} b)^{(321)}=0.
    \end{gather}
    \item The Associative contractad is self-dual: $\Ass^!\cong \Ass$.
    \item For the rooted spanning trees contractad $\RST$, its Koszul dual $\RST^!$ is the contractad generated by a generator $\nu=\mu^!$, satisfying the relations
    \begin{gather*}
    \nu \circ^{\Path_3}_{\{1,2\}} \nu= \nu \circ^{\Path_3}_{\{2,3\}} \nu,
     \\
     \nu \circ^{\Path_3}_{\{1,2\}} \nu^{(12)} = \nu^{(12)} \circ^{\Path_3}_{\{2,3\}} \nu,
     \\
     \nu \circ^{\K_3}_{\{1,2\}} \nu= \nu \circ^{\K_3}_{\{2,3\}} \nu=(\nu \circ^{\K_3}_{\{2,3\}}\nu)^{(23)}.
    \end{gather*} At the moment, the combinatorial nature of the interpretation of this object is unclear
\end{itemize}
\end{example}
The following proposition explains the motivation of the contractad $\Plan$ constructed in Section~\ref{sec::planartrees}.
\begin{prop}\label{prop::koszuldualpaths}
The contractad of Hamiltonian paths $\Ham$ and the contractad of planar-equivalent trees $\Plan$ are Koszul dual
\begin{equation}
    \Ham^!\cong \Plan.
\end{equation}
\end{prop}
\begin{proof}
Consider the quadratic representation of $\Ham$ from Theorem~\ref{thm::hampres}. Let $\mu:=\nu^!$ a Koszul dual generator of $\Ham^!$. By direct computations, we see that the orthogonal complement $\R^{\bot}$ of relations $\R_{\Ham}$ coincides with the linear span of relations of the contractad $\Plan$ given in Theorem~\ref{thm::planpres}.
\end{proof}

\subsubsection{Koszul dual modules}   Let $\Pop=\Pop(\E,\R)$ be a quadratic contractad. A quadratic right $\Pop$-module is a right $\Pop$-module that admits presentation $\M=\langle \V,\Ho \rangle_{\Pop}$ with generators $\V$ and relations $\Ho\subset \V \circ' \E$. Similarly to contractads, for a quadratic $\M$, we define its Koszul dual (co)contractad as follows.
\begin{defi}
Let $\M=\langle \V,\Ho \rangle_{\Pop}$ be a quadratic module over quadratic contractad $\Pop$ with the space of generators $\M$, whose generators are finite-dimensional componentwise. We define the Koszul dual $\M^!$ quadratic $\Pop^!$-module with the presentation
\begin{equation}\label{eq::Koszul::dual_module}
 \M^!=\langle \Susp^{-1}\V^*,\Ho^{\bot}\rangle_{\Pop^!},  
\end{equation}
where the space of relation $\Ho^{\bot}$ is the orthogonal complement of $\Ho$ with respect to the following pairing 
$$\langle-,-\rangle\colon (\V\circ'\E)\otimes (\Susp^{-1}\V^*\circ' s^{-1}\Susp^{-1}\E^*)\rightarrow \mathsf{k}.$$
Also, we define the Koszul dual  $\Pop^{\cokoszul}$-comodule $\M^{\cokoszul}$ by the rule
\begin{equation}
\M^{\cokoszul}:=\Susp^{-1}(\M^{!})^*
\end{equation}
\end{defi}
\begin{example}
For trivial $\Pop$-module $\mathbb{1}_{\Pop}$ its Koszul dual is $\mathbb{1}_{\Pop}^!\cong \Pop^!$, where $\Pop^!$ is considered as a regular $\Pop^!$-module. In particular, its Koszul dual comodule is $\mathbb{1}_{\Pop}^{\cokoszul}\cong \Pop^{\cokoszul}$. In a dual fashion, for regular $\Pop$-module $\Pop$, its Koszul dual $\Pop^!\cong \mathbb{1}_{\Pop^!}$ is a trivial $\Pop^!$-module.    
\end{example}

Through the isomorphism $(\Pop^{!})^!\cong \Pop$, the operation of taking Koszul dual module is also involutive 
\[
(\M^!)^!\cong \M.
\] Also, since of the quadratic presentation of Koszul dual cocontractad $\Pop^{\cokoszul}=\Q(s\E,s^2\R)$, the Koszul dual comodule $\M^{\cokoszul}$ is a quadratic $\Pop^{\cokoszul}$-module with cogenerators $\V$ and corelations $s\Ho$
\[
\M^{\cokoszul}=\langle \V,s\Ho\rangle_{\Pop^{\cokoszul}}. 
\]
\begin{remark}
The choice of suspensions in the dualisation~\eqref{eq::Koszul::dual_module} is good enough for unary modules over binary contractads. Indeed, if $\Pop$ is a binary contractad with no extra homological grading, its dual $\Pop^{!}$ is also binary contractad with no homological grading. If $\M=\langle\V,\Ho\rangle_{\Pop}$ is a quadratic module over $\Pop$ with generators $\V$ concentrated in component $\Path_1$ with no extra homological grading, then the Koszul dual module $\M^{!}$ is also unary module with no extra homological grading.
\end{remark}
Similarly to Proposition~\ref{prop::koszuldualpaths}, we have
\begin{prop}\label{prop::koszuldualcycles}
The right $\Ham$-module of Hamiltonian cycles $\CycHam$ and the right $\Plan$-module of cyclic equivalent trees $\Plan$ are Koszul dual
\begin{equation}
    \CycHam^!\cong \CycEq.
\end{equation}
\end{prop}
\subsection{Koszul contractads and modules}
Let $\M=\langle \V,\Ho\rangle_{\Pop}$ be a quadratic module over quadratic contractad $\Pop(\E,\R)$. Consider the free $\Pop$-module $\M^{\cokoszul}\circ\Pop$. We define an endomorphism of this module induced by
\begin{equation}
\M^{\cokoszul}\overset{\Delta'}{\rightarrow}\M^{\cokoszul}\circ' s\E \overset{s^{-1}}{\to} \M^{\cokoszul}\circ' \E\hookrightarrow \M^{\cokoszul}\circ \Pop.   
\end{equation} The morphism $\partial_{\kappa}$ is square zero because $s^{-2}(\Delta')^2$ maps generators $\M^{\cokoszul}$ to $\M^{\cokoszul}\circ \R$ and the image of relations $\R$ with respect to $\gamma_{\Pop}$ in $\Pop$ is zero by definition.
\begin{defi}[Koszul complex]
The Koszul complex of a quadratic $\Pop$-module $\M$ is the quasi-free $\Pop$-module
\[
\mathcal{K}(\M):=(\M^{\cokoszul}\circ\Pop,\partial_{\kappa})
\]
\end{defi}
There is a morphism of $\Pop$-modules $\pi\colon \mathcal{K}(\M)\to\M$ induced by the morphism of graphical collections
\[
\M^{\cokoszul}\twoheadrightarrow s\V\overset{s^{-1}}{\to}\V\hookrightarrow \M.
\] The induced morphism $\pi$ preserves differentials, where $\M$ is viewed as a dg module with zero differential. It suffices to check that composition $\pi\partial_{\kappa}$ maps weight 2 component $(\M^{\cokoszul})^{(2)}$ to zero. Indeed, since $\M^{\cokoszul}$ is a $\Pop^{\cokoszul}$ module with corelations $s\Ho$, we have $(\M^{\cokoszul})^{(2)}=s\Ho$, and
\[
\pi(\partial_{\kappa}(s\Ho))=\gamma_{\M}(\Ho)=0.
\]
\begin{example}
In practice, we take $\M=\mathbb{1}_{\Pop}$ a trivial module. Its Koszul dual comodule is the regular module $\Pop^{\cokoszul}$. So, the Koszul projection $\mathcal{K}(\mathbb{1}_{\Pop})\to \mathbb{1}_{\Pop}$ has the form
\[
\Pop^{\cokoszul}\circ\Pop\to \mathbb{1}_{\Pop}.
\] And the projection is induced from counit $\Pop^{\cokoszul}\to \mathbb{1}$.
\end{example}
\begin{defi}[Koszul modules and contractads]\label{defi::koszul}\hfill\break
\begin{itemize}
    \item[(i)] A quadratic right $\Pop$-module $\M$ is called Koszul if the Koszul projection 
    \[
    \mathcal{K}_{\bullet}(\M)=\M^{\cokoszul}\circ\Pop \overset{\simeq}{\to} \M
    \] is a quasi-isomorphism of $\Pop$-modules.
    \item[(ii)] A quadratic contractad $\Pop$ is called Koszul if the trivial $\Pop$-module $\mathbb{1}_{\Pop}$ is Koszul. In other words, we have the quasi-isomorphism
    \[
    \Pop^{\cokoszul}\circ\Pop\overset{\simeq}{\to} \mathbb{1}_{\Pop}.
    \]
\end{itemize}
\end{defi}
\begin{example}
In~\cite{lyskov2023contractads}, the author show that the contractads $\Com,\Ass$ are Koszul, while $\RST$ is not. Also, a quadratic contractad $\Pop$ is Koszul if and only if its Koszul dual $\Pop^!$ is.
\end{example}
As an immediate consequence of Definition~\ref{defi::koszul}, we have
\begin{sled}
If $\M$ is a Koszul $\Pop$-module, then the Koszul complex $\mathcal{K}_{\bullet}(\M)$ is a minimal $\Pop$-resolution of $\M$. In particular, the syzygies of $\M$ are
\[
\C_n\cong (\M^{\cokoszul})^{(n)}\cong \Susp^{-1}((\M^!)^{(n)})^*.
\]
\end{sled} 

\begin{remark}
In the prospect of homotopical algebra, there is an alternative definition of Koszul contractads~\cite[Sec.~3]{lyskov2023contractads}, based on Bar-Cobar construction. Explicitly, the Koszul property tells us that there is a quasi-free resolution $\mathsf{\Omega}(\Pop^{\cokoszul})\overset{\simeq}{\to} \Pop$, where $\mathsf{\Omega}(\Pop^{\cokoszul})$ is a free contractad on the generators $\Pop^{\cokoszul}$ with differential coming from the cocontractad structure on $\Pop^{\cokoszul}$. In particular, this definition allows us to verify the Koszul property for contractads using the technique of Gr\"obner bases~\cite[Sec.~4]{lyskov2023contractads}.
\end{remark}
\subsection{Koszulity of Hamiltonian paths and cycles} Let us state and prove the main results of this paper
\begin{theorem}\label{thm::koszul_paths}
The contractad of Hamiltonian paths $\Ham$ is Koszul.
\end{theorem}
\begin{proof}
Let us first describe the Koszul dual cocontractad $\Ham^{\cokoszul}$. Thanks to Proposition~\ref{prop::koszuldualpaths} and the definition of Koszul dual cocontractads, we have $\Ham^{\cokoszul}\cong \Susp^{-1}\Plan^*$. For an $n$-tuple $\sigma=(v_1,v_2,\cdots,v_n)\in \Plan(\Gr)$, we let $\sigma^{\vee}:=\eta_{v_1,\cdots,v_n}\otimes \sigma^*\in \Susp^{-1}\Plan^*(\Gr)$, where $\eta_{v_1,\cdots,v_n}=(s\mapsto s_{v_1}\otimes\cdots\otimes s_{v_n})$. Consider the morphism $\Delta'\colon \Ham^{\cokoszul}\to \Ham^{\cokoszul}\circ' s\E$, needed for the construction of Koszul differential. Since $\Plan$ is binary,  $\Delta'$ is given by the collection of cocontractad maps $\Delta^{\Gr}_e\colon \Ham^{\cokoszul}(\Gr)\to \Ham^{\cokoszul}(\Gr/e)\otimes \Ham^{\cokoszul}(\Gr|_e)$.  For an edge $e=(v,w)\in E_{\Gr}$, we have
\begin{equation}\label{eq::coprod_with_signs}
\Delta^{\Gr}_{e}(v_1,v_2,\cdots,v_n)^{\vee}=\begin{cases}
(-1)^{i-1}(v_1,\cdots,\{v_{i},v_{i+1}\},\cdots,v_n)^{\vee}\otimes (v_i,v_{i+1})^{\vee}\text{, if } e=(v_i,v_{i+1});
\\
0,\textit{otherwise.}
\end{cases}    
\end{equation} The sign $(-1)^{i-1}$ comes from the cocontractad maps in $\Susp^{-1}$:
\begin{equation}\label{eq::signs_susp}
 \Delta^{\Gr}_{\{v_i,v_{i+1}\}}(\eta_{v_1,\cdots,v_i,v_{i+1},\cdots,v_n})=(-1)^{i-1}\eta_{v_1,\cdots,\{v_i,v_{i+1}\},\cdots,v_n}\otimes\eta_{v_i,v_{i+1}}.   
\end{equation}

Now consider the Koszul complex $\mathcal{K}_{\bullet}(\Ham)=(\Susp^{-1}\Plan^*\circ \Ham,d_{\kappa})$ for the contractad $\Ham$. For a graph $\Gr$, the vector space $(\Susp^{-1}\Plan^*\circ\Ham)(\Gr)$ has a basis consisting of elements
\[((\{G_1\},\{G_2\},\cdots,\{G_k\})^{\vee};\Path^{(1)},\Path^{(2)},\cdots,\Path^{(k)}),\] where $(\{G_1\},\{G_2\},\cdots,\{G_k\})\in \Plan(\Gr/I)$ and $\Path^{(i)}$ are Hamiltonian paths in induced subgraphs $\Gr|_{G_i}$. On the one hand, each Hamiltonian path in a graph determines a tuple that is realised by a graph-admissible planar binary tree, so each basis element determines a realised tuple with additional partition into subpaths
\[((\{G_1\},\{G_2\},\cdots,\{G_k\})^{\vee};\Path^{(1)},\Path^{(2)},\cdots,\Path^{(k)})\mapsto (\Path^{(1)}|\Path^{(2)}|\cdots|\Path^{(k)}).\] On the other hand, thanks to Lemma~\ref{lemma:edges_in_tuples}, for each element $\sigma=(v_1,\cdots,v_n)\in \Plan(\Gr)$ and the partition into subpaths $(v_1,\cdots,v_{i_1}|v_{i_1+1},\cdots,v_{i_2}|\cdots|v_{i_{k-1}+1},\cdots,v_n)$, the contracted tuple $(\{v_1,\cdots,v_{i_1}\},\{v_{i_1+1},\cdots,v_{i_2}\},\cdots,\{v_{i_{k-1}+1},\cdots,v_n\})$ is also realised by a $\Gr/I$-admissible tree, where $I$ is the corresponding partition. So, for each component $\Gr$, the vector space $\mathcal{K}_{\bullet}(\Ham)(\Gr)$ is generated by multipartite tuples $(\Path^{(1)}|\cdots|\Path^{(k)})$ consisting of directed subpaths $\Path^{(i)}\subset \Gr$, with $\sigma=\Path^{(1)}\cdot\Path^{(2)}\cdots \Path^{(k)}\in \Plan(\Gr)$. Also, the homological degree of $(\Path^{(1)}|\cdots|\Path^{(k)})$ is $k-1$.

In these notations, thanks to equation~\eqref{eq::coprod_with_signs}, the Koszul differential $\partial_{\kappa}$ acts by the formula
\[
\partial_{\kappa}((\Path^{(1)}|\cdots|\Path^{(k)}))=\sum_{l} (-1)^{l-1} (\Path^{(1)}|\cdots|\Path^{(l)}\cdot \Path^{(l+1)}|\cdots|\Path^{(k)}),
\] where the sum is taken over all $l$ with $\Path_l\cdot \Path_{l+1}$ forms a subpath in $\Gr$. From this description, we see that the complex $\mathcal{K}_{\bullet}(\Ham)(\Gr)$ splits into a direct sum of dg complexes
\begin{equation}\label{eq::splits_paths}
\mathcal{K}_{\bullet}(\Ham)(\Gr)=\bigoplus_{\sigma\in \Plan(\Gr)} \mathcal{K}_{\bullet}(\sigma),   
\end{equation} where $\mathcal{K}_{\bullet}(\sigma)$ is a subcomplex generated by tuples $(\Path^{(1)}|\cdots|\Path^{(k)})$ with $\Path^{(1)}\cdot\Path^{(2)}\cdots \Path^{(k)}=\sigma$.

For an element $\sigma=(v_1,v_2,\cdots,v_n)\in \Plan(\Gr)$, let $\mathsf{E}(\sigma)$, be a set of indices $1\leq i\leq n-1$ such that vertices $v_i,v_{i+1}$ are adjacent. For each ordered partition $I$ of $\sigma$ into subpaths $(v_1,\cdots,v_{i_1}|v_{i_1+1},\cdots,v_{i_2}|\cdots|v_{i_k+1},\cdots,v_{n})$ one can associate the subset $S(I)$ of $\mathsf{E}(\sigma)$ by the rule
\[
S(I)=([1,i_1]\setminus \{i_1\})\cup ([i_1+1,i_2]\setminus \{i_2\})\cup\cdots\cup ([i_{k}+1,n]\setminus \{n\}).
\] Note that this map defines a one-to-one correspondence between partitions of $\sigma$ into subpaths and subsets (including empty) of $\mathsf{E}(\sigma)$. Moreover, thanks to the description of $\partial_{\kappa}$, this correspondence defines an isomorphism of dg complexes
\begin{equation}\label{eq::isomorphism_full_complex}
\mathcal{K}_{\bullet}(\sigma)\cong \Cyc^{(n-1)-\bullet}_+[\Delta_{\mathsf{E}(\sigma)}], 
\end{equation} where $\Delta_{\mathsf{E}(\sigma)}$ is the full simplicial complex on the vertex set $\mathsf{E}(\sigma)$ and $\Cyc^{\bullet}_+[\Delta_{\mathsf{E}(\sigma)}]$ is dual augmeneted dg simplicial complex. Recall that, for a set $A$, the reduced cohomology groups of a full simplex $\Delta_{A}$ are zero unless $A$ is empty. Since $\sigma\in \Plan(\Gr)$, we have $\mathsf{E}(\sigma)$ is non-empty unless $\Gr=\Path_1$ and $\sigma=(1)$. Hence, thanks to splitting~\eqref{eq::splits_paths} and isomorphism~\eqref{eq::isomorphism_full_complex}, the homology of $\mathcal{K}_{\bullet}(\Ham)(\Gr)$ are trivial unless $\Gr=\Path_1$. So, we deduce that the Koszul complex defines a minimal resolution of the trivial $\Ham$-module
\[
\Ham^{\cokoszul}\circ \Ham \overset{\simeq}{\to} \mathbb{1}_{\Ham}.
\]
\end{proof}
By similar methods, we prove the following result
\begin{theorem}\label{thm::koszul_cycles}
The right $\Ham$-module of Hamiltonian cycles $\CycHam$ is Koszul.
\end{theorem}
\begin{proof}
Consider the Koszul dual right $\Ham^{\cokoszul}$-comodule $\CycHam^{\cokoszul}$. Thanks to Proposition~\ref{prop::koszuldualcycles}, we have $\CycHam^{\cokoszul}\cong \Susp^{-1}\CycEq^*$. For a cyclic tuple $\tau=[v_1,v_2,\cdots,v_n]\in \CycEq(\Gr)$, we let $\tau^{\vee}:=\eta_{v_1,\cdots,v_n}\otimes \tau^*\in \Susp^{-1}\CycEq^*(\Gr)$. In particular, by Koszul sign rules, we have $[v_{1+i},v_{2+i},\cdots,v_{n+i}]^{\vee}=(-1)^{(n-1)i}[v_1,\cdots,v_n]^{\vee}$. Similarly to formula~\eqref{eq::coprod_with_signs}, the cocontractad maps $\Delta^{\Gr}_{e}\colon \CycHam^{\cokoszul}(\Gr)\to \CycHam^{\cokoszul}(\Gr/e)\otimes \Ham^{\cokoszul}(\Gr|_e)$ are given by the formula
\begin{equation}\label{eq::coprod_with_signs_cyc}
\Delta^{\Gr}_{e}[v_1,v_2,\cdots,v_n]^{\vee}=\begin{cases}
(-1)^{i-1}[v_1,\cdots,\{v_{i},v_{i+1}\},\cdots,v_n]^{\vee}\otimes (v_i,v_{i+1})^{\vee}\text{, if } e=(v_i,v_{i+1});
\\
(-1)^{n-1}[v_2,\cdots,\{v_n,v_1\}]^{\vee}\otimes (v_n,v_{1})^{\vee}\text{, if } e=(v_n,v_{1});
\\
0,\textit{otherwise.}
\end{cases}    
\end{equation} Next, for a cycle $n$-tuple $\tau$ we define its cycle partition as the choice of its representative $\sigma=(v_1,\cdots,v_{n})$, $[\sigma]=\tau$, and its ordered partition module the cycle action $(\Path^{(1)}|\Path^{(2)}|\cdots|\Path^{(k)})=(\Path^{(1+i)}|\Path^{(2+i)}|\cdots|\Path^{(k+i)})$. We shall use the notation $|\Path^{(1)}|\Path^{(2)}|\cdots|\Path^{(k)}|$ for a cycle partition of $\tau$. In particular, each representative of $\tau$ defines a unique cycle partition. Similarly to the proof of Theorem~\ref{thm::koszul_paths}, we see that the monomial basis of component $(\Susp^{-1}\CycEq^*\circ\Ham)(\Gr)$ can be identified with cycle partitions $|\Path^{(1)}|\cdots|\Path^{(k)}|$ of cycle tuples $\tau\in\CycEq(\Gr)$ realised by binary trees into subpaths $\Path^{(i)}\subset \Gr$. The degree of partition  $|\Path^{(1)}|\cdots|\Path^{(k)}|$ is $(k-1)$ and we have sign rule $|\Path^{(1)}|\cdots|\Path^{(k)}|=(-1)^{(k-1)i}|\Path^{(1+i)}|\cdots|\Path^{(k+i)}|$.

Similarly to decomposition~\eqref{eq::splits_paths}, for each graph $\Gr$, we have the splitting of the Koszul complex
\begin{equation}\label{eq::splits_cycles}
\mathcal{K}(\CycHam)(\Gr)=(\Susp^{-1}\CycEq^*\circ\Ham)(\Gr)\cong \bigoplus_{\tau\in \CycEq(\Gr)} \mathcal{CK}(\tau),    
\end{equation} where $\mathcal{CK}(\tau)$ is the dg complex generated by cycle partitions $|\Path^{(1)}|\Path^{(2)}|\cdots|\Path^{(k)}|$ of $\tau$ into subpaths with the differential
\[
\partial_{\kappa}(|\Path^{(1)}|\cdots|\Path^{(k)}|)=\sum_{l} (-1)^{l-1} |\Path^{(1)}|\cdots|\Path^{(l)}\cdot \Path^{(l+1)}|\cdots|\Path^{(k)}|,
\] where the sum is taken over all $l\in \mathbb{Z}_n$ with $\Path_l\cdot \Path_{l+1}$ forms a subpath in $\Gr$.

For an element $\tau=[v_1,v_2,\cdots,v_n]\in \Plan(\Gr)$, let $\mathsf{E}(\tau)$, be a set of indices $ i\in \mathbb{Z}_n$ such that vertices $v_i,v_{i+1}$ are adjacent. Consider the cyclic partition of $\tau$ and its ordered representative $I=(v_{1+j},\cdots,v_{i_1}|\cdots|v_{i_{k}+1},\cdots,v_{n+j})$. Recall from the proof of Theorem~\ref{thm::koszul_paths} that we can associate a unique subset $S(I)\subset \mathbb{Z}_n\setminus \{n+j\}$. We see that this subset does no depend on the choice of representative of cyclic partition, so there is a well-defined map from cyclic partitions to proper subsets of $\mathbb{Z}_n$. Although, this map defines a bijection between cyclic partitions of $\tau$ into subpaths and subsets of $\mathsf{E}(\tau)\setminus \{\mathbb{Z}_n\}$. Moreover, from the description of Koszul differential $\partial_{\kappa}$, we see an isomorphism of dg complexes
\begin{equation}
\mathcal{CK}_{\bullet}(\tau)\cong \Cyc^{(n-1)-\bullet}_+[\partial\Delta_{\mathbb{Z}_n}\cap \Delta_{\mathsf{E}(\tau)}],
\end{equation} where $\partial\Delta_{\mathbb{Z}_n}:=\Delta_{\mathbb{Z}_n}\setminus \{\mathbb{Z}_n\}$ is the simplicial sphere. Recall that the reduced cohomology of a sphere are one-dimensional and concentrated in the top degree. When $\tau=[v_1,\cdots,v_n]$ determines a Hamiltonian cycle, we have $\mathsf{E}(\tau)=\mathbb{Z}_n$, and hence the homology of $\mathcal{CK}(\tau)$ are one-dimensional and concentrated in degree zero. Otherwise, $\mathsf{E}(\tau)$ is a proper subset of $\mathbb{Z}_n$, so $\partial\Delta_{\mathbb{Z}_n}\cap \Delta_{\mathsf{E}(\tau)}=\Delta_{\mathsf{E}(\tau)}$, hence complex $\mathcal{CK}(\tau)$ is contractible. All in all, the homology of $\mathcal{K}(\CycHam)(\Gr)$ are concentrated in degree zero and the dimension of the zero homology group is equal to the number of Hamiltonian cycles. So, we deduce that the Koszul complex defines a minimal resolution
\[
\CycHam^{\cokoszul}\circ \Ham \overset{\simeq}{\to} \CycHam.
\] 
\end{proof} 
\section{Generating functions for contractads}\label{sec::applications}
In this section, we obtain several numerical applications to Hamiltonian paths and cycles on graphs from the Koszul property for $\Ham$ and $\CycHam$. At the end, we compute generating series for certain patterns avoiding permutations.
\subsection{Graphic functions} 
We briefly recall the notion of graphic functions, for details see~\cite[Sec.~2]{khoroshkin2024hilbert}.
A \textit{graphic function} is a function on the groupoid of connected graphs $f\colon \mathsf{CGr}\to \mathsf{k}$ with values in a commutative ring $\mathsf{k}$  which is constant on isomorphism classes. We have a natural product of graphic functions, denoted $*$,  arising from the contraction product
\[
(\phi * \psi)(\Gr):= \sum_{I \vdash \Gr} \phi(\Gr/I)\prod_{G\in I}\psi(\Gr|_G), 
\] This operation is associative, linear on the left part, and the graphic function $\varepsilon(\Gr)=\delta_{\Gr, \Path_1}$ is a unit for this operation. For a graded graphical collection $\Orb$, we define its Euler graphic function $\chi(\Orb)$ by the rule:
\[
\chi(\Orb)(\Gr):=\sum_i (-1)^i \dim \Orb_i(\Gr).
\] In particular, when $\Orb_{\bullet}$ is a dg graphical collection, the Euler graphic function is a homological invariant, $\chi(\Orb)=\chi(H_{\bullet}(\Orb))$. By the construction of $*$-product, Euler graphic functions preserve the contracted product
\[
\chi(\Pop\circ\Orb)=\chi(\Pop)*\chi(\Orb).
\] Also, we consider the involution $\omega$ of the space of graphic functions given by the rule
\[
\omega(f)(\Gr):=(-1)^{|V_{\Gr}|-1}f(\Gr).
\] In particular, we have $\chi(\Susp^{-1}\Pop)=\omega(\chi(\Pop))$, where $\Susp^{-1}$ is a desuspension.
\begin{prop}\label{prop::grfun_koszul}
If $\Pop$ is a Koszul contractad, then we have
\begin{equation}
\omega(\chi(\Pop^!))*\chi(\Pop)=\omega(\chi(\Pop))*\chi(\Pop^!)=\varepsilon.   
\end{equation} If $\M$ is a Koszul right $\Pop$-module, we have
\begin{equation}
 \chi(\M)=\omega(\chi(\M^!))*\chi(\Pop).   
\end{equation}
\end{prop}
\begin{proof}
By the definition of a Koszul contractad, we have a quasi-isomorphism
\[
\Pop^{\cokoszul}\circ \Pop\overset{\simeq}{\rightarrow}\mathbb{1},
\] so it implies the equality of Euler characteristics \[\chi(\Pop^{\cokoszul})*\chi( \Pop)=\chi(\Pop^{\cokoszul}\circ \Pop)=\chi(\mathbb{1})=\varepsilon.\] Since of effect of the desuspension $\Pop^!=\Susp^{-1}(\Pop^{\cokoszul})^*$, we have $\chi(\Pop^{\cokoszul})=\omega(\chi(\Pop^!))$. In a similar way, the formula $\chi(\M)=\omega(\chi(\M^!))*\chi(\Pop)$ follows from the quasi-isomorphism
\[
\M^{\cokoszul}\circ\Pop\overset{\simeq}{\rightarrow} \M.
\]
\end{proof}
For simplicity of notations, we let
\begin{gather}
  \HP:=\chi(\Ham),\quad \overline{\HP}:=\chi(\overline{\Ham}),\quad \PE:=\chi(\Plan), \quad \mathrm{P}:=\chi(\Perm)
  \\
  \HC:=\chi(\CycHam),\quad \overline{\HC}:=\chi(\overline{\CycHam})\quad \CE:=\chi(\CycEq), \quad \mathrm{C}:=\chi(\mathsf{CycPerm})
\end{gather}
Explicitly, $\HP(\Gr)$ and $\HC(\Gr)$ are the number of directed Hamiltonian paths and cycles in $\Gr$; $\overline{\HP}(\Gr)=\HP(\overline{\Gr})$ and $\overline{\HC}(\Gr)=\HC(\overline{\Gr})$,\ are the number of directed Hamiltonian paths and cycles in complement graph\footnote{except the convention $\overline{\HC}(\Path_1)=0$}; $\PE(\Gr)$ and $\CE(\Gr)$ are the number of ordered $(v_1,v_2,\cdots,v_n)$ and cycle $[v_1,v_2,\cdots,v_n]$ tuples realised by $\Gr$-admissible planar binary trees; $\mathrm{P}(\Gr)=|V_{\Gr}|!$, $\mathrm{C}(\Gr)=(|V_{\Gr}|-1)!$ are the number of ordered and cyclic tuples.

Thanks to Theorem~\ref{thm::perm_plan} and Theorem~\ref{thm::cycperm_plan}, we have
\begin{equation}\label{eq::p_pe}
\mathrm{P}=\overline{\HP}*\PE, \quad \mathrm{C}=\CE+\overline{\HC}*\PE.
\end{equation}

Thanks to Theorem~\ref{thm::koszul_paths}, Theorem~\ref{thm::koszul_cycles} and Proposition~\ref{prop::grfun_koszul}, we have
\begin{gather}
\omega(\PE)*\HP=\omega(\HP)*\PE=\epsilon,\label{eq::hp_inverse}
\\
\HC=\omega(\CE)*\HP.\label{eq::hc_inverse}
\end{gather}
Combining equations above with equation~\eqref{eq::p_pe}, we conclude the following identities. 
\begin{theorem}\label{thm::reccurences_paths_and_cycles}
We have
\[
\omega(\mathrm{P})*\HP=\omega(\overline{\HP}),\quad  \omega(\mathrm{C})*\HP=\HC+\omega(\overline{\HC}).
\]  In other words, for a connected graph $\Gr$, we have
\begin{equation}
   \sum_{I\vdash \Gr}(-1)^{|I|-1}|I|!\prod_{G\in I}\HP(\Gr|_G)=(-1)^{|V_{\Gr}|-1}\HP(\overline{\Gr}),
\end{equation} and if graph has at least one edge, we have
\begin{equation}
   \sum_{I\vdash \Gr} (-1)^{|I|-1}(|I|-1)!\prod_{G\in I} \HP(\Gr|_G)=\HC(\Gr)+(-1)^{|V_{\Gr}|-1}\HC(\overline{\Gr}).  
\end{equation}
\end{theorem}
\subsection{Hamiltonian paths and cycles in Complete Multipartite graphs.}
Unfortunately, for an arbitrary graph, the formulas from Theorem~\ref{thm::reccurences_paths_and_cycles} do not provide efficient ways for computing Hamiltonian paths/cycles. But if we restrict ourselves to the family of complete multipartite graphs, we can package numbers of Hamiltonian paths/cycles for all such graphs in one generating function as follows. 

Recall from page~\pageref{typesofgraphs} that, for a partition $\lambda=(\lambda_1\geq \lambda_2\geq\cdots\geq\lambda_k)$, the complete multipartite graph $\K_{\lambda}$ is the graph consisting of blocks of vertices of sizes $\lambda_1,\lambda_2,\cdots,\lambda_k$, such that two vertices are adjacent if and only if they belong to different blocks. Consider the ring of symmetric functions $\Lambda_{\mathbb{Q}}=\underset{\to}{\lim}\mathbb{Q}[x_1,\cdots,x_n]^{\Sigma_n}$ with rational coefficients. Recall that this ring has the basis consisting of monomial symmetric functions $m_{\lambda}=\mathsf{Sym}(x^{\lambda})$ indexed by Young diagrams. For a Young diagram $\lambda$, let $|\lambda|$ be a weight, $l(\lambda)$ a length, and $\lambda!=\lambda_1!\cdot\lambda_2!\cdots\cdot\lambda_k!$.

For a graphic function $f$, we define its Young generating function, by the rule
\begin{gather*}
    F_{\mathsf{Y}}^{(0)}(f)=\sum_{l(\lambda)\geq 2} f(\K_{\lambda})\frac{m_{\lambda}}{\lambda!}
    \\
    F^{(n)}_{\mathsf{Y}}(f)=\sum_{|\lambda|\geq 0} f(\K_{(1^n)\cup\lambda})\frac{m_{\lambda}}{\lambda!}\text{, for }n\geq 1
    \\
    F_{\mathsf{Y}}(f)(z)=\sum_{n\geq 0} F_{\mathsf{Y}}^{(n)}(f)\frac{z^n}{n!}=\sum_{l(\lambda)\geq 2} f(\K_{\lambda})\frac{m_{\lambda}}{\lambda!}+\sum_{n\geq 1,|\lambda|\geq 0} f(\K_{(1^n)\cup\lambda})\frac{m_{\lambda}}{\lambda!}\frac{z^n}{n!}.
\end{gather*}
\begin{example}
Let us compute the Young generating function for $\mathrm{P}=\chi(\Perm)$, $\mathrm{P}(\K_{\lambda})=|\lambda|!$. Let $p_n=m_{(n)}=\sum_{i\geq 1} x_i^n$ be the $n$-th power symmetric function. The zero term $F_{\mathsf{Y}}^{(0)}(\PE)$ has the form
\begin{gather*}
F_{\mathsf{Y}}^{(0)}(\mathrm{P})=\sum_{l(\lambda)\geq 2} |\lambda|!\frac{m_{\lambda}}{\lambda!}=\sum_{|\lambda|\geq 0} |\lambda|!\frac{m_{\lambda}}{\lambda!}-(1+\sum_{n\geq 1}\frac{p_n}{n!})=\sum_{n\geq 0} p_1^n-(1+\sum_{n\geq 1}p_n),
\end{gather*}
where the second identity follows from the formula $p_1^n=\sum_{\lambda\vdash n}\frac{n!}{\lambda!}m_{\lambda}$. Similarly, for higher terms, we have
\[
F_{\mathsf{Y}}^{(n)}(\mathrm{P})=\sum_{|\lambda|\geq 0} (n+|\lambda|!)\frac{m_{\lambda}}{\lambda!}=\sum_{k\geq 0} \binom{n+k}{n}\sum_{|\lambda|=k} \frac{k!}{\lambda!}m_{\lambda}=\sum_{k\geq 0} \binom{n+k}{n}p_1^k.
\]
All in all, we get
\begin{multline}
  F_{\mathsf{Y}}(\mathrm{P})=\sum_{k,n\geq 0} \binom{n+k}{k}z^np_1^k -(1+\sum_{n\geq 1}p_n)=\sum_{k\geq 0} (z+p_1)^k -(1+\sum_{n\geq 1}p_n)=\\=\frac{1}{1-(z+p_1)}-(1+\sum_{n\geq 1}p_n).    
\end{multline}
By similar computations, for the graphic function of cyclic equivalent trees $\mathrm{C}=\chi(\CycPerm)$, using the formula $\mathrm{C}(\K_{\lambda})=(|\lambda|-1)!$, we deduce
\begin{equation}\label{eq::young_compatibility}
F_{\mathsf{Y}}(\mathrm{C})=-\log(1-(p_1+z))-\sum_{n\geq 1} \frac{p_n}{n}.
\end{equation}
\end{example}
It was shown in~\cite[Th.~2.2.11]{khoroshkin2024hilbert}, that Young generating functions are compatible with the product of graphic functions as follows 
\begin{equation}
F_{\mathsf{Y}}(f*g)=F_{\mathsf{Y}}(f)(F_{\mathsf{Y}}(g))   
\end{equation}
Also, for involution $\omega$, we have
\begin{equation}\label{eq::involution}
F_{\mathsf{Y}}(\omega(f))(\underline{x},z)=-F_{\mathsf{Y}}(f)(-\underline{x},-z).  
\end{equation}
In~\cite{klarner1969number}, Klarner computed the generating function that counts the numbers of Hamiltonian paths in complete multipartite graphs. Now, we reprove and generalize his result.
\begin{theorem}[Hamiltonian paths and cycles in complete multipartite graphs]\label{thm::hp_hc_young_generating}
\hfill\break
\begin{itemize}
    \item[(i)]
The generating function of Hamiltonian paths in complete multipartite graphs\footnote{we use convention $\HP(\K_{(0)})=1$ for the empty graph $\K_{(0)}=\varnothing$} is given by the formula
\begin{equation}\label{eq::ham_young_generating}
  1+p_1+F_{\mathsf{Y}}(\HP)=\sum_{n\geq0, |\lambda|\geq 0} \HP(\K_{(1^n)\cup\lambda})\frac{z^n}{n!}\frac{m_{\lambda}}{\lambda!}=\frac{1}{1-(z +\sum_{n\geq 1} (-1)^{n-1}p_n)}.  
\end{equation}
\item[(ii)] The generating function of Hamiltonian cycles in complete multipartite graphs is given by the formula
\begin{equation}\label{eq::hamcyc_young_generating}
F_{\mathsf{Y}}(\HC)=\sum_{n+l(\lambda)\geq 2} \HC(\K_{(1^n)\cup\lambda})\frac{z^n}{n!}\frac{m_{\lambda}}{\lambda!}=-\log(1-(z+\sum_{n\geq 1}(-1)^{n-1}p_n))+\sum_{n\geq 1}(-1)^n\frac{p_n}{n}.  
\end{equation}
\end{itemize}
\end{theorem}
\begin{proof}
For a connected complete multipartite graph $\K_{\lambda}$, its complement graph $\overline{\K_{\lambda}}$ is a disjoint union of complete graphs $\overline{\K_{\lambda}}=\sqcup_i \K_{\lambda_i}$, so the complement graph of $\K_{\lambda}$ have no Hamiltonian paths/cycles. Thanks to Theorem~\ref{thm::reccurences_paths_and_cycles} and equation~\eqref{eq::young_compatibility}, we have
\begin{gather}
    \label{eq::func_equation} F_{\mathsf{Y}}(\omega(\mathrm{P}))(F_{\mathsf{Y}}(\HP)(z))=F_{\mathsf{Y}}(\omega(\mathrm{P})*\HP)(z)=F_{\mathsf{Y}}(\omega(\overline{\HP}))=z,
    \\ \label{eq::func_equation_cyc}
    F_{\mathsf{Y}}(\omega(\mathrm{C}))(F_{\mathsf{Y}}(\HP)(z))=F_{\mathsf{Y}}(\HC)+F_{\mathsf{Y}}(\omega(\overline{\HC}))=F_{\mathsf{Y}}(\HC).
\end{gather} Thanks to equation~\eqref{eq::involution}, we have
\begin{gather*}
F_{\mathsf{Y}}(\omega(\mathrm{P}))(\underline{x},z)=-F_{\mathsf{Y}}(\mathrm{P})(-\underline{x},-z)=1+\sum_{n\geq 1} (-1)^np_n -\frac{1}{1+z+p_1}, 
\\
F_{\mathsf{Y}}(\omega(\mathrm{C}))=\log(1+(p_1+z))+\sum_{n\geq 1}(-1)^n\frac{p_n}{n}.
\end{gather*} So, the formula~\eqref{eq::ham_young_generating} is exactly the solution of the functional equation~\eqref{eq::func_equation}. The formula~\eqref{eq::hamcyc_young_generating} follows from equation~\eqref{eq::func_equation_cyc},
\[
F_{\mathsf{Y}}(\HC)=\log(1+p_1+F_{\mathsf{Y}}(\HP)))+\sum_{n\geq 1} (-1)^n\frac{p_n}{n}.
\]
\end{proof}
The series expansion of the right hand side of~\eqref{eq::ham_young_generating} and~\eqref{eq::hamcyc_young_generating} in $p$-basis are given by
\begin{gather*}
 \frac{1}{1-z+\sum_{n\geq 1} (-1)^np_n}=\sum_{k\geq 0}k!z^k\sum_{n\geq 0} (-1)^n\sum_{\lambda\vdash n} (-1)^{l(\lambda)}\binom{k+l(\lambda)}{l(\lambda)}\binom{l(\lambda)}{m(\lambda)}p_{\lambda},
 \\
 -\log(1-(z+\sum_{n\geq 1}(-1)^{n-1}p_n))=\sum_{k\geq 0}k!z^k\sum_{n\geq 0} (-1)^n\sum_{\lambda\vdash n} (-1)^{l(\lambda)}\frac{1}{k+l(\lambda)}\binom{k+l(\lambda)}{l(\lambda)}\binom{l(\lambda)}{m(\lambda)}p_{\lambda}
\end{gather*} where $p_{\lambda}=p_{\lambda_1}\cdot p_{\lambda_2}\cdots p_{\lambda_k}$,  $\binom{l(\mu)}{m(\mu)}=\frac{l(\mu)!}{m_1(\mu)!\cdots m_n(\mu)!}$ and $m_i(\mu)=|\{j|\lambda_j=i\}|$- cycle types. Recall that the transition matrix from $p$-basis to $m$-basis is lower triangular
\begin{equation}
p_{\mu}=\sum_{\lambda \geq \mu} L_{\mu\lambda}m_{\lambda},   
\end{equation} where the sum ranges over all partitions $\lambda$ of the size $|\mu|$ that \textit{dominate} $\mu$: $\lambda_1\geq \mu_1, \lambda_1+\lambda_2\geq \mu_1+\mu_2$, and so on. Specifically, the coefficient $L_{\mu \lambda}$ is the number of $l(\lambda)\times l(\mu)$-matrices with entries in $\mathbb{N}$ whose column sum is $\mu$ and row sums is $\lambda$ and there is exactly one non-zero entry for each column.
\begin{sled}\label{thm::number_hp_lambda}
For a partition $\lambda$ and $k\geq 0$, we have 
\begin{equation}
\HP(\K_{(1^k)\cup \lambda})= k!\lambda!\sum_{\mu \leq \lambda} \varepsilon_{\mu}\binom{l(\mu)}{m(\mu)}\binom{l(\mu)+k}{l(\mu)}L_{\mu \lambda}, \text{ where } \varepsilon_{\mu}=(-1)^{|\mu|-l(\mu)}   
\end{equation} If $k+l(\lambda)\geq 2$, we have
\begin{equation}
\HC(\K_{(1^k)\cup \lambda})=k!\lambda!\sum_{\mu \leq \lambda} \varepsilon_{\mu}\frac{1}{k+l(\mu)}\binom{l(\mu)+k}{l(\mu)}\binom{l(\mu)}{m(\mu)}L_{\mu \lambda}.
\end{equation}
In particular, for $k=0$ and $l(\lambda)\geq 2$, we have
\[
\HP(\K_{\lambda})=\lambda!\sum_{\mu \leq \lambda} \varepsilon_{\mu}\binom{l(\mu)}{m(\mu)}L_{\mu \lambda}, \quad \HC(\K_{\lambda})=\lambda!\sum_{\mu \leq \lambda} \varepsilon_{\mu}\frac{1}{l(\mu)}\binom{l(\mu)}{m(\mu)}L_{\mu \lambda}.
\]
\end{sled}
\subsection{Paths and cycles}\label{sec::separable_permutations}  In this Section we restrict ourselves to the case of Path and Cycle graphs. Surprisingly, we reprove and generalize classical results concerning patterns avoiding permutations.
\subsubsection{Hertzsprung's problem} We recall the Hertzsprung's problem~\cite{hertzsprung1887kombinationsopgave}: \textit{How many ways $n$ non-attacking kings can be placed on an $n\times n$ board, $1$ in each row and column?} Equivalently, find the number of permutations without rising or falling successions. Note that there is a one to one correspondence between such permutation and Hamiltonian paths in the complement to the path graph $\overline{\Path}_{n}$. We show how to resolve the Hertzsprung's problem using results of this section. 

For a graphic function $f$, we define its Path generating function by the rule
\begin{gather}
 F_{\Path}(f)(t)=\sum_{n\geq 1} f(\Path_n)t^n, 
\end{gather} In particular, we have
\[
\sum_{n\geq 1} H_n t^n=F_{\Path}(\overline{\HP}),
\] where $H_n=\HP(\overline{\Path}_n)$ numbers from Hertzsprung's problem. It was shown in~\cite[Pr.~2.2.2]{khoroshkin2024hilbert}, that Path generating functions are compatible with $*$-product as follows
\begin{gather}
 F_{\Path}(f*g)(t)=F_{\Path}(f)(F_{\Path}(g)(t)), 
\end{gather} 

Thanks to equation above and Theorem~\ref{thm::reccurences_paths_and_cycles}, we have
\begin{equation}\label{eq::functional_equation_complement_paths}
F_{\Path}(\overline{\HP})=F_{\Path}(\mathrm{P})(F_{\Path}(\omega(\HP))(t)).   
\end{equation} Since each path on at least two vertices has exactly two directed Hamiltonian paths, we have
\begin{equation}\label{eq::formula_HP_paths}
F_{\Path}(\omega(\HP))(t)=-F_{\Path}(\HP)(-t)=t-2t^2+2t^3-2t^4+\cdots=\frac{t-t^2}{1+t}.    
\end{equation} So, combining above equations, we recover the generating series~\cite[page~737]{flajolet2009analytic}.

\begin{prop} The generating series resolving Hertzsprung's problem is given by
\[
\sum_{n\geq 1} H_nt^n= \sum_{n\geq 1} n!(\frac{t-t^2}{1+t})^n=t+2t^4+14t^5+90t^6+646t^7+\cdots
\]
\end{prop}

In a similar way, we could consider the Cyclic version of Hertzsprung's problem: \textit{ How many Hamiltonian cycles are in the complete graph $\K_n$ that do not share edge with the cycle $[1,2,\cdots,n]$?} Equivalently, we need to find the number of Hamiltonian cycle in the complement graph to the cycle $\overline{\Cyc}_n$. Such numbers and their asymptotic are well-studied, for example~\cite{aspvall1980dinner}. 

Similarly to the Path case, we consider Cycle generating function by the rule
\begin{gather}\label{eq::cyclic_generating_function}
 F_{\Cyc}(f)(t)=\sum_{n\geq 1} f(\Cyc_n)\frac{t^n}{n},
\end{gather} which is compatible with $*$-product as follows
\begin{gather}
 F_{\Cyc}(f*g)(t)=F_{\Cyc}(f)(F_{\Path}(g)(t))-f(\Path_1)F_{\Path}(g)(t)+f(\Path_1)F_{\Cyc}(g)(t),  
\end{gather} see~\cite[Pr.~2.2.2]{khoroshkin2024hilbert} for details. In particular, by Theorem~\ref{thm::reccurences_paths_and_cycles}, we have
\begin{equation}\label{eq::functional_equation_complement_cycles}
F_{\Cyc}(\overline{\HC})(t)=F_{\Cyc}(\mathrm{C})(F_{\Path}(\omega(\HP))(t))-F_{\Path}(\omega(\HP))(t)+F_{\Cyc}(\omega(\HP))(t)-F_{\Cyc}(\omega(\HC))(t), 
\end{equation} from which we deduce

\begin{prop}
The generating series resolving Cyclic Hertzsprung's problem is given by
\begin{multline}
\sum_{n\geq 1} CH_n\frac{t^n}{n}=t+F_{\Cyc}(\overline{\HC})=\frac{3t^2}{2}+\sum_{n\geq 1} \frac{(n-1)!}{n}(\frac{t-t^2}{1+t})^n+\sum_{n\geq 3}(-1)^n\frac{2t^n}{n}=\\=t+2\frac{t^5}{5}+6\frac{t^6}{6}+46\frac{t^7}{7}+354\frac{t^8}{8}+\cdots 
\end{multline} The numbers $CH_n$ form the sequence $\mathrm{A078603}$ from the OEIS~\cite{oeis}.
\end{prop}
\subsubsection{Separable permutations} Recall that a permutation $\sigma=(s_1,s_2,\cdots,s_n)$ contains a permutation $\tau=(t_1,t_2,\cdots,t_k)$ as a pattern if there is subsequence $\sigma'=(s_{i_1},s_{i_2},\cdots,s_{i_k})$ with $s_{i_l}<s_{i_r}$ if and only if $t_l<t_r$. A permutation $\sigma$ is called \textit{separable} if it avoids patterns $(2413),(3142)$. According to~\cite{bose1998pattern}, a permutation is separable if it has a separating tree ( see Example~\ref{ex::separable_paths}). In~\cite{shapiro1991bootstrap}, Shapiro and Stephens showed that the numbers of separable permutations are equal to little Schr\"oder numbers.  Now we reprove and generalise their results using operadic methods.

Recall that a \textit{non-symmetric operad} is a graded vector space $\Orb=\bigoplus_{n\geq 1} \Orb(n)$ with a collection of infinitesimal compositions 
\[
\circ_i\colon \Orb(n)\otimes \Orb(m)\to \Orb(n+m-1),\quad i=1,2,\cdots,n
\] satisfying certain associative axioms, for details see~\cite[Sec.~5.8]{loday2012algebraic}. For a path $\Path_{n+m-1}$, its induced and contracted graphs are also paths $\Path_{n+m-1}|_{\{i,i+1,\cdots,i+m-1\}}\cong \Path_m$, $\Path_{n+m-1}/_{\{i,i+1,\cdots,i+m-1\}}\cong \Path_n$. It was shown, that for a contractad $\Orb$ the restriction to paths determines a non-symmetric operad $\Path_{*}(\Orb)$ with components $\Path_{*}(\Orb)(n):=\Orb(\Path_n)$ and infinitesimal compositions induced from contractad structure $\circ_i:=\circ^{\Path_{n+m-1}}_{\{i,i+1,\cdots,i+m-1\}}$. Thanks to Example~\ref{ex::separable_paths}, the restriction of the contractad $\Plan$ to paths determines a non-symmetric operad $\Path_{*}(\Plan)$ of permutations realised by separating trees. Thanks to Theorem~\ref{thm::planpres}, this ns operad is generated by two binary generators $\mu^+$, $\mu^-$, satisfying the relations
\[
\mu^+\circ_1 \mu^{+}= \mu^+\circ_2 \mu^+, \quad \mu^-\circ_1 \mu^-= \mu^-\circ_2 \mu^-.
\] In particular, this ns operad is isomorphic to the free product $\mathsf{As}\star\mathsf{As}$ of two copies of associative ns operads.

Consider the set $\B_{\Path}=\{(2413),(3142)\}\subset \Sigma_4$ and let $\mathsf{Av}_n(\B_{\Path})$ be a set of separable $n$-permutations. 
\begin{theorem}\label{thm::separable_paths}
A permutation $\sigma$ is separable if and only if it is realised by a separating tree
\[
\mathsf{Av}_n(\B_{\Path})=\Plan(\Path_n).
\] Moreover, the generating function of separable permutations is given by the formula
\begin{equation}\label{eq::separable_shroder}
\sum_{n\geq 1} |\mathsf{Av}_n(\B_{\Path})|t^n=\frac{1-t-\sqrt{t^2-6t+1}}{2}=t+2t^2+6t^3+22t^4+90t^5+394t^6+\cdots.
\end{equation}
\end{theorem}
\begin{proof}
Consider the ns operad of permutations $\Sigma$ with components $\Sigma(n):=\Sigma_n$ and operadic structure given by substitutions. Note that the restriction ns operad $\Path_{*}(\Plan)$ forms a ns suboperad of $\Sigma$. By direct verification, we see that the collection of separable permutations $\mathsf{Av}(\B_{\Path})=\{\mathsf{Av}_n(\B_{\Path})\}_{n\geq 1}$ also forms a suboperad. Since both $\Path_{*}(\Plan)$ and $\mathsf{Av}(\B_{\Path})$ are suboperads of $\Sigma$ with the same generators, we conclude the coincidence $\Path_*(\Plan)=\mathsf{Av}(\B_{\Path})$.

Let us deduce the generating series for separable permutations. By previous statement, we have
\[
\sum_{n\geq 1} |\mathsf{Av}_n(\B_{\Path})|t^n=F_{\Path}(\PE).
\] Thanks to equation above and equation~\eqref{eq::hp_inverse}, we have
\begin{equation}\label{eq::functional_equation_paths}
 F_{\Path}(\omega(\HP))(F_{\Path}(\PE)(t))=t.   
\end{equation} So, the generating series~\eqref{eq::separable_shroder} follows from the formula~\eqref{eq::formula_HP_paths} and functional equation~\eqref{eq::functional_equation_paths}.
\end{proof}

Thanks to Lemma~\ref{lemma:edges_in_tuples}, we obtain the following recursive linear time algorithm of determining whether a permutation is separable. The algorithm is: for a permutation $\sigma=(s_1,s_2,\cdots,s_n)$\begin{itemize}
    \item[(i)] Find first index $i$ such that $|s_i-s_{i+1}|=1$. If there is no such $i$, then $s$ is non-separable 
    \item[(ii)] Consider a new permutation $\sigma'=(s'_1,s'_2,\cdots,s'_{n-1})$ by replacing $s_i,s_{i+1}$ with $\min(s_i,s_{i+1})$ and renormalising remaining $s_j$.
    \item[(iii)] A permutation $\sigma$ is separable if and only if $\sigma'$ is separable. So, continue the same procedure for $\sigma'$.
\end{itemize}
The results of this subsection suggests us that,  for particular series of graphs, we can describe the corresponding components of $\Plan$ in terms of patterns avoiding permutations. 

\begin{example} Consider the series of cycles $\Cyc=\{\Cyc_n\}$. For $n\leq 4$, we have $\Plan(\Cyc_n)=\Sigma_n$. For $n=5$, the inclusion is proper and the complement $\B_{\Cyc}:=\Sigma_5\setminus \Plan(\Cyc_5)$ consists of $10$ elements of the form 
\[
\B_{\Cyc}=\{(i_1,i_2,\cdots,i_5)| i_k-i_{k+1} = \pm 2 \mod 5\}.
\] We state that, permutations realised by $\Cyc_n$-admissible binary planar trees are exactly $\B_{\Cyc}$-avoiding permutations, $\Plan(\Cyc_n)=\mathsf{Av}_n(\B_{\Cyc})$. The proof mimics the proof of Theorem~\ref{thm::separable_paths}, so we just give a sketch. The idea is that, for a contractad $\Orb$, the restriction to cycles $\Cyc_*(\Orb)(n):=\Orb(\Cyc_n)$ defines a right-module over ns operad $\Path_*(\Orb)$ with additional actions of cycle groups $\mathbb{Z}_n$ on components $\Cyc_*(\Orb)(n)$. Using the isomorphism $\Path_*(\Orb)\cong \mathsf{Av}_n(\B_{\Path})$, we see that both $\Cyc_*(\Plan)$ and $\mathsf{Av}(\B_{\Cyc})$ form cycle modules over ns operad $\mathsf{Av}(\B_{\Path})$. Similarly to the proof of Theorem~\ref{thm::separable_paths}, we can see that these modules coincide.

Next, similarly to the second part of the proof of Theorem~\ref{thm::separable_paths} and using logarithmic generating functions for cycles~\eqref{eq::cyclic_generating_function}, we deduce that
\[
F_{\Cyc}(\PE)=\sum_{n\geq 1} |\mathsf{Av}_n(\B_{\Cyc})|\frac{t^n}{n}=\frac{3t-t^2-t\sqrt{t^2-6t+1}}{2}.
\] By comparing with series~\eqref{eq::separable_shroder}, we see that, for $n\geq 2$, we have $|\mathsf{Av}_n(\B_{\Cyc})|=n|\mathsf{Av}_{n-1}(\B_{\Path})|$.
\end{example}
\begin{question*}
For which type of patterns $\B_{\mathsf{X}}$ we could find a series of graphs $\mathsf{X}=\{\mathsf{X}_n\}$, such that a permutation is realised by a $\mathsf{X}$-admissible binary planar tree if and only if it avoids $\B_{\mathsf{X}}$-patterns,  $\mathsf{Av}_n(\B_{\mathsf{X}})=\Plan(\mathsf{X}_n)$?
\end{question*}
\bibliographystyle{alpha}
\bibliography{biblio.bib}
\end{document}